\documentclass[12pt]{article}
\usepackage{framed}
\usepackage[utf8]{inputenc}
\usepackage{times}
\usepackage[margin=1in]{geometry}
\usepackage{amsfonts} 
\usepackage{graphicx}
\usepackage{authblk}
\usepackage{enumerate}
\usepackage[inline]{enumitem}
\usepackage{hyperref}
\usepackage[all]{xy}
\usepackage{appendix}
\usepackage{amssymb}
\usepackage{fourier}
\usepackage{stmaryrd}
\usepackage{url}
\usepackage[normalem]{ulem} 
\usepackage{stmaryrd}
\usepackage{tikz-cd}
\usepackage{tablefootnote}
\usetikzlibrary{calc,trees,chains,matrix,arrows,decorations.pathreplacing,shapes.misc,shapes.multipart}
\tikzset{commutative diagrams/diagrams={ampersand replacement=\&}}

\tikzset{dbl/.style={double,
		double equal sign distance,
		-implies,
		shorten >=10pt,
		shorten <=10pt}}

\tikzset{
	invisible/.style={opacity=0},
	visible on/.style={alt={#1{}{invisible}}},
	alt/.code args={<#1>#2#3}{%
		\alt<#1>{\pgfkeysalso{#2}}{\pgfkeysalso{#3}}%
	}
}

\let\OLDthebibliography\thebibliography
\renewcommand\thebibliography[1]{
  \OLDthebibliography{#1}
  \setlength{\parskip}{0pt}
  \setlength{\itemsep}{3pt plus 0.3ex}
}

\usepackage{relsize}
\usepackage{amsmath,amscd}
\usepackage{mathtools}
\usepackage{amsthm}

\usepackage{footmisc}
\usepackage{thmtools} 
\usepackage{thm-restate}
\usepackage{hyperref}

\usepackage{cleveref}

\declaretheorem[name=Theorem,numberwithin=section]{theorem}

\definecolor{darkblue}{rgb}{0.0, 0.0, 0.8}
\definecolor{darkred}{rgb}{0.8, 0.0, 0.0}
\definecolor{darkgreen}{rgb}{0.0, 0.8, 0.0}
\definecolor{grey}{rgb}{0.5, 0.5, 0.5}
\definecolor{aliceblue}{rgb}{0.94, 0.97, 1.0}

\makeatletter
\providecommand*{\twoheadrightarrowfill@}{%
  \arrowfill@\relbar\relbar\twoheadrightarrow
}
\providecommand*{\twoheadleftarrowfill@}{%
  \arrowfill@\twoheadleftarrow\relbar\relbar
}
\providecommand*{\xtwoheadrightarrow}[2][]{%
  \ext@arrow 0579\twoheadrightarrowfill@{#1}{#2}%
}
\providecommand*{\xtwoheadleftarrow}[2][]{%
  \ext@arrow 5097\twoheadleftarrowfill@{#1}{#2}%
}
\makeatother

\theoremstyle{definition}
\newtheorem{lemma}[theorem]{Lemma}
\newtheorem{corollary}[theorem]{Corollary}
\newtheorem{proposition}[theorem]{Proposition}

\newtheorem{example}[theorem]{Example}
\newtheorem{remark}[theorem]{Remark}
\newtheorem{definition}[theorem]{Definition}
\newtheorem{convention}[theorem]{Convention}
\newtheorem{claim}[theorem]{Claim}

\newcommand{\db}{d_{\mathrm{B}}}
\newcommand{\dfhat}{\widehat{\df}}

\newcommand{\length}{\mathrm{length}}

\newcommand{\crit}{\mathbf{crit}}
\newcommand{\VR}{\mathrm{VR}}
\newcommand{\R}{\mathbb{R}}
\newcommand{\Rplus}{\mathbb{R}_{\geq0}}
\newcommand{\Rop}{\R^{\mathrm{op}}}
\newcommand{\Z}{\mathbb{Z}}
\newcommand{\Zplus}{\Z_{\geq0}}
\newcommand{\Zop}{\Z_{\geq0}^{\mathrm{op}}}
\newcommand{\N}{\mathbb{N}}
\newcommand{\X}{\mathbb{X}}

\newcommand{\norm}[1]{\left\lVert#1\right\rVert}

\newcommand{\e}{\varepsilon}

\newcommand{\PP}{\mathcal{P}}
\newcommand{\PPP}{(\PP,\leq)}
\newcommand{\PPPP}{(\PP,\leq,\Omega)}
\newcommand{\QQ}{\mathcal{Q}}

\newcommand{\SSS}{\mathcal{Q}}

\newcommand{\Ifunc}{\mathbf{I}}

\newcommand{\Ffunc}{\mathbf{F}}

\newcommand{\Gfunc}{\mathbf{G}}
\newcommand{\Sfunc}{\mathbf{S}}
\newcommand{\Hfunc}{\mathbf{H}}

\newcommand{\Part}{\mathbf{Part}}
\newcommand{\Simp}{\mathbf{Simp}}

\newcommand{\simp}{\Delta}

\newcommand{\eps}{\varepsilon}

\newcommand{\subpart}{\mathbf{SubPart}}
\newcommand{\Formi}{\mathbf{Formi}}

\newcommand{\Int}{\mathbf{Int}}

\newcommand{\vect}{\mathbf{vect}}

\newcommand{\tripod}{R:\ X \xtwoheadleftarrow{\varphi_X} Z \xtwoheadrightarrow{\varphi_Y} Y}
\newcommand{\tripodd}{R_1:\ X \xtwoheadleftarrow{\varphi_X} Z_1 \xtwoheadrightarrow{\varphi_Y} Y}
\newcommand{\tripoddd}{R_2:\ Y \xtwoheadleftarrow{\psi_Y} Z_2 \xtwoheadrightarrow{\psi_W} W}
\newcommand{\dintf}{d_{\mathrm{GH}}}
\newcommand{\df}{d_\mathrm{F}}
\newcommand{\dgh}{d_{\mathrm{GH}}}
\newcommand{\dhaus}{d_{\mathrm{H}}}
\newcommand{\dint}{d_{\mathrm{I}}}

\newcommand{\upcode}{cosheaf-code{}}

\newcommand{\abs}[1]{\left\lvert{#1}\right\rvert}
\newcommand{\pow}{\mathbf{pow}_{\geq 1}}
\newcommand{\dis}{\mathrm{dis}}
\newcommand{\ijr}{\mathrm{ijr}}
\newcommand{\poset}{(\PP,\leq)}
\newcommand{\posetflow}{(\PP,\leq,\Omega)}
\newcommand{\semilattice}{(\SSS,\leq)}

\newcommand{\rk}{\mathrm{rk}}

\newcommand{\dero}{d_{\mathrm{E}}}
\newcommand{\Hrm}{\mathrm{H}}
\newcommand{\C}{\mathcal{C}}

\newcommand{\dt}{d_{\mathrm{GH}}}

\newcommand{\domhat}{d_{\widehat{\Omega}}}
\newcommand{\omegahat}{\widehat{\Omega}}
\newcommand{\pe}{poset map}
\newcommand{\F}{\mathbb{F}}
\newcommand{\Acal}{\mathcal{A}}
\newcommand{\Bcal}{\mathcal{B}}

\title{Interleaving by Parts: Join Decompositions of Interleavings and Join-Assemblage of Geodesics}

\author[1]{Woojin~Kim}
\author[2]{Facundo~Mémoli}
\author[3]{Anastasios~Stefanou}

\affil[1]{Department of Mathematics, 
		Duke University.
		\thanks{\texttt{woojin@math.duke.edu}}}

\affil[2]{Department of Mathematics and Department of Computer Science and Engineering, 
		The Ohio State University.\thanks{\texttt{memoli@math.osu.edu}}}
		
\affil[3]{Department of Mathematics and Computer Science, 
		University of Bremen.
		\thanks{\texttt{stefanou@uni-bremen.de}}}

\begin{document}

\maketitle

\begin{abstract} 
Metrics of interest in topological data analysis (TDA) are often explicitly or implicitly in the form of an interleaving distance $\dint$ between poset maps (i.e.~order-preserving maps), e.g. the Gromov-Hausdorff distance between metric spaces can be reformulated in this way.
   
   We propose a representation of a poset map $\Ffunc:\PP\to\QQ$ as a join (i.e.~supremum) $\bigvee_{b\in B} \Ffunc_b$ of simpler poset maps $\Ffunc_b$ (for a join dense subset $B\subset \QQ$) which in turn yields a  decomposition of $\dint$ into a product metric.      The decomposition of $\dint$ is simple, but its ramifications are manifold: (1) We can construct a geodesic path between any poset maps $\Ffunc$ and $\Gfunc$ with $\dint(\Ffunc,\Gfunc)<\infty$ by assembling geodesics between all $\Ffunc_b$s and $\Gfunc_b$s via the join operation. This construction generalizes at least three constructions of geodesic paths that have appeared in the literature. 
      (2) We can extend the Gromov-Hausdorff distance to a distance between simplicial filtrations over an arbitrary poset with a flow, preserving its universality and geodesicity. 
  (3) We can clarify equivalence between several known metrics on multiparameter hierarchical clusterings.  (4) We can illuminate the relationship between the \emph{erosion distance} by Patel and the \emph{graded rank function} by Betthauser, Bubenik, and Edwards, which in turn takes us to an interpretation on the representation $\bigvee_b \Ffunc_b$ as a generalization of persistence landscapes and graded rank functions.

\end{abstract}

\paragraph{Acknowledgements.}
WK thanks Parker Edwards, Alex McCleary and  Justin Curry for beneficial discussions.  We also thank Zane Smith for helping with  Thm.~\ref{thm:dH=dB}. WK and FM were supported by NSF through grants  DMS-1723003, CCF-1740761, RI-1901360, and CCF-1526513. AS was supported by NSF  through grants CCF-1740761,
DMS-1440386 
and  RI-1901360.

\section{Introduction}

\paragraph{Persistent homology and interleaving distances.}
Persistent homology plays a central role in topological data analysis (TDA) \cite{carlsson2009topology,edelsbrunner2008persistent,ghrist2008barcodes}.  The most basic construction in persistent homology consists of applying the homology functor to an $\R$-indexed nested family of topological spaces or simplicial complexes such as the Vietoris-Rips filtration on a metric space. 
By utilizing homology with coefficients in a field $\F$, we obtain so-called \emph{persistence modules}, which are $\R$-indexed functors valued in the category  $\vect$ of vector spaces and linear maps over $\F$. Generalizing this notion, a poset-indexed functor valued in a certain category $\C$ is called a \emph{generalized persistence module} with values in $\C$ \cite{bubenik2014categorification}. 

One of the most prevalent metrics for quantifying the dissimilarity between two persistence modules is the \emph{interleaving distance} $\dint$. Since $\dint$ was first introduced in order to compare $\R$-indexed persistence modules \cite{chazal2009proximity}, 
it has been generalized to various different settings 
\cite{botnan2018algebraic,botnan2020relative,bubenik2015metrics,curry2013sheaves,de2016categorified,deSilva2018,lesnick2015theory,scoccola2020locally}.
One of the main uses of 
$\dint$ is for comparing $\R^n$-indexed persistence modules, where its computation is known to be NP-hard for $n\geq 2$ \cite{bjerkevik2019computing}. 

While poly-time computable lower bounds for $\dint$ have been studied for $n=2$  \cite{bjerkevik2021asymptotic,cerri2013betti,kerber2018exact,landi2018rank}, its extension to the case $n\geq 3$ is not much known. The \emph{erosion distance} introduced by Patel \cite{patel2018generalized} is an attractive alternative in this respect and will be subsequently further discussed.

\paragraph{Interleavings between poset maps.} 
Partially ordered sets are simply called \emph{posets}. An order-preserving map $\PP\rightarrow \QQ$ between posets $\PP$ and $\QQ$ is called a \emph{poset map}. By viewing the target poset $\QQ$ as a category (each point  $p\in \QQ$ is an object and a unique arrow $p\rightarrow q$ exists whenever $p\leq q$ in $\QQ$), a poset map can be viewed as a generalized persistence module. Poset maps are omnipresent in TDA, e.g. simplicial filtrations (indexed by arbitrary posets), hierarchical clusterings  (indexed by arbitrary posets), and (generalized) rank functions of persistence modules. 
Poset maps have also been utilized  in discrete Morse theory, cf.  \cite[Thm.~11.4]{kozlov2008combinatorial}.

When $\PP$ is equipped with a notion of \emph{flow}  \cite{bubenik2015metrics,deSilva2018,scoccola2020locally}, we can define an \emph{interleaving distance} $\dint$ between two \pe{}s $\PP\rightarrow \QQ$. Examples of such include the following.

\emph{1. Interleavings between simplicial filtrations.} A  distance  between $\R$-indexed simplicial filtrations  was proposed by M\'emoli \cite{memoli2017distance}. It turned out that this distance is a generalization of the Gromov-Hausdorff distance between finite metric spaces and thus called the Gromov-Hausdorff distance and denoted by $\dgh$ \cite{memoli2019quantitative}. 
This distance was proved to upper bound the bottleneck distance between persistence diagrams \cite[Thm.~4.2]{memoli2017distance} and even the homotopy interleaving distance by Blumberg and Lesnick \cite{blumberg2017universality}; see \cite{scoccola2020locally}. A certain variant of $\dgh$ also appears in the study of metrics on \emph{Reeb graphs} \cite{bauer2020reeb}.

\begin{figure}
    \centering
    \includegraphics[width=\textwidth]{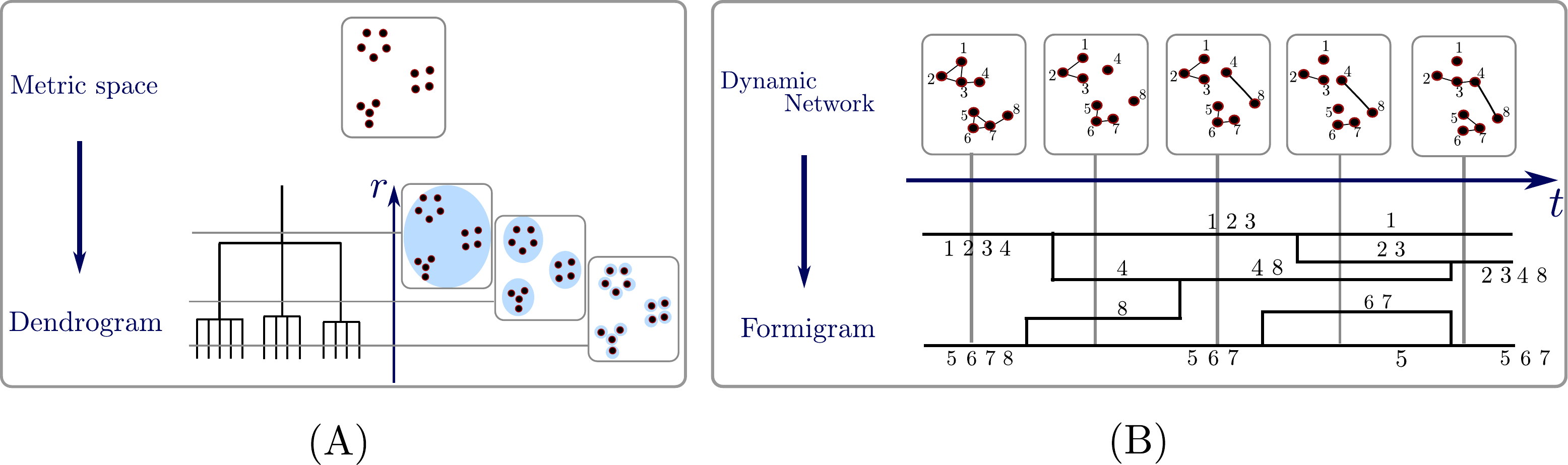}
    \caption{(A) A dendrogram derived from a hierarchical clustering method on a metric space (Ex.\ref{ex:SLHC}). The dendrogram captures  multiscale clustering features of the given metric space. (B) The formigram* derived from a dynamic network. The formigram tracks the evolution of connected components in the dynamic network. {\footnotesize *The name formigram is a combination of the words {formicarium} (a.k.a. ant farm) and {diagram}. Synthetic flocking behaviors have been successfully classified via a certain lower bound for a distance between formigrams (Defn.~\ref{def:general formigram distance}) \cite{kim2020analysis}.} }
    \label{fig:dendrogram and formigram}
\end{figure}

\emph{2. Interleavings between hierarchical clusterings over a poset.} 
Hierarchical clustering methods on a metric space $X$  
uncovers multiscale clustering features of $X$ \cite{carlsson2010characterization}, yielding a \emph{dendrogram} (Fig.~\ref{fig:dendrogram and formigram} (A)), a poset map from $[0,\infty)$ to the lattice of partitions of $X$.
Dendrograms are widely generalized for density-sensitive hierarchical clustering \cite{cai2020elder,carlsson2010multiparameter,rolle2020stable},  for hierarchical clustering on an asymmetric network \cite{smith2016hierarchical}, for
summarizing clustering features in dynamic networks  \cite{buchin2013trajectory,kim2018formigrams}, Fig.~\ref{fig:dendrogram and formigram} (B)). These structures also arise in phylogenetic trees \cite{billera2001geometry,sokal1962comparison,woese1990towards} and phylogenetic networks  \cite{griffiths1997,huson2010phylogenetic,martin2013genome,parida2015topological}. We remark that these structure are finer than merge trees \cite{gasparovic2019intrinsic,morozov2013interleaving} or Reeb graphs \cite{bauer2020reeb,Bauer2015b,chambers2020family,de2016categorified,Stefanou_2020}, as addressed in \cite{kim2017stable}.

\emph{3. Interleavings between poset maps into a Grothendieck group.} 
The \emph{erosion distance} $\dero$ was introduced for comparing \emph{generalized persistence diagrams} for $\R$-indexed persistence modules valued in certain categories $\C$ beyond $\vect$ \cite{patel2018generalized}. Generalized persistence diagrams 
can be encoded as certain poset maps, called \emph{rank functions}, whose target is the \emph{Grothendieck group} of $\C$ \cite{weibel2013k}, equipped with a natural order. $\dero$ is actually an interleaving distance between rank functions.
Even though $\dero$ has been 
adapted to more abstract settings \cite{kim2018generalized,puuska2017erosion}, its most basic use is as
a tractable lower bound for $\dint$ between $\R^n$-indexed persistence modules  \emph{for any $n$} \cite[Section 5]{kim2020spatiotemporal}. Indeed, $\dero$  has been utilized for classifying spatiotemporal persistent homology (encoded as $\R^3$-indexed persistence modules) of time-varying data \cite{clause2020spatiotemporal}.

\paragraph{Other related work.}  The categorification of persistent homology has provided a fertile interpretation of persistence theory and the interleaving distance   \cite{botnan2020relative,bubenik2015metrics,bubenik2017interleaving,bubenik2014categorification,deSilva2018,scoccola2020locally}. Among others, Scoccola \cite{scoccola2020locally} introduced a notion of a \emph{locally persistent category}, which is a category with a notion of approximate morphism. This enables us to define an interleaving
distance between objects in the category, which encompasses many distances in TDA and facilitates a uniform treatment of those distances. For example, sufficient conditions under which an interleaving
distance is geodesic have been found \cite[Thms.4.5.2 and 4.5.16]{scoccola2020locally}.

\paragraph{Our contributions.}

Let $\posetflow$ be a poset \emph{with a flow} (Defn.~\ref{def:flow}) and let $\SSS$ be a poset with a \emph{join-dense} subset $B\subset \QQ$ (which always exists; see Sec.~\ref{sec:lattices}). Let $\Ffunc,\Gfunc:\posetflow\rightarrow \semilattice$ be any two \pe{}s and let $\dint(\Ffunc,\Gfunc)$ be their interleaving distance (Defn.~\ref{def:Interleavings of persistent elements}).

\begin{enumerate}[label=(\roman*),topsep=0pt,itemsep=-1ex,partopsep=1ex,parsep=1ex]
    \item \label{item:contribution 1} We identify join representations $\Ffunc=\bigvee_{b\in B} \Ffunc_b$ and $\Gfunc=\bigvee_{b\in B} \Gfunc_b$ such that  \[\dint(\Ffunc,\Gfunc)=\sup_{b\in B} \dint(\Ffunc_b,\Gfunc_b),\] 
    where each of the $\Ffunc_b$s and $\Gfunc_b$s is structurally simple.
    These join representations can be seen as a rendition of the algebraic decomposition of persistence modules (Rmks.~\ref{rem:interleaving by parts analogy}, \ref{rmk:comparison with isometry theorem} \ref{item:comparison with isometry theorem}) as well as a generalization of \emph{persistence landscapes} \cite{betthauser2019graded,bubenik2015statistical} (Rmk.~\ref{rem:erosion distance} \ref{item:erosion distance1}).
\end{enumerate}
We harness item \ref{item:contribution 1} in order to establish \emph{all} of the following: 
\begin{enumerate}[label=(\roman*),resume,topsep=0pt,itemsep=-1ex,partopsep=1ex,parsep=1ex]
    \item \label{item:contribution 1.5} We show that $\dint$ is an $\ell^\infty$-product\footnote{Given any metric spaces $(M_i,d_i)$, $i\in I$, the $\ell^\infty$-product metric is defined to be the metric $\sup_{i\in I} d_i$ on $\Pi_{i\in I}M_i$.} of multiple copies of a distance between upper sets in $\PP$ (Thm.\ref{thm:general-metric-dec}). 
    \item \label{item:contribution 2} We show that $\dint$, as a distance between  poset maps $\Ffunc,\Gfunc:\PP\to\QQ$, is geodesic under the assumption that $\QQ$ is a complete lattice (Thm.~\ref{thm:di is geodesic}). More specifically, we obtain a geodesic path between $\Ffunc$ and $\Gfunc$ by \emph{assembling} geodesic paths between all $\Ffunc_b$s and $\Gfunc_b$s  via the join operation.

    All the metrics mentioned in subsequent  items can be incorporated into the framework of interleaving distances between poset maps. This enables us to prove in a uniform way that all metrics in the items below  are  geodesic.\footnote{Some of those distances are already  known to be geodesic, but some are not. Known results will be cited at suitable places in the paper.} 
    \end{enumerate}
   
    \begin{enumerate}[label=(\roman*),resume,topsep=0pt,itemsep=-1ex,partopsep=1ex,parsep=1ex]
    \item We show that computing the erosion distance between rank functions of persistence modules 
    amounts to  computing a finite number of Hausdorff distances between certain geometric signatures of \emph{graded rank functions} \cite{betthauser2019graded} (Thm.~\ref{thm:rank invariant}). 
    An analogous statement holds when comparing multiparameter hierarchical clusterings (Thm.~\ref{thm:hierarchical clustering}).  \label{item:contribution 3}
    \item We generalize 
    the Gromov-Hausdorff distance between metric spaces to a distance between simplicial filtrations over $\PP$. This distance inherits a universal property of the original Gromov-Hausdorff distance (Thm.~\ref{thm:universality}),  of which the celebrated Vietoris-Rips filtration stability theorem \cite{chazal2009gromov,chazal2014persistence} becomes a consequence (Thm.~\ref{thm:H_k-lower-bound} and Rmk.~\ref{rem:H_k lower bound} \ref{item:H_k lower bound_VR stability generalization}). 
    \item We establish the equivalence between several known metrics on multiparameter hierarchical clusterings (Defn.~\ref{def:general formigram distance}, Thm.~\ref{thm:structmax}, Rmk.~\ref{rem:relating to GH}).

    \item We elucidate the computational complexity of the interleaving distance between formigrams (Thm.~\ref{thm:df complexity}).
    \label{item:contribution last}

\end{enumerate}

\paragraph{Organization.} Sec.~\ref{sec:preliminaries} reviews the notions of lattices, subpartitions, interleaving distances, and formigrams.  
Sec.~\ref{sec:generalizations} addresses items \ref{item:contribution 1}--\ref{item:contribution 2} above and 
Sec.~\ref{sec:applications} addresses the rest of the items. Sec.~\ref{sec:discussions} discusses open questions.

\section{Preliminaries}\label{sec:preliminaries}

 We review the notions of lattices (Sec.~\ref{sec:lattices}), subpartitions  (Sec.~\ref{sec:subpartitions}), interleaving distances (Sec.~\ref{sec:posets with a flow}), and formigrams (Sec.~\ref{sec:formigrams}).

\subsection{Posets, lattices, and \pe{}s.}\label{sec:lattices}

In this section, we recall basic terminology from the theory of ordered sets and lattices  \cite{erne2003posets,roman2008lattices}. 

A \textbf{poset} $\PP=(\PP,\leq)$ is a nonempty set $\PP$ equipped with a partial order, i.e.~a reflexive, anti-symmetric, and transitive relation $\leq$ on $\PP$. An element $0\in\PP$ is said to be a \textbf{zero element} if $0\leq p$ for all $p\in \PP$. If a zero element exists in $\PP$, then it is unique.
 Thus we refer to $0$ as \textit{the} zero element in $\PP$. An element $1\in \PP$ is said to be a \textbf{unit element} if $p\leq 1$ for all $p\in \PP$. For $p,q\in \PP$ with $p\leq q$, we write $[p,q]$
 for the set $\{r\in \PP: p\leq r\leq q\}$. Also, we write $p^\uparrow$ for the set $\{r\in \PP: p\leq r\}$. 
 An \textbf{upper set} in $\PP$ is a subset $A\subset \PP$ such that if $p\in A$ and $p\leq q$ in $\PP$, then  $q\in A$. 

 A \textbf{join} (a.k.a. least upper bound) of $p_1,\ldots,p_n\in \PP$ is an element $q\in \PP$ such that (i) $p_i\leq q$, for all $i=1,\ldots,n$, and (ii) for any $s\in \PP$, if $p_i\leq s$ for all $i=1,\ldots, n$, then $q\leq s$.
 A \textbf{meet} (a.k.a. greatest lower bound)  of $p_1,\ldots,p_n\in \PP$ is an element $r\in \PP$ such that (i) $r\leq p_i$, for all $i=1,\ldots,n$, and (ii) for any $s\in P$, if $s\leq p_i$ for all  $i=1,\ldots, n$, then $s\leq r$.
 If a join and a meet of $p_1,\ldots,p_n$ exist, then they are unique. Hence,  whenever they exist, we refer to them as \emph{the} join (denoted by $\bigvee \{p_i\}_{i=1}^n$) and \emph{the} meet (denoted by $\bigwedge\{p_i\}_{i=1}^n$), respectively. The poset $\PP$ is said to be a \textbf{join-semilattice} (resp. \textbf{meet-semilattice}) if $\PP$ allows all finite joins (resp. meets). If $\PP$ is both join- and meet-semilattice, then $\PP$ is said to be a \textbf{lattice}. $\PP$ is called a \textbf{complete lattice} if the meet and join of \emph{any} subset (possibly infinite) $A\subset \PP$ exist. 

If a poset $\PP$ has a zero element $0$, then any nonzero $p \in \PP$ such that $[0,p]=\{0,p\}$ is called an \textbf{atom} of $\PP$.
A nonzero element $p$ of a lattice $\PP$ is \textbf{(join-)irreducible} if $p$ is not the join of two
smaller elements, that is, if $p=q\vee r$, then  $p=q$ or $p=r$.  Note that every atom is join-irreducible. For example, for the ordered set $\QQ=\{p<q<r\}$, $p$ is the zero element, $q$ is the unique atom, and 
$q$ and $r$ are join-irreducible.  A \textbf{join representation} of $p\in \PP$ is an expression $\bigvee A$ which evaluates to $p$ for some $A\subset \PP$. When $\PP$ includes a zero element, the zero element has the join representation $\bigvee \emptyset$. A join representation $\bigvee A$ of $p$ is \textbf{irredundant} of $\bigvee A' <\bigvee A$ for each proper subset $A'\subset A$. We say that a subset $A\subset\PP$ \textbf{join-refines} another subset $B\subset \PP$ if, for each $a\in A$, there exists some element $b\in B$ such that $a\leq b$. Join-refinement defines a preorder $\preceq$ on the subsets of $\PP$. Fix $p\in \PP$ and let $\ijr(p)$ be the set of irredundant join representations of $p$. If $(\ijr(p),\preceq)$ has a unique minimum element $A\subset \PP$, then $A$ (or $\bigvee A$) is called the \textbf{canonical join representation of $p$.} 

A subset $B\subset \PP$ is said to be \textbf{join-dense} if every element of $\PP$ is the join of a subset of $B$. Trivially, $\PP$ itself is join-dense. 
The poset $\PP$ is said to be \textbf{$\bigvee$-irreducibly generated} if the set of all join-irreducible elements of $\PP$ is join-dense. 
Every finite poset is $\bigvee$-irreducibly generated \cite{erne1987compact}.

\begin{remark}\label{rem:smallest join-dense}
Any join-dense subset must contain all irreducible elements. Therefore, when $\PP$ is finite, the set of irreducible elements in $\PP$ is the smallest join-dense subset of $\PP$.
\end{remark}

 Given any two posets $\PP$ and $\QQ$, a map $\Ffunc:\PP\rightarrow\QQ$ is called an \textbf{order-preserving map} or a \textbf{poset map} if $p\leq q$ in $\PP$ implies $\Ffunc(p)\leq \Ffunc(q)$. The collection of all poset maps from $\PP$ to $\QQ$ will be denoted by $[\PP,\QQ]$. 
 
 \begin{remark} \label{rmk:functor poset}We regard $[\PP,\QQ]$ as a poset equipped with the partial order inherited from $\QQ$, i.e. 
 $\Ffunc \leq \Gfunc\ \mbox{in $[\PP,\QQ]$} \Leftrightarrow \Ffunc(p)\leq \Gfunc(p), \mbox{for all $p\in \PP$.}$
 \end{remark}

\subsection{Lattice of subpartitions}\label{sec:subpartitions}
In this section we review the notion of subpartition \cite{kim2018formigrams,rolle2020stable,smith2016hierarchical} and show that the collection of all subpartitions of a set is a lattice.

Let us fix a nonempty finite set $X$. A \textbf{partition of $X$} is a collection $P$ of nonempty disjoint $B\subset X$, called \textbf{blocks}, such that the union of all blocks $B$ is equal to $X$. Every partition $P$ of $X$ induces the equivalence relation $\sim_P$ given by: $x\sim_P x' \Leftrightarrow x,x'\in B$ for some block $B\in P$. Reciprocally, any equivalence relation $\sim$ of $X$ induces the partition $X/\!\!\sim$. A  \textbf{subpartition $Q$ of $X$} is 
a partition $Q$ of some $X'\subset X$. The set $X'$ is said to be the \textbf{underlying set} of $Q$. The equivalence relation on $X'$ induced by $Q$ is said to be a \textbf{subequivalence relation} on $X$. 

\begin{definition}[Poset of (sub)partitions] By $\Part(X)$ (resp.
$\subpart(X)$), we denote the set of all partitions (resp. subpartitions) of $X$. Let $P_1,P_2\in \subpart(X)$.
$P_1$ is said to \textbf{refine $P_2$} and write $P_1\leq P_2$, if for any block $B_1\in P_1$, there exists a block $B_2\in P_2$ such that $B_1\subset B_2$ (see Fig. \ref{fig:Hasse} for an example).\footnote{See \cite[Sec.~4]{roman2008lattices} for properties of $\Part(X)$.}  
\end{definition}
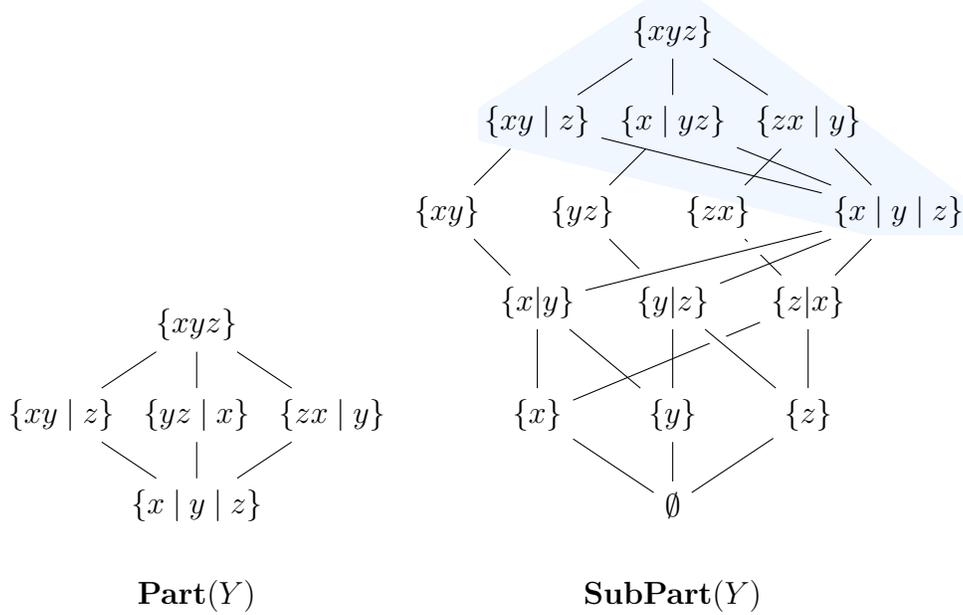
\begin{figure}
    \centering
\begin{tikzpicture}[scale=.6]
  \node (one) at (0,2) {$\{xyz\}$};
  \node (a) at (-3,0) {$\{x y\ |\ z\}$};
  \node (b) at (0,0) {$\{y z\ |\ x\}$};
  \node (c) at (3,0) {$\{z x\ |\ y\}$};
  \node (zero) at (0,-2) {$\{x\ |\ y\ |\ z\}$};
  \draw (zero) -- (a) -- (one) -- (b) -- (zero) -- (c) -- (one);
  \node (l) at (0,-4) {$\Part(Y)$};
\end{tikzpicture}
\begin{tikzpicture}[scale=.6]
 \filldraw[fill=aliceblue,draw=aliceblue,smooth cycle] (0.6,5.2)--(-0.6,5.2)--(-4.3,2.8)--(-4.3,2.1)--(4.3,0)--(6.5,0)--(6.5,0.8)--(0.6,5.2);
  \node (one) at (0,4.5) {$\{xyz\}$};
  \node (a) at (-3,2.5) {$\{x y\ |\ z\}$};
  \node (b) at (0,2.5) {$\{x \ |\ yz\}$};
  \node (c) at (3,2.5) {$\{z x\ |\ y\}$};
  \node (d) at (-5,0.5) {$\{xy\}$};
  \node (e) at (-2,0.5) {$\{yz\}$};
  \node (f) at (1,0.5) {$\{z x\}$};
  \node (g) at (5,0.5) {$\{x\ |\ y\ |\ z\}$};
  \node (dd) at (-3,-1.5) {$\{x|y\}$};
  \node (ee) at (0,-1.5) {$\{y|z\}$};
  \node (ff) at (3,-1.5) {$\{z|x\}$};
  \node (h) at (-3,-4) {$\{x\}$};
  \node (i) at (0,-4) {$\{y\}$}; 
  \node (j) at (3,-4) {$\{z\}$};  
  \node (zero) at (0,-6) {$\emptyset$};
  \draw (zero) -- (h) -- (dd); 
  \draw (d) -- (a);
  \draw (f) -- (c);
  \draw (a) -- (one) --(b); 
  \draw (b) -- (e);
  \draw (i) -- (zero) -- (j) -- (ff); 
  \draw (h) -- (ff);
  \draw (i) -- (dd);
  \draw (c) -- (one); 
  \draw[preaction={draw=white, -,line width=6pt}] (j) -- (ee);
  \draw (g) -- (a);
  \draw (g) -- (b);
  \draw (g) -- (c);
  \draw (dd) -- (d);
  \draw (ee) -- (e);
  \draw[preaction={draw=white, -,line width=6pt}] (ff) -- (f);
  \draw (ee) -- (i);
  \draw[preaction={draw=white, -,line width=6pt}] (dd) -- (g);
  \draw (ee) -- (g);
  \draw (ff) -- (g);
  \node (l) at (0,-8) {$\subpart(Y)$};
\end{tikzpicture}
    \caption{The figure shows the embedding of $\Part(Y)$ into $\subpart(Y)$ for $Y=\{x,y,z\}$ (subdiagram in shaded region).  For simplicity, distinct blocks in a partition are separated by $|$ instead of curly brackets. In $\Part(Y)$, the join-irreducible elements are the atoms $\{xy\ |\ z\},\{yz\ |\ x\},\{zx\ |\ y\}$. On the other hand, in $\subpart(Y)$, the join-irreducible elements are $\{x\},\{y\},\{z\},\{xy\},\{yz\},\{zx\}$, and the first three singletons are in particular its atoms.}
    \label{fig:Hasse}
\end{figure}

Note that for any $P\in\subpart(X)$, we have $\emptyset\leq P\leq \{X\}$, i.e. $\emptyset$ and $\{X\}$ are the zero and unit elements of $\subpart(X)$, respectively. It is well-known that $\Part(X)$ is a lattice \cite{roman2008lattices}. We show that $\Part(X)$ is a sublattice of $\subpart(X)$: Given any $P_1,P_2\in \subpart(X)$, let $X_1$ and 
$X_2$ be the underlying sets of $P_1$ and $P_2$, respectively.
\begin{enumerate}[label=(\roman*)]
    \item The join $P_1\vee P_2$ is the quotient set $(X_1\cup X_2)/\!\!\sim$, where $\sim$ is the smallest equivalence relation on $X_1\cup X_2$ containing $\sim_{P_1}$ and $\sim_{P_2}$. The join $P_1\vee P_2$ is also called the \textbf{finest common coarsening} of $P_1$ and $P_2$.
    \item The meet $P_1\wedge P_2$ is the quotient set $(X_1\cap X_2)/(\sim_{P_1}\cap\sim_{P_2})$. The meet $P_1\wedge P_2$ is also called the \textbf{coarsest common refinement} of $P_1$ and $P_2$.
\end{enumerate}
If $X=X_1=X_2$, then $P_1\vee P_2$ and $P_1\wedge P_2$ clearly belong to $\Part(X)$. Hence:

\begin{proposition}\label{prop:subpart is a lattice}
$\subpart(X)$ equipped with the refinement relation is a lattice. In particular, $\Part(X)$ is a sublattice of $\subpart(X)$.
\end{proposition}

\begin{remark}\label{rem:about subpart(X)} \begin{enumerate*}[label=(\roman*)]
    \item The atoms of $\subpart(X)$ are $\{\{x\}\}$ for $x\in X$. In what follows, let us assume that $\abs{X}\geq 2$.\label{item:about subpart(X)1}
    \item The join-irreducible elements of $\subpart(X)$ consist of all the atoms and all sets of the form $\{\{x,x'\}\}$ for different $x,x'\in X$. \label{item:about subpart(X)2}
\end{enumerate*}
\end{remark}

\begin{proposition}\label{prop:join representation} Let $X$ be any nonempty finite set.
\begin{enumerate*}[label=(\roman*)]
    \item $\subpart(X)$ is $\bigvee$-irreducibly generated. \label{item:join representation1}
    \item Let $P\in \subpart(X)$ be a nonzero element.  If $P$ includes a block $B$ with $\abs{B}\geq 3$, then $P$ has \emph{no} canonical join representation. \label{item:join representation2}
\end{enumerate*}
\end{proposition}

\begin{proof}\ref{item:join representation1}: Let $\sim_P$ be the subequivalence relation on $X$ corresponding to $P$. Then $P=\bigvee A$ where $A:=\left\{\{x,x'\}: x\sim_P x' \right\}$ (when $x=x'$, the set $\{x,x'\}$ is a singleton). By Rmk.~\ref{rem:about subpart(X)} \ref{item:about subpart(X)2}, $P=\bigvee A$ is a join representation by irreducible elements. 

\ref{item:join representation2}: Without loss of generality, assume that $\{x,y,z\} \in P$. Observe that the following three are minimal join representations of the single block partition $\{xyz\}$ and hence a minimum does not exist:
 \[\bigvee\left\{\{xy\},\{yz\}\right\},\ \ \bigvee\left\{\{yz\},\{zx\}\right\},\ \  \bigvee\left\{\{zx\},\{xy\}\right\}.\]
This implies that, given a minimal join representation $\bigvee A$ of $P$ by irreducible elements, exactly one of the three subsets $\left\{\{xy\},\{yz\}\right\},\ \left\{\{yz\},\{zx\}\right\},\  \left\{\{zx\},\{xy\}\right\}$ is a subset of $A$. Whatever the case is, the corresponding subset can be replaced by any of the other two, and thus there is no canonical join representation of $P$.
\end{proof}

\subsection{Posets with a flow and interleaving distances}\label{sec:posets with a flow}
We review the notion of \emph{poset with a flow} and its associated interleaving distance \cite{bubenik2015metrics,deSilva2018}.

\paragraph{Flows and interleavings}
For a poset $\PP$, let $\Ifunc_{\PP}$ be the identity map on $\PP$.
 \begin{definition}\label{def:flow}
A (strict) \textbf{flow} on a poset $\PP$ is a family  $\Omega:=\{\Omega_\eps:\PP\to \PP\}_{\eps\in [0,\infty)}$ of \pe{}s on $\PP$ such that
(i) $\Omega_s\leq\Omega_t$ for all $s\leq t$, (ii) $\Omega_{t}\Omega_s=\Omega_{t+s}$, for all $s,t\in [0,\infty)$, and (iii) $\Ifunc_{\PP}=\Omega_0$. We call the triple $\posetflow$ a \textbf{poset with a flow.}
\footnote{By weakening conditions (ii) and (iii), we obtain the notion of \emph{coherent flow} \cite{deSilva2018} (or  \textit{superlinear family of translations} \cite{bubenik2015metrics}). This level of generality is not required for the purpose of this paper.} 
 \end{definition}

We define an extended pseudometric between point in a poset with a flow \cite{deSilva2018} as follows:

  \begin{definition}[Interleaving of poset elements]
  \label{dfn:poset-ints}
  Let $\posetflow$ be a poset with a flow. For $\eps\in [0,\infty)$, any $p,q\in \PP$ are said to be \textbf{$\e$-interleaved} if $p\leq \Omega_\e(q)$ and $q\leq \Omega_\e(p)$.
  The \textbf{interleaving distance between $p$ and $q$} is defined as
 \[d_{\mathrm{\Omega}}(p,q):=\inf\left\{\e\in [0,\infty) : p,q\text{ are }\e\text{-interleaved}\right\}.\]
 If $p,q$ are not $\e$-interleaved for any $\e\in [0,\infty)$, then we declare that $d_{\mathrm{\Omega}}(p,q)=\infty$.
 \end{definition}

By \cite[Lem.~3.7]{deSilva2018} and \cite[Thm.~3.21]{bubenik2015metrics}, we know that $d_{\mathrm{\Omega}}$ is an extended pseudometric on $\PP$. For example, let $\R^n$ be equipped with the product order given as $(x_1,\ldots,x_n)\leq (y_1,\ldots,y_n)$ $\Leftrightarrow$ $x_i\leq y_i$ for each $i=1,\ldots,n$. Then, the supremum norm distance $\norm{\cdot-\cdot}_\infty$ in $\R^n$ coincides with the interleaving distance with the flow $\Omega$ \cite{deSilva2018} given as
\begin{equation}\label{eq:l-infinity}
   \Omega:= \big(\Omega_{\eps}:(x_1,\ldots,x_n)\mapsto (x_1,\ldots,x_n)+\eps(1,\ldots,1)\big)_{\eps\in [0,\infty)}.
\end{equation} 
Next, we introduce two special examples of Defn. \ref{dfn:poset-ints}.

\paragraph{Interleaving distance between \pe{}s.} 

 Let $(\PP,\leq,\Omega)$ be a poset with a flow, and let $(\QQ,\leq)$ be another poset.
 Recall that $[\PP,\QQ]$ is a poset (Rmk.~\ref{rmk:functor poset}).
 The flow $\Omega$ on $(\PP,\leq)$ yields the flow $-\cdot\Omega$, given by pre-composition with $\Omega$, on $[\PP,\QQ]$.
 Thus,  Defn.~\ref{dfn:poset-ints} is specialized to:
 
 \begin{definition}
 \label{def:Interleavings of persistent elements}
 The \textbf{interleaving distance between \pe{}s  $\Ffunc,\Gfunc:(\PP,\leq,\Omega)\to(\QQ,\leq)$} is defined as:
 \[\dint(\Ffunc,\Gfunc):=\inf\left\{\e\in [0,\infty): \Ffunc,\Gfunc\text{ are }\e\text{-interleaved w.r.t. the flow $-\cdot\Omega$}\right\}.\]
 \end{definition}

\paragraph{Interleaving distance between upper sets.}  
Let $(U(\PP),\subset)$ be the poset of upper sets of $\PP$ with the inclusion relation. 
Then, the flow $\Omega$ on $\PP$ gives rise to a family   $\widehat{\Omega}=(\widehat{\Omega}_\e)_{\e\in [0,\infty)}$ of \pe{}s  $U(\PP)\rightarrow U(\PP)$ given, for each $A\in U(\PP)$, as
\[\widehat{\Omega}_\e(A):=\left\{p\in \PP:\Omega_\e(p)\in A\right\}.\]
Indeed $\widehat{\Omega}_\e(A)$ is an upper set. To see this let $x\in \widehat{\Omega}_\e(A)$ and let $x\leq y$ in $\PP$. Then $\Omega_\e(x)\in A$ and $\Omega_\e(x)\leq \Omega_\e(y)$. Since $A$ is an upper set, $\Omega_\e(y)\in A$, implying that $y\in\widehat{\Omega}_\e(A)$.

\begin{proposition}
$\widehat{\Omega}$ is a flow on $U(\PP)$.
\end{proposition}
\begin{proof}
Let $A\in U(\PP)$. The equality $A=\widehat{\Omega}_0(A)$ is clear. Let $t\leq s$ in $[0,\infty)$. Then, 
    \begin{align*}
    \widehat{\Omega}_t(A)&=\left\{p\in\PP:\Omega_t(p)\in A\right\}\\
    &\subset\left\{p\in\PP: \Omega_{s}(p)\in A\right\}&\text{ since }\Omega_t(p)\leq \Omega_{s}(p)\text{ and }A\text{ is an upper set.}\\
    &=    \widehat{\Omega}_s(A).
    \end{align*}
Also, for any $s,t\in [0,\infty)$,
\[ 
    \widehat{\Omega}_t(\widehat{\Omega}_s(A))=\left\{p\in\PP:\Omega_t(p)\in \widehat{\Omega}_s(A)\right\}
    =\left\{p\in\PP:\Omega_s(\Omega_t(p))\in A\right\}
        =\left\{p\in\PP:\Omega_{t+s}(p)\in A\right\}
    =\widehat{\Omega}_{t+s}(A).
\] 
\end{proof}

\begin{definition}[Interleaving of upper sets]\label{def:Interleavings of upper sets}
  Let $(\PP,\leq,\Omega)$ be a poset together with a flow. Then, $d_{\widehat{\Omega}}$ will denote the induced interleaving distance on the poset $(U(\PP),\subset, \widehat{\Omega})$ of upper sets of $\PP$.
\end{definition}

The following remark and example will be useful in later sections.
\begin{remark}\label{rem:upper sets and omega hat properties} 
\begin{enumerate}[label=(\roman*)]
\item Arbitrary unions and intersections of upper sets in $\PP$ yield upper sets, implying that $U(\PP)$ is a complete lattice.\label{item:complete lattice} 
\item 
For any family $(A_i)_i$ of upper sets in $\PP$, we have
\[\widehat{\Omega}_\e\left(\bigcap_i A_i\right)=\bigcap_i\widehat{\Omega}_\e(A_i)\ \mbox{and } \widehat{\Omega}_\e\left(\bigcup_i A_i\right)=\bigcup_i\widehat{\Omega}_\e(A_i).\]
In other words, $\widehat{\Omega}_\eps$ preserves arbitrary meets and joins.\label{item:union and intersection preserved}
\end{enumerate}

\end{remark}

\begin{example}\label{ex:upset of int} 
Consider the poset
    $\Int:=\{(a,b)\in \R^2: a\leq b\}$,
the upper-half plane in $\R^2$ above the diagonal line $y=x$, equipped with the order $(a,b)\leq (a',b')$ $\Leftrightarrow$ $a'\leq a<b\leq b'$. Let us define the flow $\Omega$ on $\Int$ by 
\begin{equation}\label{eq:Int flow}
\Omega:= \big(\Omega_{\eps}:(a,b)\mapsto (a-\eps,b+\eps)\big)_{\eps\in [0,\infty)}.
\end{equation}
For the poset $(U(\Int),\subset,\widehat{\Omega})$ of upper sets in $\Int$, the distance $d_{\widehat{\Omega}}$  coincides with the Hausdorff distance $\dhaus$ in $(\Int,\norm{-}_{\infty})$ (Defn. \ref{def:Hausdorff distance}). This follows from the observation that for any $A\in U(\Int)$, the $\eps$-thickened set 
\begin{equation}\label{eq:thickening}
    A^{\eps}: =\left\{(a,b)\in \Int: \exists (a',b')\in A, \mbox{ such that } \norm{(a,b)-(a',b')}_{\infty}\leq \eps\right\},
\end{equation} coincides with $\widehat{\Omega}_{\eps}(A)$.
\end{example}

\subsection{Formigrams and their interleavings}\label{sec:formigrams}

In this section we review the notion of formigrams and their specialized interleaving  and Gromov-Hausdorff distances \cite{kim2018formigrams,kim2017stable} (note: Secs.~\ref{sec:generalizations}, \ref{sec:erosion and graded}, and \ref{sec:tripod} can be read without reading this section).

\paragraph{Formigrams.} We begin by reviewing the definition of dendrograms. Let us fix a nonempty finite set $X$.

\begin{definition}[\cite{carlsson2010characterization}]\label{dfn:dendrograms}	A \textbf{dendrogram} over a finite set $X$ is any function $\theta:\Rplus\rightarrow \Part(X)
 	$ such that the following properties hold: (1) $\theta(0)=\{\{x\}: x\in X\}$, (2) if $t_1\leq t_2$, then \hbox{$\theta(t_1) \leq \theta(t_2)$}, (3) there exists $T>0$ such that $\theta(t)=\{X\}$ for $t\geq T$, (4)  for all $t$ there exists $\eps>0$ s.t. $\theta(s)=\theta(t)$ for $s\in [t,t+\eps]$ (right-continuity) (see Fig.~\ref{fig:dendrogram and formigram} (A)). 
The \textbf{ultrametric induced by $\theta$} is the distance function $u_{\theta}:X\times X\rightarrow \Rplus$ defined as
\[u_{\theta}(x,x'):=\min\left\{t\in \R: \mbox{$x\sim_{\theta(t)}x'$}\right\}.\]
Note that we have the ultra-triangle inequality: $u_{\theta}(x,x'')\leq \max\left\{u_{\theta}(x,x'),u_{\theta}(x',x'')\right\}$ for every $x,x',x''\in X$.
\end{definition}

\begin{example}\label{ex:SLHC} The \textbf{single linkage hierarchical clustering} method on a finite metric space $(X,d_X)$ induces the dendrogram $\theta:\Rplus\rightarrow \Part(X)$ given as $t\mapsto X/\!\!\sim_t$, where $\sim_t$ is the transitive closure of the relation $R_t:=\{(x,x')\in X\times X:d_X(x,x')\leq t\}$ (Fig. \ref{fig:dendrogram and formigram} (A)). The mapping $(X,d_X)\mapsto (X,u_{\theta})$ is known to be 1-Lipschitz with respect to the Gromov-Hausdorff distance (Def.~\ref{def:Gromov-Hausdorff}) \cite[Prop.2]{carlsson2010characterization}, i.e.
\begin{equation}\label{eq:SHLH GH stability}
    \dgh((X,u_{\theta_X}),(Y,u_{\theta_Y}))\leq  \dgh((X,d_X),(Y,d_Y)).
\end{equation}
\end{example}

 \emph{Formigrams}, a mathematical model for time-varying clusters in  dynamic networks {\cite{kim2018formigrams}}, are a generalization of dendrograms. Formigrams are defined as \emph{constructible cosheaves over $\R$ \cite{curry2016classification,de2016categorified}} with values in $\subpart(X)$, which amounts to a (\emph{costalk}-)function $\R\to\subpart(X)$ described as follows.

\begin{definition}
\label{dfn:formi}
A \textbf{formigram} over $X$ is a \emph{function}\footnote{But not necessarily a poset map.}  $\theta:\R\to\subpart(X)$  satisfying the following:
there exists a finite set $\mathbf{crit}(\theta)=\{t_1,\ldots,t_n\}\subset \R$ of \textbf{critical points} s.t.
\begin{enumerate*}[label=(\roman*)]
    \item $\theta$ is constant on $(t_i,t_{i+1})$ for each $i=0,\ldots,n$, where $t_0=-\infty$ and $t_{n+1}=\infty$.
    \item At each critical point, $\theta$ is locally maximal, i.e.
     \label{item:formi-comparability}
\end{enumerate*}
\begin{equation}\label{eq:comparability}
\mbox{ for $t_i\in \crit(\theta)$ and for $\eps\in \big[0,\min_{j\in\{i,i+1\}}(t_j-t_{j-1})\big)$,  }\ \ 
        \theta(t_i-\e)\leq \theta(t_i) \geq \theta(t_i+\e).
    \end{equation} 
\end{definition}

See Fig.~\ref{fig:dendrogram and formigram} (B), Fig.~\ref{fig:formi2}, and Fig.~\ref{fig:precosheaf}~(A) for illustrative examples. Note that $\crit(\theta)$ is not necessarily unique nor nonempty. We also remark that, given a dendrogram $\theta:\Rplus\rightarrow \Part(X)$,  $\theta$ can be seen as a formigram by trivially extending its domain to $\R$, i.e. $\theta(t):=\emptyset$ for $t\in (-\infty,0)$.

Since $\subpart(X)$ is a poset,  the collection of all formigrams on $X$ can be regarded as a poset in its own right when endowed with the partial order defined as 
    $\theta\leq\theta' \Leftrightarrow \theta(t)\leq\theta'(t)\text{, for all }t\in\R.$

\begin{definition}\label{dfn:poset of formigrams}
By $\Formi(X)$, we denote the poset of all formigrams over $X$. 
\end{definition}

\paragraph{Interleaving distance on $\Formi(X)$.}
For $t\in\R$ and $\e\in[0,\infty)$ we denote the closed interval $[t-\e,t+\e]$ of $\R$ by $[t]^\e$.
\begin{definition}[{\cite{kim2017stable}}]
\label{def:smoothing}
Let $\theta:\R\to\subpart(X)$ be a formigram and let $\e\in [0,\infty)$.
We define the \textit{$\e$-smoothing $\Sfunc_\e(\theta):\R\to\subpart(X)$ of $\theta$} as $\Sfunc_\e(\theta)(t):=\bigvee\left\{\theta(s):{s\in [t]^\e}\right\}\text{, for }t\in\R.$
\end{definition}

The family $\Sfunc=(\Sfunc_{\eps})_{\eps\in [0,\infty)}$ is a flow on the poset $\Formi(X)$ (Defn. \ref{def:flow}):
\begin{proposition}[{\cite{kim2017stable}}]\label{prop:S_e-properties}
Let $\e\in [0,\infty)$, and let $\theta$ be a formigram over $X$. Then, \begin{enumerate*}[label=(\roman*)] \item the $\e$-smoothing $\Sfunc_\e(\theta)$ of $\theta$ is also a  formigram over $X$ and $\Sfunc_0(\theta)=\theta$. \item Also, for any $a,b\in [0,\infty)$, 
    we have: $\Sfunc_{a}(\Sfunc_{b}(\theta))=\Sfunc_{a+b}(\theta)$. \item Let $\theta'$ be another formigram over $X$.   For any $\e\in [0,\infty)$, 
    we have: $\theta\leq \theta'\Rightarrow \Sfunc_\e(\theta)\leq\Sfunc_\e(\theta')$.
\end{enumerate*} 
\end{proposition}

By Defn.~\ref{dfn:poset-ints} and Prop.~\ref{prop:S_e-properties}, we have: 
\begin{definition}
\label{def:intrinsic}
The \textbf{interleaving distance between $\theta,\theta'\in \Formi(X)$} is
\[\df(\theta,\theta'):=\inf\left\{\e\in [0,\infty) : \theta,\theta'\text{ are }\e\text{-interleaved w.r.t. the flow $\Sfunc$}\right\}.\] 
\end{definition}

Since Defn.~\ref{def:intrinsic} is a special case of Defn.~\ref{dfn:poset-ints}, we readily know that $\df$ is an extended pseudometric. In fact, $\df$ is an extended \emph{metric}, not just a pseudometric. This can be proved by exploiting the fact that a formigram has finitely many critical points; similar ideas can be found in \cite[Thm.~4.5]{kim2017stable}. We remark that when restricted to \emph{treegrams} \cite{smith2016hierarchical} (cf. Rmk.~\ref{rem:isometry theorem for formigrams} \ref{item:isometry theorem for formigrams2}) $\df$ agrees with what's been called the \emph{$\ell^\infty$-cophenetic metric on phylogenetic trees \cite{cardona2013cophenetic,munch2019ell}}.

\begin{example}
\label{ex:loop}
Let $\delta>0$. Consider the formigrams $\theta,\theta':\R\to\subpart(\{x,y\})$ defined as:
\begin{figure}
 \centering
\includegraphics[width=\textwidth]{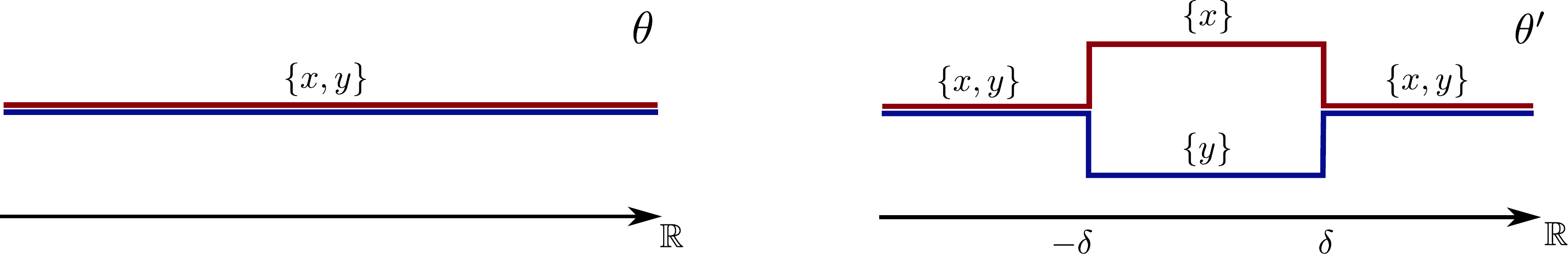}
    \caption{The two formigrams in Example \ref{ex:loop} 
    }
    \label{fig:formi2}
\end{figure}
\[\theta(t)=\{\{x,y\}\}\text{, for all }t\in\R,\ \ \  \ \ \ \ 
\theta'(t)=
\begin{cases} 
      \{\{x,y\}\}, & \text{if }t\leq -\delta\\
     \{\{x\},\{y\}\}, & \text{if }-\delta< t<\delta\\
     \{\{x,y\}\}, & \text{if }t\geq\delta
   \end{cases}
\]
(see Fig.~\ref{fig:formi2}).
We prove $d_\mathrm F(\theta,\theta')=\delta$. Let $\e\in [0,\infty)$.
Because the partition $\{\{x\},\{y\}\}$  refines the partition $\{\{x,y\}\}$ , we  have $\theta'\leq\Sfunc_\e(\theta)$.
Let us assume $\e\in [0,\delta)$: Then, we have $\Sfunc_\e(\theta')(0)=\bigvee_{s\in [0]^\e}\theta'(s)=\{\{x\},\{y\}\}$ and $\theta(0)=\{\{x,y\}\}$.
Thus, $\theta'\not\leq\Sfunc_\e(\theta)$, and $\theta,\theta'$ are not $\e$-interleaved.
On the other hand, $\Sfunc_\e(\theta')=\theta$ for $\e\in[\delta,\infty)$.
and thus $d_\mathrm F(\theta,\theta')=\delta$.
\end{example}

\paragraph{Gromov-Hausdorff distance between formigrams.} We may wish to compare formigrams over different sets. To this end, we revisit the Gromov-Hausdorff distance $\dintf$ between formigrams {\cite{kim2017stable,kim2018formigrams}}; the naming of the distance is based on the fact that it generalizes the Gromov-Hausdorff distance between finite ultrametric spaces \cite{kim2018formigrams} (see Rmk.~\ref{rem:relating to GH} \ref{item:relating to GH2}). See Defn.~\ref{def:Gromov-Hausdorff} for the original definition of the Gromov-Hausdorff distance between compact metric spaces. (\emph{Note}: All the Gromov-Hausdorff distances in this paper will turn out to be special instances of the generalized Gromov-Hausdorff distance from Defn.~\ref{def:generalized tripod distance}, as proved in Prop.~\ref{prop:omnipresence of dt} in the appendix.)

 Let $X,Z$ be two nonempty sets, let $P\in\subpart(X)$, and let $\varphi:Z\twoheadrightarrow X$ be a surjective map. The \textbf{pullback} of $P$  via $\varphi$ is the subpartition of $Z$ defined as $\varphi^\ast P:=\{\varphi^{-1}(B)\subset Z: B\in P\}$, implying: \begin{equation}\label{eq:pullback equivalence}\mbox{$z,z'\in Z$ belong to the same block of $\varphi^\ast P$ $\Leftrightarrow \varphi(z),\varphi(z')\in X$ belong to the same block of $P$.}\end{equation}

 Let $\theta_X$ be a formigram over $X$. The pullback of $\theta_X$ via $\varphi$ is the formigram 
\[\mbox{$\varphi^\ast\theta_X:\R\rightarrow \subpart(Z)$ such that   }\left(\varphi^\ast\theta_X\right)(t):=\varphi^\ast\left(\theta_X(t)\right).
\]

Let $X$ and $Y$ be any two nonempty sets. A \emph{tripod $R$ between $X$ and $Y$} is a pair of surjections $\tripod$ from any set $Z$ onto $X$ and $Y$ \cite{memoli2017distance}. For $x\in X$ and $y\in Y$, we write $(x,y)\in R$ when there exists $z\in Z$ such that $x=\varphi_X(z)$ and $y=\varphi_Y(z)$.

\begin{definition}\label{def:general formigram distance} Let $\theta_X$,$\theta_Y$ be any two formigrams over $X$ and $Y$, respectively. The \textbf{Gromov-Hausdorff distance between $\theta_X$ and $\theta_Y$} is defined as: 
\[\dintf(\theta_X,\theta_Y):=\frac{1}{2}\min_R \df (\varphi_X^\ast\theta_X,\varphi_Y^\ast\theta_Y), \]
where the minimum is taken over all tripods between $X$ and $Y$.\footnote{Because $X$ and $Y$ are finite, the minimum is always achieved by a certain tripod  $R$. In fact, it suffices to consider subsets $Z\subset X\times Y$ which project \emph{onto} $X$ and $Y$ via the canonical projections with  $\varphi_X,\varphi_Y$ being the canonical projections.} \footnote{The Gromov-Hausdorff distance between formigrams in this paper is actually called the \emph{formigram interleaving distance} in \cite{kim2018formigrams}, \cite[Defn. 4.11]{kim2017stable}.  In this paper, we reserve the name \emph{interleaving distance} for Defn.~\ref{def:intrinsic} because $\dgh$ is not the interleaving distance in the sense of Defn.~\ref{dfn:poset-ints}. A close relationship between the Gromov-Hausdorff distance and the interleaving distance is highlighted in \cite{bubenik2017interleaving,rolle2020stable,scoccola2020locally}.} 
\end{definition}

It directly follows from Defn.~\ref{def:intrinsic} and \ref{def:general formigram distance} that, given any two formigrams $\theta_X$ and $\theta_X'$ over the \emph{same} underlying set $X$, we have
$2\ \dgh(\theta_X,\theta_X') \leq \df(\theta_X,\theta_X').$

\section{Interleaving by parts and geodesicity} 
\label{sec:generalizations}

We \emph{decompose} the interleaving distance $\dint$ between poset maps (Sec.~\ref{sec:join decomposition of interleavings}) and harness it to show the geodesicity of $\dint$ in certain settings (Sec.~\ref{sec:geodesic}).

\subsection{Join representations of interleavings between poset maps}\label{sec:join decomposition of interleavings}

The goal of this section is to establish  Thms.~\ref{thm:gen-metric-dec} and \ref{thm:general-metric-dec}.

\paragraph{Join representations for \pe{}s.}  Recall that any poset $\PP$ contains at least one trivial join-dense subset ($\PP$ itself). In this section, we adhere to:

\begin{convention}\label{convention:irreducibly generated}
 $(\PP,\leq)$ denotes a poset and  $(\QQ,\leq)$ denotes 
 a poset  containing a zero element\footnote{If, a priori, $\QQ$ does not contain a zero element, one can simply add one to $\QQ$. Namely, replace $\QQ$ by $\QQ\cup\{0\}$ where $0$ is forced to be the smallest element in $\QQ\cup\{0\}$, by definition.} and  with a distinguished join-dense subset $B\subset\QQ$. 

\end{convention} 

\begin{proposition}\label{prop:canonical representation}Let $q\in \QQ$. Then, $q=\bigvee 
\left(B\cap [0,q]\right)$.
\end{proposition}

\begin{proof} Clearly, we have $\bigvee [0,q]=q$.
As $B$ is join-dense, there exists $B'\subset B$ s.t. $\bigvee B'=q$. 
Since $B' \subset B\cap [0,q]\subset [0,q]$, we have $\bigvee\left( B\cap [0,q]\right)=q$, as desired.
\end{proof}

The following proposition is straightforward. 
\begin{proposition}
\label{prop:semilattice} 

Given any subfamily $\{\Ffunc_a:a\in A\}$ of $[\PP,\QQ]$, assume that for all $p\in \PP$, the join $\bigvee_{a\in A} \Ffunc_a(p)$ exists in $\QQ$. Then the join $\bigvee \{\Ffunc_a: a\in A\}$ exists in the poset $[\PP,\QQ]$, which is given by $p \mapsto \bigvee_{a\in A} \Ffunc_a(p)$.
\end{proposition}
Let $q\in \QQ$ and let $U\subset \PP$ be an upper set of $\PP$.
We define the upper set indicator map $\Ifunc^{U}_{q}:(\PP,\leq)\to(\QQ,\leq)$ as
\begin{equation}\label{eq:indicator}
        (\Ifunc^{U}_{q})(p):=
\begin{cases} 
      q, & p\in U \\
      0, & \text{otherwise.}  
   \end{cases}
\end{equation}

\begin{definition}\label{dfn:upper set indicators}
Let $\Ffunc \in [\PP,\QQ]$ and let $b\in B$. 
For the upper set $\Ffunc^{-1}(b^\uparrow):=\{p\in \PP: b\leq \Ffunc(p)\}$, let us define the \textbf{$b$-part} of $\Ffunc$ as $\Ffunc_b:=\Ifunc_b^{\Ffunc^{-1}(b^\uparrow)}$. 
\end{definition} 
We introduce a certain join representation of a \pe{} which is the source of many results in this paper.

\begin{proposition}[Join representation via indicator maps]\label{prop:alg-dec}
For any \pe{}
$\Ffunc:\PP\to \QQ$,
\begin{equation}\label{eq:alg-dec}
    \Ffunc=\bigvee \left\{\Ffunc_b:{b\in B} \right\}.
\end{equation}
\end{proposition}

\begin{proof}
Fix $p\in \PP$. Let us define $A_p,B_p\subset \QQ$ as $A_p:=B\cap[0,\Ffunc(p)]$ and $B_p:=\left\{\Ffunc_b(p):{b\in B} \right\}$. 
Then,
\begin{align*}
  \Ffunc(p)&=
    \bigvee A_p
    &\mbox{by Prop.~\ref{prop:canonical representation}}
    \\
    &=\bigvee \left(A_p\cup\{0\}\right)\\
   &=\bigvee \left(B_p\cup \{0\}\right) &\mbox{since $A_p\cup\{0\}=B_p\cup\{0\}$}\\
    &=\bigvee B_p\\
    &=\left(\bigvee\left\{\Ffunc_b:{b\in B} \right\}\right)(p)&\mbox{by Prop. \ref{prop:semilattice},}
\end{align*}
which proves the equality in Eqn.~(\ref{eq:alg-dec}).
\end{proof}

\begin{remark}\label{rem:about join representation of F}
\begin{enumerate}[label=(\roman*),topsep=0pt,itemsep=-1ex,partopsep=1ex,parsep=1ex]
    \item The join representation in (\ref{eq:alg-dec}) is functorial in the sense that
$\Ffunc\leq \Gfunc$ in $[\PP,\QQ]$ $\Leftrightarrow$  for all $b\in B$,\  $\Ffunc_b\leq \Gfunc_b.$  \label{item:functoriality} 
\item One can establish ``duals" of Conv.~\ref{convention:irreducibly generated}, Props.~\ref{prop:canonical representation},  \ref{prop:semilattice} and Defn.~\ref{dfn:upper set indicators} which permit representing $\Ffunc$ as a \emph{meet} of $\Ffunc$'s ``dual" $b$-parts, instead of the join representation given in (\ref{eq:alg-dec}). We have not found any significant use of this dual statement though for the purpose of this paper.
\item 
The  $\Ffunc_b$s in Eqn.~(\ref{eq:alg-dec}) are \emph{not} necessarily join-irreducible in $[\PP,\QQ]$. However, there is a special case: Let $\PP$ be a totally ordered set and assume that $B$ consists solely of join-irreducible elements of $\QQ$ (cf. Rmk.~\ref{rem:smallest join-dense}). Then, each $\Ffunc_b$ is join-irreducible and thus, in that case,  equality (\ref{eq:alg-dec}) could be regarded as a join \emph{decomposition} of $\Ffunc$.
\end{enumerate}

\end{remark}

\paragraph{Interleaving by parts.}
We show that the interleaving distance between \pe{}s admits a join representation which is compatible with the join representation in (\ref{eq:alg-dec}) . 

\begin{theorem}[Interleaving by parts]
\label{thm:gen-metric-dec}
Let $(\PP,\leq,\Omega)$ be a poset with a flow, and let $\SSS$ be a poset 
with zero and a join-dense subset $B$.
For any $\Ffunc,\Gfunc:(\PP,\leq)\to (\SSS,\leq)$, 
\begin{equation}\label{eq:gen-metric-dec}
\dint(\Ffunc,\Gfunc)=\sup_{b\in B}\dint(\Ffunc_b,\Gfunc_b).
\end{equation}
\end{theorem}
Note that, in the standard ordered set $(\R,\leq)$, we have $\bigvee A=\sup A$ for any $A\subset \R$. Therefore, invoking Eqn.~(\ref{eq:alg-dec}), we can rewrite Eqn.~(\ref{eq:gen-metric-dec}) as:
\[\dint\left(\bigvee_{b\in B} \Ffunc_b,\bigvee_{b\in B} \Gfunc_b\right)=\bigvee_{b\in B}\dint(\Ffunc_b,\Gfunc_b),
\]
which shows that the join operation of Eqn.~(\ref{eq:alg-dec}) commutes with $\dint$.

\begin{remark}\label{rem:interleaving by parts analogy} Eqn.~(\ref{eq:gen-metric-dec}) is strongly analogous to the well-known decomposability of the interleaving distance between persistence modules $\R\rightarrow \vect$ (Defn.~\ref{def:interleaving distance between persistence modules} in the appendix).

Assume that $B$ is finite for simplicity. Then, one can check that the RHS of (\ref{eq:gen-metric-dec}) is equal to 
\begin{equation}\label{eq:interleaving decomposition1}
    \min_{\tau:B\rightarrow B}\max_{b\in B}\dint(\Ffunc_b,\Gfunc_{\tau(b)})
\end{equation}
where the minimum is taken over all bijections $\tau:B\rightarrow B$. 
Now consider two any persistence modules $M,N:\R\rightarrow \vect$ with indecomposable direct sum decompositions $M\cong \bigoplus_{i\in I} M_i$ and $N\cong \bigoplus_{j\in J} N_j$ where $\abs{I},\abs{J}<\infty$. Let $K:=I\sqcup J$. Extend $M$ and $N$ to $K$ as follows: $M_j:=0$ for $j\in J$ and let $N_i:=0$ for $i\in I$. Then $\dint(M,N)$ is equal to the following which is in the same form as (\ref{eq:interleaving decomposition1}): \[\min_{\tau:K\rightarrow K}\max_{k\in K} \dint(M_k,N_{\tau(k)}).\]
\end{remark}

\begin{proof}[Proof of Thm. \ref{thm:gen-metric-dec}]Let $\eps\in [0,\infty)$.
\begin{align*}
          &\Ffunc,\Gfunc\text{ are }\e\text{-interleaved}\\
          \Leftrightarrow\hspace{1em}&\Ffunc\leq \Gfunc\Omega_\e\text{ and }\Gfunc\leq \Ffunc\Omega_\e.\\
          \Leftrightarrow      \hspace{1em}&\text{For any }b\in B, \text{ }\Ffunc_{b}\leq (\Gfunc\Omega_\e)_b\text{ and }\Gfunc_b\leq (\Ffunc\Omega_\e)_b &\text{  by  Rmk.~\ref{rem:about join representation of F} \ref{item:functoriality} }\\
         \stackrel{(\ast)}{\Leftrightarrow}\hspace{1em}&\text{For any } b\in B, \text{ }\Ffunc_b\leq \Gfunc_b\Omega_\e\text{ and }\Gfunc_b\leq \Ffunc_b\Omega_\e &\text{ see below}\\         \Leftrightarrow\hspace{1em}&\text{For any } b\in B, \text{ }\Ffunc_b,\Gfunc_b\text{ are }\e\text{-interleaved} & \text{by Def. \ref{def:Interleavings of persistent elements}.}
         \end{align*}
For $(\ast)$, it suffices to show that $(\Ffunc\Omega_\eps)_{b}=\Ffunc_b\Omega_\eps$. To this end, we prove $p\in (\Ffunc\Omega_{\eps})^{-1}(b^\uparrow) \Leftrightarrow \Omega_\eps(p)\in\Ffunc^{-1}(b^\uparrow)$. Indeed, for $p\in \PP$,
\[    p\in (\Ffunc\Omega_{\eps})^{-1}(b^\uparrow)\ \Leftrightarrow \ b\leq (\Ffunc\Omega_\e)(p)
    \ \Leftrightarrow \  b\leq \Ffunc(\Omega_\e(p)) 
    \ \Leftrightarrow \ \Omega_\e(p)\in \Ffunc^{-1}(b^\uparrow). 
\]
\end{proof}

\paragraph{Interleaving distance between upper set indicator maps.} We represent the RHS of (\ref{eq:gen-metric-dec}) as the interleaving distance between upper sets of $\PP$ (Defn.~\ref{def:Interleavings of upper sets}). Recall Defn~\ref{dfn:upper set indicators}. 
\begin{proposition}\label{prop:uppersets}
 Let $(\PP,\leq,\Omega)$ be a poset with flow. Let $U,V\in U(\PP)$ and let $x\in \SSS$. Then, we have:
 \[d_\mathrm{I}(\Ifunc_x^{U},\Ifunc_x^{V})=d_{\mathrm{\widehat{\Omega}}}(U,V).
 \]
 \end{proposition}
 \begin{proof}
 \begin{align*}
          &\Ifunc_x^{U},\Ifunc_x^{V}\text{ are }\e\text{-interleaved.}\\
          \Leftrightarrow\hspace{1em}&\Ifunc_x^{U}\leq \Ifunc_x^{V} \Omega_\e\text{ and }\Ifunc_x^{V}\leq \Ifunc_x^{U}\Omega_\e.\\
         \Leftrightarrow\hspace{1em}&\text{ For any }p\in \PP, \Ifunc_x^{U}(p)\leq \Ifunc_x^{V}( \Omega_\e(p))\text{ and }\Ifunc_x^{V}(p)\leq \Ifunc_x^{U}(\Omega_\e(p)).\\
  \Leftrightarrow\hspace{1em}&\text{ For any }p\in \PP, \left(p\in U\Rightarrow \Omega_\e(p)\in V\right)\text{ and }\left(p\in V\Rightarrow \Omega_\e(p)\in U\right).\\
   \Leftrightarrow\hspace{1em}&\left(\text{For any }p\in \PP: p\in U\Rightarrow \Omega_\e(p)\in V\right)\text{ and }\left(\text{For any }p\in \PP:p\in V\Rightarrow \Omega_\e(p)\in U\right).\\
 \Leftrightarrow\hspace{1em}&U\subset\widehat{\Omega}_\e(V)\text{ and } V\subset\widehat{\Omega}_\e(U)\\
   \Leftrightarrow\hspace{1em}&U,V\text{ are }\e\text{-interleaved.}
  \end{align*}
 \end{proof}
By Thm.~\ref{thm:gen-metric-dec} and Prop. \ref{prop:uppersets}, $\dint$ coincides with the $\ell^\infty$-product metric 
of $\abs{B}$ copies of $d_{\mathrm{\widehat{\Omega}}}$ on $U(\PP)$:

\begin{theorem}[Decomposition of $\dint$]
\label{thm:general-metric-dec}
Let $(\PP,\leq,\Omega)$ be a poset with a flow, and let $\SSS$ be a poset 
with zero and a join-dense subset $B$.
For any $\Ffunc,\Gfunc:(\PP,\leq)\to(\SSS,\leq)$ we have
 $$\dint(\Ffunc,\Gfunc)=\sup_{b\in B}d_{\mathrm{\widehat{\Omega}}}\big(\Ffunc^{-1}(b^\uparrow),\Gfunc^{-1}(b^\uparrow)\big).$$
 \end{theorem}
 
\begin{remark}\label{rmk:comparison with isometry theorem}
Thm.~\ref{thm:general-metric-dec} is analogous to the celebrated \textit{isometry theorem} in TDA \cite{bauer2013induced,chazal2009proximity,lesnick2015theory} in the following sense: By Thm.~\ref{thm:general-metric-dec}, we can make use of the collections $$\mbox{$B(\Ffunc):=\left\{\Ffunc^{-1}(b^\uparrow)\right\}_{b\in B}$  and  $B(\Gfunc):=\left\{\Gfunc^{-1}(b^\uparrow)\right\}_{b\in B}$},$$
for computing $\dint(\Ffunc,\Gfunc)$. Analogously, for any persistence modules $M,N:\R\rightarrow \vect$ (Defn.~\ref{dfn:persistence module}), their \textit{barcodes} (or equivalently \textit{persistence diagrams}) 
are utilized for computing the interleaving distance between $M$ and $N$ via the bottleneck distance \cite{bauer2013induced,chazal2009proximity}. \label{item:comparison with isometry theorem}
\end{remark}

\subsection{Geodesicity of interleavings between poset maps}\label{sec:geodesic}

The goal of this section is to prove that the interleaving distance between \pe{}s $\PP\rightarrow \QQ$ is geodesic \emph{when $\QQ$ is a complete lattice} (Thm.~\ref{thm:di is geodesic}). 
This proves in a uniform way that many known metrics that will be discussed in later sections are all geodesic.

In the rest of the paper, any extended pseudometric space $(M,d)$ is simply referred to as a metric space. Any $x\in M$ will be identified with the class $[x]:=\{y:d(x,y)=0\}$. Let $x,y\in M$ with $d(x,y)<\infty$. A continuous map $g:[0,1]\rightarrow M$ is called a \textbf{path} from $x$ to $y$  if 
$g(0)=x$ and $g(1)=y$. The path $g$ is called \textbf{geodesic} if
\[d(g(s),g(t))=\abs{s-t}\cdot d(x,y)\] for all $s,t\in [0,1]$. If there exists a geodesic path between every pair of points in $M$ within a finite distance, $M$ is called \textbf{geodesic}.

\begin{proposition}\label{prop:geodesic for upper sets}For any poset $(\PP,\leq,\Omega)$ with a flow, the distance $d_{\mathrm{\widehat{\Omega}}}$ on $U(\PP)$ (Defn. \ref{def:Interleavings of upper sets}) is geodesic.

\begin{proof}Let $A,B\in U(\PP)$ with $\domhat(A,B)=\rho\in (0,\infty)$. For $t\in[0,1]$, let $A_t:=\omegahat_{\rho t}(A)\cap\omegahat_{\rho(1-t)}(B)$. We claim that $t\mapsto A_t$ for $t\in [0,1]$ is a geodesic path from $A$ to $B$.\footnote{This construction is similar to the construction of Hausdorff geodesic paths given in \cite{serra1998hausdorff}.}  First, we claim $\domhat(A,A_0)=0$ (the equality $\domhat(B,B_0)=0$ is proved similarly). To this end, we show that, for \emph{any} $\delta>0$, $A_0\subset \omegahat_\delta(A)$ and $A\subset \omegahat_\delta(A_0)$.  Let $\delta>0$. By construction we have $A_0\subset A \subset \omegahat_\delta(A)$. Next, 
\begin{align*}
    A&\subset \omegahat_\delta(A)\cap \omegahat_{\delta+\rho}(B)& A\subset\omegahat_{\rho+\delta}(B)\ \mbox{since $\domhat(A,B)=\rho$}\\&\subset \omegahat_\delta(A\cap\omegahat_{\rho}(B))&\mbox{by Rmk.~\ref{rem:upper sets and omega hat properties} \ref{item:union and intersection preserved}}\\&=\omegahat_\delta(A_0)&\mbox{by definition}. 
\end{align*}
Now fix any $s<t$ in $[0,1]$ and we show that $\domhat(A_s,A_t)=(t-s)\rho$. We claim that $\domhat(A_s,A_t)\leq (t-s)\rho$. 
By Rmk.~\ref{rem:upper sets and omega hat properties} \ref{item:union and intersection preserved}, \[\omegahat_{\rho(t-s)}(A_s)=\omegahat_{\rho t}(A)\cap \omegahat_{\rho (1+t-2s)}(B)\supset \omegahat_{\rho t}(A)\cap \omegahat_{\rho (1-t)}(B)=A_t.\]
Similarly, one has $\omegahat_{\rho(t-s)}(A_t)\supset A_s$. This shows that $\domhat(A_s,A_t)\leq (t-s)\rho$.
\end{proof}

\end{proposition}

\begin{theorem}\label{thm:di is geodesic}
Let $\posetflow$ be a poset with a  flow and $(\QQ,\leq)$ is a complete lattice. The interleaving distance $\dint$ on $[\PP,\QQ]$ is geodesic.
\end{theorem}
We actually do not need the completeness of $\QQ$ with respect to meets in the assumption. However, it is automatically guaranteed under the assumption that $\QQ$ is complete with respect to joins  \cite[Thm.~3.24]{roman2008lattices}, which we need in the proof below.

\begin{proof}
Let $B$ be any join-dense subset of $\PP$. By Prop.~\ref{prop:geodesic for upper sets}, for each $b\in B$, there exists a geodesic path $g_b:[0,1]\rightarrow U(\PP)$ from $\Ffunc^{-1}(b^\uparrow)$ to $\Gfunc^{-1}(b^\uparrow)$. We  \emph{embed} these paths into $[\PP,\QQ]$ and \emph{assemble} the embedded paths via $\bigvee$: For each $t\in [0,1]$, define $\Hfunc_t:(\PP,\leq)\to(\SSS,\leq)$ by
    \[p\mapsto \bigvee_{b\in B}  \Ifunc^{g_b(t)}_{b}(p),\]
which is well-defined since $\QQ$ is a complete lattice.

Now we claim that $t\mapsto \Hfunc_t$ is a geodesic path from $\Ffunc$ to $\Gfunc$. It is clear that $\Hfunc_0=\Ffunc$ and $\Hfunc_1=\Gfunc$. For every $s,t\in [0,1]$:
\begin{align*}
    \dint(\Hfunc_s,\Hfunc_t)&=\dint \left(\bigvee_{b\in B}\Ifunc^{g_b(s)}_{b},\bigvee_{b\in B}\Ifunc^{g_b(t)}_{b}\right)&\mbox{by definition}
    \\&=\sup_{b\in B} d_{\mathrm{\widehat{\Omega}}}(g_b(s),g_b(t))&\mbox{by Thm.~\ref{thm:general-metric-dec}}
    \\&=\sup_{b\in B}\abs{s-t}\cdot d_{\mathrm{\widehat{\Omega}}}\left(\Ffunc^{-1}(b^\uparrow),\Gfunc^{-1}(b^\uparrow)\right)&\mbox{since $g_b$ is a geodesic path}\\&=\abs{s-t}\cdot \sup_{b\in B} d_{\mathrm{\widehat{\Omega}}}\left(\Ffunc^{-1}(b^\uparrow),\Gfunc^{-1}(b^\uparrow)\right)\\&=\abs{s-t}\cdot \dint (\Ffunc,\Gfunc)&\mbox{Thm.~\ref{thm:general-metric-dec}.}
\end{align*}
\end{proof}
We remark that when the target poset $\QQ$ is not a complete lattice, one may replace $\QQ$ by its \emph{Dedekind–MacNeille completion $\bar{\QQ}$ (i.e.~the smallest complete lattice containing $\QQ$) and consider the geodesic path in $[\PP,\bar{\QQ}]$.} 

\begin{remark}The construction of the geodesic path above is analogous to the construction of a geodesic path between persistence diagrams of persistence modules $\R\rightarrow\vect$  via a linear interpolation guided by an optimal matching \cite{chowdhury2019geodesics}.
In the proof above,  the collection $\{\Ffunc^{-1}(b^\uparrow):b\in B\}$ can be viewed as a proxy for the ``persistence diagram" of $\Ffunc$.
\end{remark}

\section{Consequences}\label{sec:applications}

This section describes a number of consequences of Thms.~\ref{thm:gen-metric-dec} and \ref{thm:general-metric-dec}. 

In Sec.~\ref{sec:erosion and graded} 
we establish a connection between the erosion distance \cite{patel2018generalized} and the graded rank functions \cite{betthauser2019graded}. In Sec.~\ref{sec:tripod} we show that the Gromov-Hausdorff distance can be recast within the framework of interleaving distances and thereby obtain a far reaching generalization of this distance. Therein, basic properties of this distance are established including its universality and geodesicity.
In Sec.~\ref{sec:Interleaving distance between hierarchical clusterings} and \ref{sec:Structural theorem for the formigram interleaving distance} we prove that computing the interleaving distance between multiparameter hierarchical clusterings reduces to computing  Hausdorff distances in Euclidean spaces. This establishes an equivalence between known metrics for comparing multiparameter hierarchical clusterings.
In Sec.~\ref{sec:Structural theorem for the formigram interleaving distance} we determine the computational complexity of the interleaving distance between formigrams in Defn.~\ref{def:intrinsic}.

\begin{framed}
\begin{convention}
In the rest of the paper, the posets $\R_{\geq 0}$ and $\R^n$ (for any $n\in\N$) are equipped with the flow in Eqn.~(\ref{eq:l-infinity}). Also, the poset  $\Int$ is equipped with the flow in Ex.~\ref{ex:upset of int}.
\end{convention}
\end{framed}

\subsection{Erosion distance and graded rank functions}\label{sec:erosion and graded}

The goal of this section is to establish a connection between the erosion distance and the \emph{graded rank function} \cite{betthauser2019graded} and to provide an interpretation of the join representation given in (\ref{eq:alg-dec}).

Although the erosion distance has a fairly general form \cite{kim2018generalized,patel2018generalized,puuska2017erosion}, its most basic use is for comparing two monotonically decreasing integer-valued maps as in \cite[Sec.~5]{kim2020spatiotemporal}. We restrict ourselves to this basic setting. Let $\Zop$ be the opposite poset of the nonnegative integers, i.e. $a\leq b$ in $\Zop$ $\Leftrightarrow$ $b\leq a$ in $\Zplus$.

\begin{definition}\label{def:erosion distance} Let $\posetflow$ be a poset with a flow. Given $\Ffunc,\Gfunc:\posetflow\rightarrow \Zop$, the interleaving distance $\dint(\Ffunc,\Gfunc)$ is called the ($\PP$-)\textbf{erosion distance}. 
\end{definition}

In $\Zop$, we have $\vee\{m,n\}=\min\{m,n\}$ and each element of $\Zop$ is join-irreducible. This implies that $\Zop$ itself is the unique join-dense subset of $\Zop$. By virtue of Thm.~\ref{thm:general-metric-dec}, we have

\[\dint(\Ffunc,\Gfunc)=\sup_{n\in \Zplus}d_{\widehat{\Omega}}\left(\Ffunc^{-1}[0,n],\Gfunc^{-1}[0,n]\right).\]

Recall the Hausdorff distance in $(\Int,\norm{-}_{\infty})$ 
(Ex.~\ref{ex:upset of int} and Defn.~\ref{def:Hausdorff distance}). We establish the following relationship between the erosion distance and the \emph{graded rank function} \cite{betthauser2019graded}, 
which directly follows from the above equality and Ex. \ref{ex:upset of int}.
\begin{theorem}\label{thm:rank invariant}The erosion distance between $\Ffunc,\Gfunc:(\Int,\leq)\rightarrow \Zop$ is equal to
\begin{equation}\label{eq:rank invariant}
    \dint(\Ffunc,\Gfunc)=\sup_{n\in \Zplus}\dhaus\left(\Ffunc^{-1}[0,n],\Gfunc^{-1}[0,n]\right). 
\end{equation}
\end{theorem}

\begin{figure}
    \centering
    \includegraphics[width=0.6\textwidth]{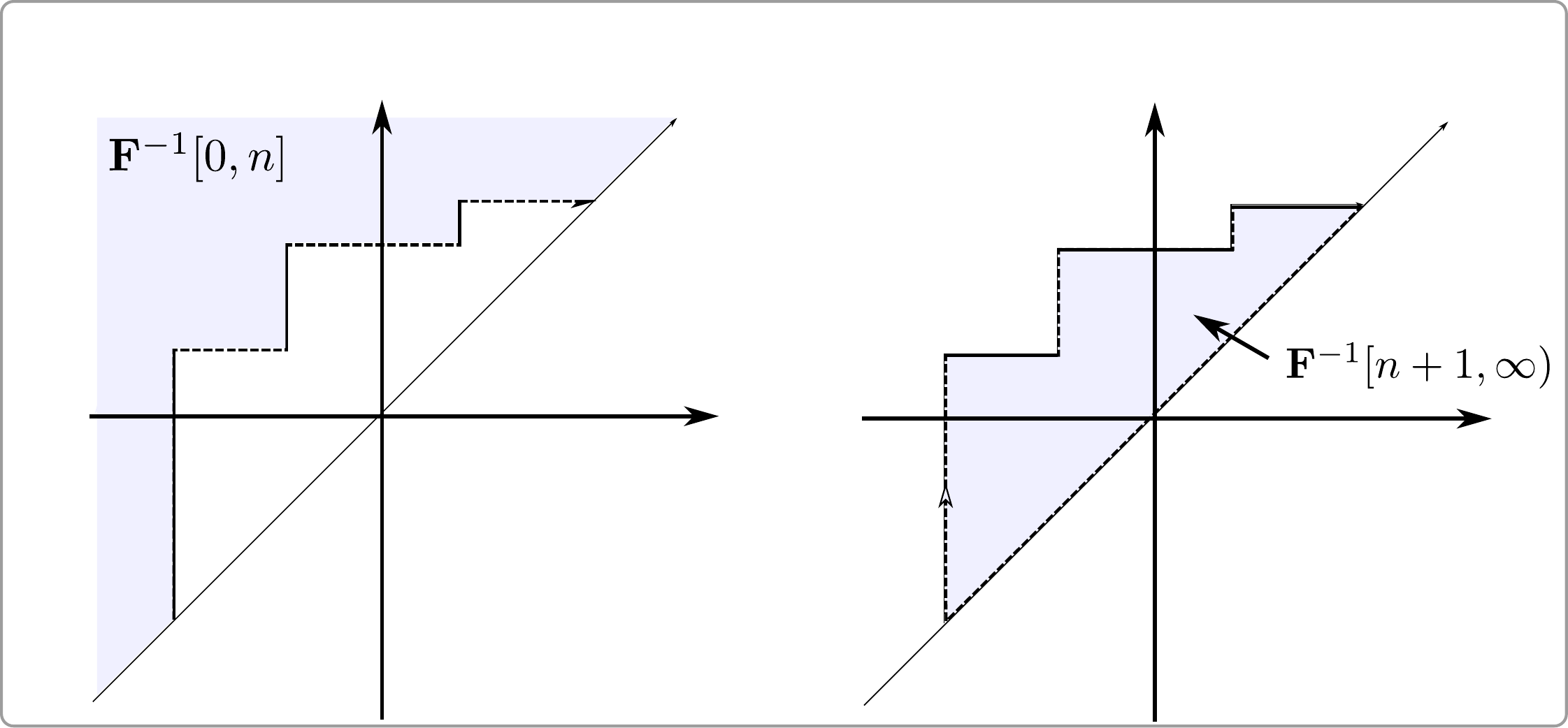}
    \caption{Illustration of $\Ffunc^{-1}[0,n]$ and  $\Ffunc^{-1}[n+1,\infty)$ for the case when $\Ffunc$ is the rank function of a generic constructible persistence module (Defn.~\ref{dfn:persistence module} and \ref{dfn:rank invariant}). The boundary line between $\Ffunc^{-1}[0,n]$ and $\Ffunc^{-1}[n+1,\infty)$ is a visual representation of the $(n+1)$-th persistence landscape. }
    \label{fig:graded_rank_inv}
\end{figure}

\begin{remark}\label{rem:erosion distance}
\begin{enumerate}[label=(\roman*),topsep=0pt,itemsep=-1ex,partopsep=1ex,parsep=1ex]
    \item If $\Ffunc$ is the \emph{rank function} of a persistence module $M:\R\rightarrow \vect$ (Defns.~\ref{dfn:persistence module} and \ref{dfn:rank invariant}), then $\Ffunc^{-1}[n+1,\infty)=\Int\setminus \Ffunc^{-1}[0,n]$ is the support of the $(n+1)$-th \emph{graded rank function} of $M$ \cite[Defn.~4.2]{betthauser2019graded} (Fig.~\ref{fig:graded_rank_inv}). In light of this, graded rank functions (or \emph{persistence landscapes} \cite{bubenik2015statistical}) are special instances of the join-representation (of rank functions) described in Prop.~\ref{prop:alg-dec}. Stability properties of graded rank functions have been discussed in \cite{betthauser2019graded}. \label{item:erosion distance1}
    \item The interleaving distance between persistence modules $M,N:\R^d\rightarrow \vect$ (Defn.~\ref{def:interleaving distance between persistence modules}) is known to be bounded from below by the erosion distance between the rank functions of $M$ and $N$; see \cite[Thm.~17]{bubenik2015statistical} \cite[Thm.~8.2]{patel2018generalized} for $d=1$ and \cite[Thm.~6.2]{kim2020spatiotemporal} for arbitrary $d$.
    \item An optimal algorithm for computing $\dint(\Ffunc,\Gfunc)$ in Eqn.~(\ref{eq:rank invariant}) and more general results are given in \cite[Sec.~5]{kim2020spatiotemporal}.
\end{enumerate}
\end{remark}

\subsection{Generalization of the Gromov-Hausdorff distance}\label{sec:tripod}

In this section we show the following:

\begin{enumerate}[label=(\roman*),topsep=0pt,itemsep=-1ex,partopsep=1ex,parsep=1ex]
    \item The Gromov-Hausdorff distance $\dt$ (between metric spaces \cite{burago2001course} or between $\R$-indexed simplicial filtrations \cite{memoli2017distance,memoli2019quantitative})   can be incorporated into the framework of interleaving distances of Sec.~\ref{sec:posets with a flow}; Thm. \ref{thm:tripod is interleaving}.\label{item:Omnepresence contribution 1} This leads to the next item.
    \item We obtain a far reaching generalization of $\dt$, which will be also denoted by $\dt$ (Defn.~\ref{def:generalized tripod distance}); although the instance of $\dt$ described in Defn.~\ref{def:generalized tripod distance} is a distance between simplicial filtrations over a poset with a flow, there is a precise sense in which it generalizes all the other Gromov-Hausdorff distances mentioned in this paper  (Rmk.~\ref{rem:ubiquity of dgh}). 
    \item This generalized $\dgh$ inherits a universal property satisfied by the original Gromov-Hausdorff distance (Thm.~\ref{thm:universality}),  of which the celebrated Vietoris-Rips filtration stability theorem \cite{chazal2009gromov,chazal2014persistence} becomes a consequence   (Thm.~\ref{thm:H_k-lower-bound} and Rmk.~\ref{rem:H_k lower bound} \ref{item:H_k lower bound_VR stability generalization}).
    \item Using Thm.~\ref{thm:di is geodesic} we show that the generalized $\dt$ is also geodesic. Interestingly, our construction of geodesic paths generalizes a known one between compact metric spaces \cite{chowdhury2018explicit} in a precise sense; see Rmk.~\ref{rem:geodesic path generalization}.
\end{enumerate}

Some basic properties of the generalized $\dgh$ are deferred to the appendix.

 Throughout this section, $X$, $Y$, and $Z$ will denote nonempty finite sets. By $\Simp(X)$, we denote the collection of \emph{abstract simplicial complexes} \cite{munkres2018elements} over a vertex set $A\subset X$ ordered by inclusion.  $\Simp(X)$ is a complete lattice whose joins are unions and meets are intersections. By $\pow(X)$, we denote the the collection of nonempty subsets of $X$ ordered by inclusion.

\begin{definition}[Simplexization]\label{def:simplexization} 
    Let us define the \pe{} $s:\pow(X)\rightarrow \Simp(X)$ as $\sigma \mapsto \pow(\sigma)$. In words, nonempty $\sigma\subset X$ is sent to the simplicial complex consisting solely of the ($\abs{\sigma}-1$)-simplex $\sigma$. 
\end{definition}

\begin{remark}\label{rmk:simp}
Let $\simp(X)$ be the image of the map $s$. The collection of all join-irreducible elements of $\Simp(X)$ equals $\simp(X)$. Hence, $\simp(X)$ is join-dense in $\Simp(X)$ (Rmk.\ref{rem:smallest join-dense}).
\end{remark} 

A \pe{} $\Ffunc_X:(\PP,\leq)\rightarrow \Simp(X)$ is said to be an ($\PP$-indexed simplicial) \textbf{filtration} (over $X$). The filtration is called \textbf{full} if there exists $p\in \PP$ such that $X\in \Ffunc_X(p)$. 

Assume that $(\PP,\leq)=(\R,\leq)$. For any $\sigma\in \pow(X)$, the \textbf{birth time of $\sigma$} is defined as $$b_{\Ffunc_X}(\sigma):=\inf\left\{t\in \R:\sigma\in \Ffunc_X(t)\right\}.$$ 
If $\sigma$ does not belong to $\Ffunc_X(t)$ for any $t\in \R$, then $b_{\Ffunc_X}(\sigma)$ is defined to be $\infty$.

\medskip
We review the Gromov-Hausdorff distance between $\R$-indexed filtrations (introduced by M\'emoli \cite{memoli2017distance} and further studied by  M\'emoli and Okutan \cite{memoli2019quantitative}):

\begin{definition}[{\cite[p.4]{memoli2017distance}}]\label{def:tripod distance} Given any $\Ffunc_X:(\R,\leq)\rightarrow \Simp(X)$ and $\Gfunc_Y:(\R,\leq)\rightarrow \Simp(Y)$, the \textbf{Gromov-Hausdorff distance} between $\Ffunc_X$ and $\Gfunc_Y$ is defined by
    \[ \dt(\Ffunc_X,\Gfunc_Y):=\frac{1}{2}\min_R \max_{\sigma\in \pow(Z)} \abs{b_{\Ffunc_X}(\varphi_X(\sigma))-b_{\Gfunc_Y}(\varphi_Y(\sigma))},
    \]
where the minimum is taken over all tripods $\tripod$.
\end{definition}

At first glance, this distance may not appear to be related to an interleaving type distance between poset maps. However, we will see that such a relation exists.

Given a surjective map $\varphi_X:Z\twoheadrightarrow X$ and a simplicial filtration $\Ffunc_X:(\PP,\leq)\rightarrow \Simp(X)$, we define the \emph{pullback} $\varphi_X^{\ast}\Ffunc_X$ of $\Ffunc_X$ via $\varphi_X$ as the filtration $(\PP,\leq)\rightarrow \Simp(Z)$ as follows. The filtration $\varphi_X^{\ast}\Ffunc_X$ sends each $p\in \PP$ to the smallest simplicial complex $K\in \Simp(Z)$ that contains \emph{all} the preimages $\varphi_X^{-1}(\tau)$ for $\tau\in \Ffunc_X(p)$.
\[\xymatrix{
 & \Simp(Z) \\
 \mathcal{P} \ar @{.>}[ur]^{\varphi_X^* {\Ffunc_X}} \ar[r]_{\Ffunc_X} & \Simp(X) \ar[u]_{\varphi_X^*} }
\]
Let $\dint^{\Simp(Z)}$ be the interleaving distance on $[\R,\Simp(Z)]$. 
By Thm.~\ref{thm:general-metric-dec} we can reformulate Defn. \ref{def:tripod distance} as follows.

\begin{theorem}\label{thm:tripod is interleaving}Given any filtrations $\Ffunc_X:(\R,\leq)\rightarrow \Simp(X)$ and $\Gfunc_Y:(\R,\leq)\rightarrow \Simp(Y)$,
\[\dt(\Ffunc_X,\Gfunc_Y)=\frac{1}{2}\min_R \dint^{\Simp(Z)}(\varphi_X^\ast\Ffunc_X,\varphi_Y^\ast\Gfunc_Y),
\] where the minimum is taken over all tripods $\tripod$.
\end{theorem}

\begin{proof} Let $\dhaus$ be the Hausdorff distance in $\R$ (Defn. \ref{def:Hausdorff distance}). We have:
\begin{align*}
    \dint^{\Simp(Z)}(\varphi_X^\ast\Ffunc_X,\varphi_Y^\ast\Gfunc_Y)&=\max_{K\in \simp(Z)}\dhaus\left((\varphi_X^\ast\Ffunc_X)^{-1}(K^\uparrow),(\varphi_Y^\ast\Gfunc_Y)^{-1}(K^\uparrow)\right)&\mbox{by  Thm.~\ref{thm:general-metric-dec} and Rmk.~\ref{rmk:simp}}\\
    &\stackrel{(\ast)}{=}\max_{\sigma\in \pow(Z)}\dhaus\big([b_{\Ffunc_X}(\varphi_X(\sigma)),\infty),[b_{\Gfunc_Y}(\varphi_Y(\sigma)),\infty)\big)&\mbox{see below}
    \\ &=\max_{\sigma\in \pow(Z)}\abs{b_{\Ffunc_X}(\varphi_X(\sigma))-b_{\Gfunc_Y}(\varphi_Y(\sigma))},
\end{align*}
where $(\ast)$ follows from the bijection $s:\pow(X)\rightarrow \simp(X)$ (Defn. \ref{def:simplexization} and Rmk. \ref{rmk:simp}). The desired equality directly follows.
\end{proof}

In the rest of this section we fix a poset with a flow $\PPPP$. In  light of Thm.~\ref{thm:general-metric-dec} and \ref{thm:tripod is interleaving}, we obtain the following generalization of $\dt$ which we still denote by the same symbol:

\begin{definition}\label{def:generalized tripod distance} Given any filtrations $\Ffunc_X:\PPP\rightarrow \Simp(X)$ and $\Gfunc_Y:\PPP\rightarrow \Simp(Y
)$, the \textbf{generalized Gromov-Hausdorff distance} between them is defined by:
\[\dt(\Ffunc_X,\Gfunc_Y):=\frac{1}{2}\min_{R}\dint^{\Simp(Z)}\left(\varphi_X^{\ast}\Ffunc_X,\varphi_Y^\ast\Gfunc_Y\right)=\frac{1}{2}\min_R\max_{\sigma\in \pow(Z)}
d_{\widehat{\Omega}}\left((\varphi_X^{\ast}\Ffunc_X)^{-1}(\sigma^\uparrow),(\varphi_Y^{\ast}\Gfunc_Y)^{-1}(\sigma^\uparrow)\right), 
\]
where the minimum is taken over all tripods $\tripod$.

\end{definition}

That $\dt$ above is an extended pseudometric is shown in the appendix  (Prop.~\ref{prop:dt is an extended pseudometric}). We remark that (1) a sufficient condition for $\dt(\Ffunc_X,\Gfunc_Y)$ to be finite is that $\Ffunc_X$ and $\Gfunc_Y$ are full, and $\domhat$ is finite for every pair of upper sets in $\PP$. (2) When $\PP=\R^n$, $d_{\widehat{\Omega}}$ above boils down to the Hausdorff distance between upper sets in $\R^n$. 

\begin{remark}\label{rem:ubiquity of dgh}
All instances of $\dgh$ mentioned in this paper can be viewed as special instances of $\dt$ in Defn.~\ref{def:generalized tripod distance}; see Prop.~\ref{prop:omnipresence of dt} in the appendix for the precise statement.
\end{remark}

\paragraph{Universal property of $\dgh$}. Let $\mathrm{Met}$ be the space of finite pseudometric spaces. It is known that the Gromov-Hausdorff distance $\dgh^{\mathrm{Met}}$ between finite pseudometric spaces is the \emph{largest} distance $D$ 
satisfying the following two conditions: (i) If there is a surjection $\varphi:(Z,d_Z)\twoheadrightarrow (X,d_X)$ such that $d_Z(z,z')=d_X(\varphi(z),\varphi(z'))$ for all $z,z'\in Z$, then $D((X,d_X),(Z,d_Z))=0$. (ii) For any two metrics $d_1,d_2$ on a set $X$, we have: \[2\cdot D((X,d_1),(X,d_2))\leq \max_{x,x'\in X}\abs{d_1(x,x')-d_2(x,x')}.\] (see \cite[Theorem F, p.11]{scoccola2020locally} for a more general statement). We extend this universal property to:

\begin{theorem}[Universality]\label{thm:universality} Let $\dgh$ the one in Defn.~\ref{def:generalized tripod distance}. Let $D$ be any metric on $\PP$-indexed finite simplicial filtrations with the following properties.
\begin{enumerate}[label=(\roman*),topsep=0pt,itemsep=-1ex,partopsep=1ex,parsep=1ex]
    \item Given any $\Ffunc_X:\PPP\rightarrow \Simp(X)$ and a surjection $\varphi_X:Z\twoheadrightarrow X$, we have $D(\Ffunc_X,\varphi_X^\ast\Ffunc_X)=0$.\label{item:universal property condition 1}
    \item Given any $\Ffunc_X,\Gfunc_X:\PPP\rightarrow \Simp(X)$, we have $2\cdot D(\Ffunc_X,\Gfunc_X)\leq \dint^{\Simp(X)}(\Ffunc_X,\Gfunc_X)$.\label{item:universal property condition 2}
\end{enumerate}
Then, $D\leq \dgh$.
\end{theorem}
In this theorem, by restricting $\dgh$ to the \emph{Vietoris-Rips filtrations} of finite pseudo metric spaces, this proposition boils down to the aforementioned universal property of $\dgh^{\mathrm{Met}}$; this is a consequence of the equality in Prop.~\ref{prop:omnipresence of dt} \ref{item:dgh=dt} of the appendix. 

Our proof of Thm.~\ref{thm:universality} is similar to the one in \cite{bauer2020reeb}.
\begin{proof}[Proof of Thm.~\ref{thm:universality}]
Given any two filtrations $\Ffunc_X:\PPP\rightarrow \Simp(X)$ and $\Gfunc_Y:\PPP\rightarrow \Simp(Y)$, pick any tripod $\tripod$. Then, by invoking the two conditions in order, we have:
\[D(\Ffunc_X,\Gfunc_Y)= D(\varphi_X^\ast\Ffunc_X,\varphi_Y^\ast \Gfunc_Y)\leq \dint^{\Simp(Z)}(\varphi_X^\ast\Ffunc_X,\varphi_Y^\ast \Gfunc_Y).\]
Since $R$ is arbitrary, we have $D(\Ffunc_X,\Gfunc_Y)\leq 2\cdot\dgh(\Ffunc_X,\Gfunc_Y)$.
\end{proof}

Now we generalize the fact that $\dgh^{\mathrm{Met}}$ is geodesic \cite{chowdhury2018explicit,memoli2017distance}, \cite[Section 6.8]{scoccola2020locally}: 
\begin{theorem}\label{thm:dt is geodesic}
The generalized Gromov-Hausdorff distance is geodesic.

\end{theorem}
\begin{proof}
Let $\Ffunc_X:\PPP\to  \Simp(X)$ and $\Ffunc_Y:\PPP\to\Simp(Y)$ be any two filtrations with $\dt(\Ffunc_X,\Ffunc_Y)=:\rho<\infty$. Let us take any tripod minimizer $\tripod$
of $\dt(\Ffunc_X,\Ffunc_Y)$.
The existence of $R$ is guaranteed by the fact that $X$ and $Y$ are finite. Note that $\dt(\Ffunc_X,\varphi_X^* \Ffunc_X)=0$ by taking the tripod  $X\xtwoheadleftarrow{\varphi_X} Z\xtwoheadrightarrow{\mathbf{id}_Z}Z$. Similarly, $\dt(\Ffunc_Y,\varphi_Y^* \Ffunc_Y)=0$. Hence, $2\rho=\dt(\varphi_X^* \Ffunc_X,\varphi_Y^* \Ffunc_Y)$ and it suffices to construct a geodesic path between $\varphi_X^*\Ffunc_X$ and $\varphi_Y^* \Ffunc_Y$ with respect to $d_{\mathrm{I}}^{\Simp(Z)}$. Invoking that $\Simp(Z)$ is a complete lattice, a geodesic path between $\varphi_X^* \Ffunc_X$ and $\varphi_Y^* \Ffunc_Y$ exists in $[\PPP,\Simp(Z)]$, as desired.
\end{proof}

By $\Simp$ we denote the category of abstract simplicial complexes and simplicial maps. By $\Simp^{\PP}$, we denote the category of $\PP$-indexed simplicial filtrations (i.e.~functors $\PP\rightarrow \Simp$) and natural transformations between them.

\begin{remark}[Generalization]\label{rem:geodesic path generalization}
The construction of a geodesic path in the proof above 
generalizes two  constructions of geodesic paths existing in the literature: A geodesic path $g:[0,1]\rightarrow (\mathrm{Met},\dgh)$ between any $(X,d_X)$ and $(Y,d_Y)$ was constructed in \cite{chowdhury2018explicit}. The corresponding path $\VR(g):[0,1]\rightarrow (\Simp^{\R},\dt)$ coincides with the geodesic between $\VR(X,d_X)$ and $\VR(Y,d_Y)$ given in \cite{memoli2017distance}. Both are special cases of the construction described in the proof above.
\end{remark}

For $k\in \Zplus$, $\Hrm_k$ will denote the $k$-th simplicial homology with coefficients in a field $\F$ \cite{munkres2018elements}. 
Given an arbitrary category $\C$, the interleaving distance $\dint^{\C}$ between two functors $\PPP\rightarrow \C$ is recalled in Defn.~\ref{def:interleaving distance between persistence modules} of the appendix. Recall that $\vect$ denotes the category of finite dimensional vector spaces over a field $\F$. We have:
\begin{theorem} 
\label{thm:H_k-lower-bound}
Given any two filtrations $\Ffunc_X:\PPP\rightarrow \Simp(X)$ and $\Gfunc_Y:\PPP\rightarrow \Simp(Y
)$, for any $k\in \Zplus$, we have: 
\begin{equation}\label{eq:PH interleaving}
    \dint^{\vect}(\Hrm_k(\Ffunc_X),\Hrm_k(\Gfunc_Y))\leq 2\ \dt(\Ffunc_X,\Gfunc_Y).
\end{equation}
\end{theorem}

\begin{proof} Consider the metric $D:=\frac{1}{2}\cdot\dint^\vect(\Hrm_k(-),\Hrm_k(-))$ on $\Simp^{\PP}$. Quillen’s Theorem A \cite{quillen1973higher} implies that, given any surjection $\varphi:Z \twoheadrightarrow X$, we have the homotopy equivalence $\Ffunc_X\simeq\varphi^*\Ffunc_X$. This implies $\Hrm_k(\varphi^{\ast}\Ffunc_X)\cong\Hrm_k(\Ffunc_X)$ and in turn that $D$ satisfies condition \ref{item:universal property condition 1} in Thm.~\ref{thm:universality}.  Functoriality of $\Hrm_k$ guarantees condition \ref{item:universal property condition 2} in Thm.~\ref{thm:universality}. Now the claim follows from Thm.~\ref{thm:universality}.  
\end{proof}

\begin{remark}\label{rem:H_k lower bound}
\begin{enumerate}[label=(\roman*),topsep=0pt,itemsep=-1ex,partopsep=1ex,parsep=1ex]
    \item The theorem above subsumes \cite[Thm.~4.2]{memoli2017distance} which addresses the case $\PP=\R$. Our proof relying on the universality of $\dgh$ gives an alternative proof of \cite[Thm.~4.2]{memoli2017distance}.
    \item \label{item:H_k lower bound_VR stability generalization} When $\Ffunc_X$ and $\Gfunc_Y$ are the (spatiotemporal) Vietoris-Rips filtrations of (dynamic) metric spaces on $X$ and $Y$, the inequality in (\ref{eq:PH interleaving}) coincides with the (spatiotemporal) Vietories-Rips filtration stability theorems \cite{chazal2009gromov,chazal2014persistence,kim2020spatiotemporal}. This is a corollary of Prop.~\ref{prop:omnipresence of dt} \ref{item:dgh=dt} and \ref{item:dgh=dt DMSs}.
    \item The distance $\dt(-,-)$ is more discriminative than $\max_{k\in \Zplus}\dint^{\vect}(\Hrm_k(-),\Hrm_k(-))$ as can be seen through the following example.
\end{enumerate}
\end{remark}

\begin{example}\label{ex:filtrations}
Let $X=\{x_1\}$ and let $Y=\{y_1,y_2\}$. Define $\Ffunc_X:\R\rightarrow \Simp(X)$ and $\Ffunc_Y:\R\rightarrow \Simp(Y)$  by $\Ffunc_X(t)=\begin{cases}
      x_1,&t\in [0,\infty)\\ \emptyset,&\mbox{otherwise,}
\end{cases}$ 
and $\Ffunc_Y(t):=\begin{cases}
      \left\{\{y_1\}\right\},&t\in [0,1)\\ \left\{\{y_1\},\{y_2\},\{y_1,y_2\}\right\},&t\in [1,\infty)\\\emptyset,&\mbox{otherwise},
\end{cases}$ respectively. Note that $\displaystyle\max_{k\in \Zplus}\dint^{\vect}\left(\Hrm_k(\Ffunc_X),\Hrm_k(\Ffunc_{Y})\right)=0$, but $\dt\left(\Ffunc_X,\Ffunc_Y\right)=1$. Noting that $\Ffunc_X$ and $\Ffunc_Y$ are homotopy equivalent, this example shows that $\dt$ between homotopy equivalent filtrations can be strictly positive. 

We further remark that $\dgh$ between $\R$-indexed simplicial filtrations can be arbitrarily larger than the homotopy interleaving distance \cite{blumberg2017universality}.
e.g. when comparing the Vietoris-Rips filtrations of metric spaces that are (almost) homotopy equivalent but are far from being isometric, such as a pair of a circle and a circle with long flares, cf. \cite[Fig.~1]{memoli2019quantitative}.
\end{example}

In the appendix we provide other upper and lower bounds for $\dt$; see Sec.~\ref{sec:on the tripod distance}.

\subsection{Hierarchical clusterings over posets} \label{sec:Interleaving distance between hierarchical clusterings}

Let $X$ be a nonempty finite set. A \textbf{$\poset$-indexed hierarchical clustering} (over $X$) is any \pe{} $\poset\rightarrow \subpart(X)$. Standard examples include the case of $\PP=\R^n$ with the product order, often referred to as \emph{hierarchical clustering} (when $n=1$) or \emph{multiparameter hierarchical clustering} (when $n\geq 2$) \cite{cai2020elder,carlsson2010characterization,carlsson2010multiparameter,kim2020spatiotemporal,rolle2020stable,smith2016hierarchical}. From Rmk.~\ref{rem:about subpart(X)} \ref{item:about subpart(X)1}, \ref{item:about subpart(X)2}, and Prop.~\ref{prop:join representation} \ref{item:join representation1} recall that $\subpart(X)$ is $\bigvee$-irreducibly generated and the join-irreducible elements of $\subpart(X)$ are either singleton or doubleton blocks of $X$. 

Recall that given any point $q$ in a poset $\QQ$,  $q^\uparrow:=\left\{r\in \PP: q\leq r\right\}$. Invoking Thm.~\ref{thm:general-metric-dec}, we have:

\begin{theorem}\label{thm:hierarchical clustering}
 Let $\PPPP$ be a poset with a flow. For any $\Ffunc,\Gfunc:\PPPP \rightarrow \subpart(X)$,

\begin{equation}\label{eq:interleaving beween multiclustering}
\dint\left(\Ffunc,\Gfunc\right)=\max_{x,x'\in X}d_{\widehat{\Omega}}\left(\Ffunc^{-1}\left(\left\{x,x'\right\}^\uparrow\right),\Gfunc^{-1}\left(\left\{x,x'\right\}^\uparrow\right)\right).    
\end{equation}
(Note: when $x=x'$, the set $\left\{x,x'\right\}$ is the singleton $\left\{x\right\}$). 
\end{theorem}
If $\PP=\R^n$ (resp. $\Int$), the equality above implies that computing the interleaving distance between multiparameter hierarchical clusterings into computing the Hausdorff distance between upper sets of $\R^n$ (resp. $\Int$) a finite number of times. In the next section, we discuss the computational complexity of the RHS when $\PP=\Int$ (Thm.~\ref{thm:df complexity}).  Also we will discuss the comparison of two $\PPP$-indexed hierarchical clusterings \emph{over different underlying sets}.

\subsection{Computing distances between formigrams }\label{sec:Structural theorem for the formigram interleaving distance}\label{ssec:metric-dec}

In this section we elucidate the structure of the two distances $\df$ and $\dgh$ between formigrams introduced in Sec.~\ref{sec:formigrams} (Thms.~\ref{thm:metric-dec-formi} and \ref{thm:structmax}). Thereby we find equivalences between several known metrics for comparing hierarchical clusterings (Rmk.~\ref{rem:relating to GH} \ref{item:relating to GH1}). Also, we clarify the computational costs of $\df$, $\dgh$ and other related metrics (Rmk.~\ref{rem:relating to GH} \ref{item:relating to GH2} and Thm.~\ref{thm:df complexity}). $\dgh$ between formigrams will be extended to a distance between any poset-indexed hierarchical clusterings (Defn.~\ref{def:dGH between HC over P}).

\paragraph{Formigrams can be viewed as \pe{}s.}

The poset $\Int$ in Ex.~\ref{ex:upset of int} is isomorphic to the poset of nonempty closed intervals of $\R$ whose partial order is inclusion. 
Each $(a,b)\in \Int$ will be identified with the closed interval $[a,b]\subset \R$. Let us fix a nonempty finite set $X$. 
A \pe{} $(\Int,\leq)\to(\subpart(X),\leq)$ will be simply denoted by $\Int\to\subpart(X)$. Any formigram $\theta$ over $X$ induces a map $\Int\to\subpart(X)$:

\begin{definition}
\label{dfn:theta hat}
For $\theta\in \Formi(X)$, define $\widehat{\theta}:\Int\rightarrow\subpart(X)$ as $I\mapsto \bigvee_{s\in I}\theta(s).$ 
\end{definition}
See Fig.~\ref{fig:precosheaf} (A) and (B) for an illustrative example of Defn.~\ref{dfn:theta hat}.

Recall the flow $\Omega$ on $\Int$ in Eqn.~(\ref{eq:Int flow}). Defn.~\ref{def:Interleavings of persistent elements} can be specialized as follows:
Given $\alpha,\alpha':\Int\to\subpart(X)$ 
\[\dfhat\left(\alpha,\alpha'\right):=\inf\left\{\e\in [0,\infty): \alpha,\alpha'\text{ are }\e\text{-interleaved w.r.t. $-\cdot\Omega$}\right\}.
\]
It is not difficult to check that $\df$ in Defn.~\ref{def:intrinsic} coincides with $\dfhat$:

\begin{proposition}[{\cite[Defn.~4.11]{kim2017stable}}]
\label{thm:metric-embed}
For any $\theta,\theta'\in \Formi(X)$, we have:
$\df\left(\theta,\theta'\right)=\dfhat\left(\widehat{\theta},\widehat{\theta'}\right).$
\end{proposition}

In view of Defn. \ref{dfn:theta hat} and Prop. \ref{thm:metric-embed}, in what follows, any formigram over $X$ will be identified with a  poset map $\Int\rightarrow \subpart(X)$ and $\df$ will be identified with $\dfhat$.

\paragraph{$\df$ via interleaving by parts.}
Let $\dhaus$ be the \emph{Hausdorff distance} (Defn.~\ref{def:Hausdorff distance}) in $(\Int,\norm{-}_\infty)$ (Ex.~\ref{ex:upset of int}). 
As a corollary of Thm.~\ref{thm:hierarchical clustering}, we have:
\begin{theorem}\label{thm:metric-dec-formi}
For any two formigrams $\theta$ and $\theta'$ over $X$,
\begin{equation}\label{eq:structure thm}
     \df\left(\theta,\theta'\right)=\max_{x,x'\in X}\dhaus\left(\theta^{-1}\left(\left\{x,x'\right\}^{\uparrow}\right),\theta'^{-1}\left(\left\{x,x'\right\}^{\uparrow}\right)\right).
\end{equation}
\end{theorem}

We utilize Thm.~\ref{thm:metric-dec-formi} for elucidating both the  computational complexity of $\df$  (Prop. \ref{prop:complexity of df}). 
Inspired by Thm. \ref{thm:metric-dec-formi} we define: 
\begin{definition}\label{dfn:upcode}
 For $\theta\in \Formi(X)$, we call $B(\theta):=\left\{\theta^{-1}\left(\left\{x,x'\right\}^{\uparrow}\right)\right\}_{x,x'\in X}$
the \textbf{\upcode{}} of $\theta$ (see Fig.~\ref{fig:precosheaf} for an example). 
\end{definition}

See Fig. \ref{fig:e-shifted-formi} and \ref{fig:multiple-pre-cosheaves} for an illustrative example of 
an application of Thm. \ref{thm:metric-dec-formi}. 
\begin{figure}
    \centering
\includegraphics[width=\textwidth]{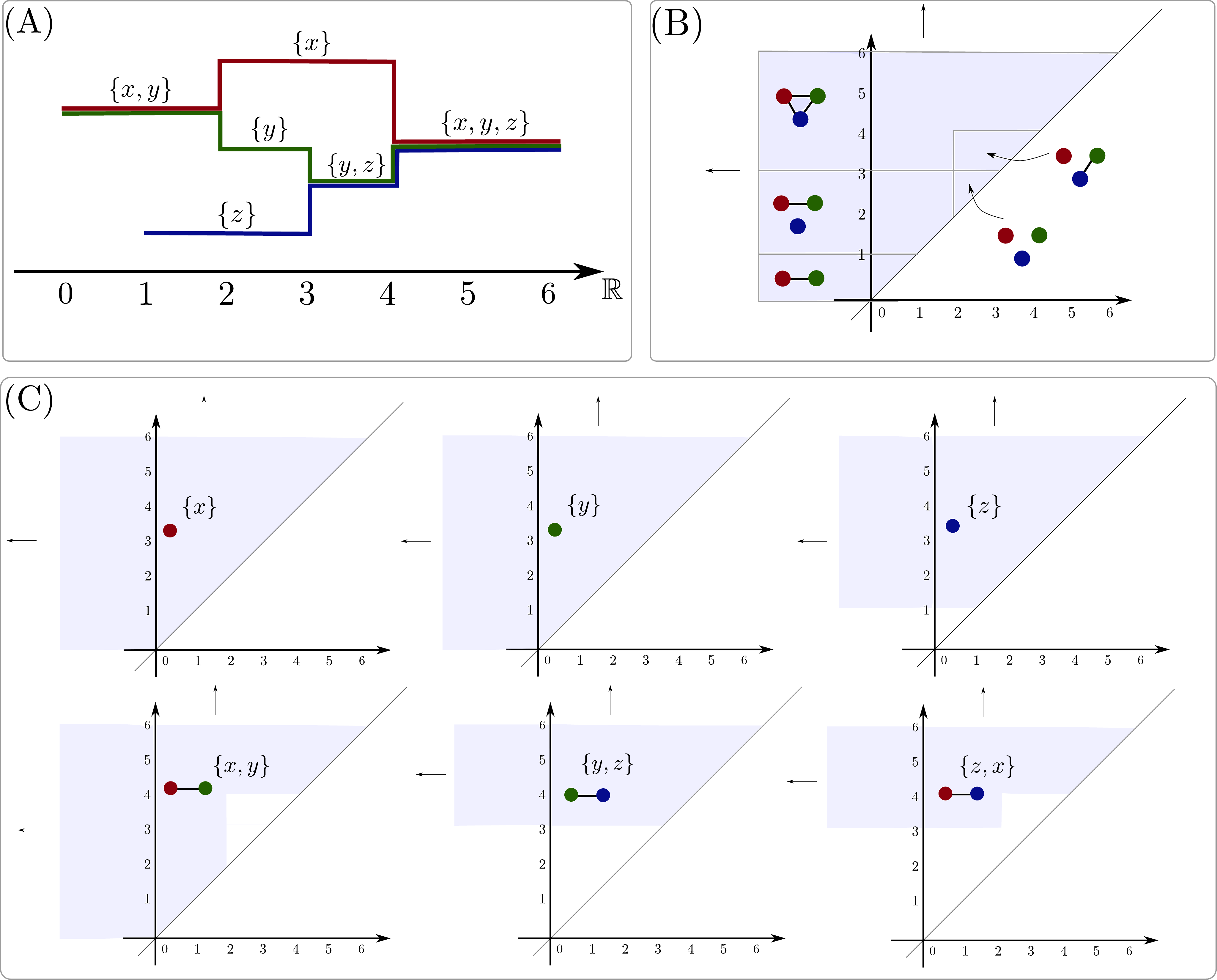}
    \caption{(A) A formigram $\theta$ over $\{x,y,z\}$ such that $\theta(t)=\emptyset$ for $t\notin [0,6]$. $x,y,z$ are colored in red, green and blue,  respectively. (B) The corresponding map $\widehat{\theta}:\Int\rightarrow \subpart(X)$ (Defn. \ref{dfn:theta hat}). (C)  The \upcode{} of $\theta$ (Defn. \ref{dfn:upcode}).}
    \label{fig:precosheaf}
\end{figure}

\begin{figure}
  \centering
\includegraphics[width=1\textwidth]{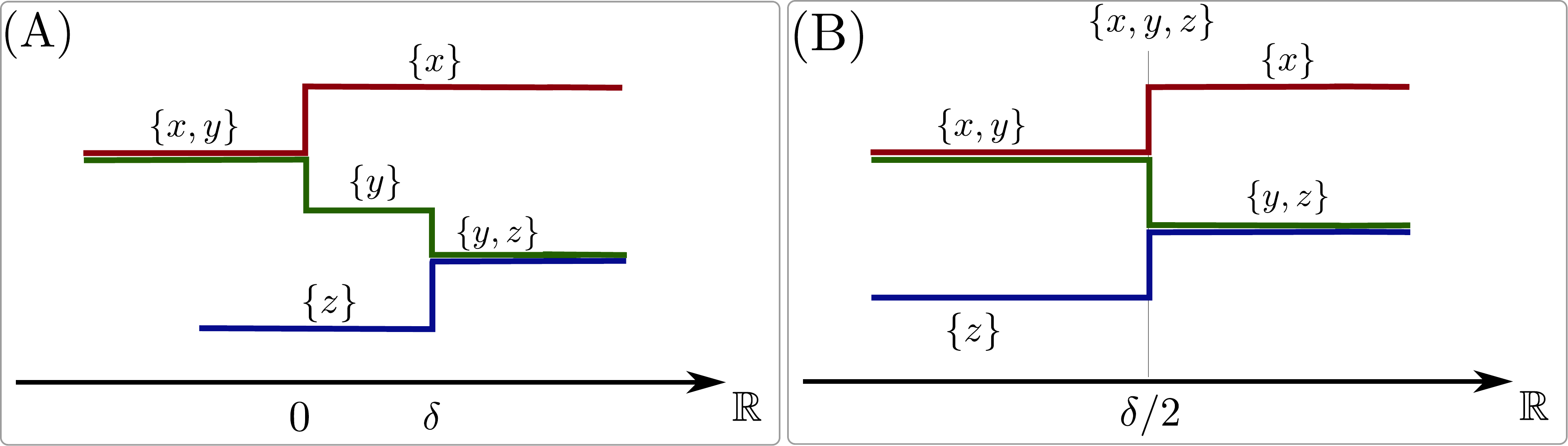}
        \caption{(A) A formigram $\theta$ over $\{x,y,z\}$. $x,y,z$ are colored in red, green and blue,  respectively. (B) The formigram $\theta':=\Sfunc_{\delta/2}(\theta)$ (cf. Defn. \ref{def:smoothing}).
The \upcode{}s of  $\theta$ and $\theta'$ are illustrated in Fig.~\ref{fig:multiple-pre-cosheaves}.} 
        \label{fig:e-shifted-formi}
\end{figure}

\begin{figure}
    \includegraphics[width=\textwidth]{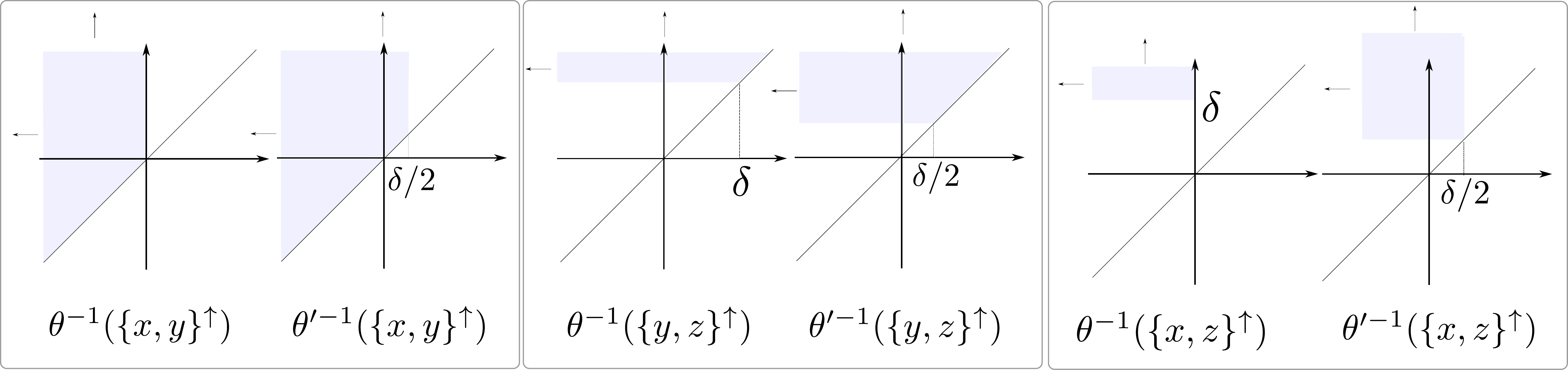}

    \caption{Consider the formigrams $\theta$ and $\theta'$ in Fig.~\ref{fig:e-shifted-formi}. Observe that $\dhaus\left(\theta^{-1}(\{x,y\}^\uparrow),\theta'^{-1}(\{x,y\}^\uparrow)\right)=\dhaus\left(\theta^{-1}(\{y,z\}^\uparrow),\theta^{-1}(\{y,z\}^\uparrow)\right)=\dhaus\left(\theta^{-1}(\{x,z\}^\uparrow),\theta'^{-1}(\{x,z\}^\uparrow)\right)=\delta/2$. Also, for any $w\in\{x,y,z\}$,  $\dhaus\left((\theta^{-1}(\{w\}^\uparrow),\theta'^{-1}(\{w\}^\uparrow)\right)=\dhaus\left(\Int,\Int\right)=0$. By Thm.~\ref{thm:metric-dec-formi} we obtain $\df(\theta,\theta')=\delta/2.$} 
     \label{fig:multiple-pre-cosheaves} 
\end{figure}

\begin{remark}\label{rem:isometry theorem for formigrams} 
\begin{enumerate}[label=(\roman*)]
    \item   Thm.~\ref{thm:metric-dec-formi} is analogous to the \emph{isometry theorem for zigzag modules} \cite{bjerkevik2016stability,botnan2018algebraic}, which says that a certain interleaving distance between $\vect$-valued zigzag modules is equal to the bottleneck distance between their \emph{block barcodes}.  
    \label{item:isometry theorem for formigrams1}

    \item  Recall the notions of dendrograms and the ultrametrics induced by dendrograms  (Defn.~\ref{dfn:dendrograms}). In Thm.~\ref{thm:metric-dec-formi}, let us assume that $\theta$ and $\theta'$ are dendrograms over $X$. Then, for each $x,x'\in X$, $\theta^{-1}\left(\left\{x,x'\right\}^{\uparrow}\right)=\left\{(a,b)\in \Int:b\in [u_{\theta}(x,x'),\infty)\right\}$ and similarly for $\theta'^{-1}\left(\left\{x,x'\right\}^{\uparrow}\right)$. 
Therefore,
\[\dhaus\left(\theta^{-1}(\{x,x'\}^{\uparrow}),\theta'^{-1}(\{x,x'\}^{\uparrow})\right)=|u_{\theta}(x,x')-u_{\theta'}(x,x')|\] 
and in turn
 $\df(\theta,\theta')=\displaystyle \max_{x,x'\in X}|u_\theta(x,x')-u_{\theta'}(x,x')|.$\label{item:isometry theorem for formigrams2} 
\end{enumerate}
\end{remark}

\paragraph{Structure of $\dintf$ and related metrics.}
We can reformulate the Gromov-Hausdorff distance between formigrams (Defn.~\ref{def:general formigram distance}) via the \upcode{}s of formigrams:

\begin{theorem}
\label{thm:structmax}
Let $\theta_X$,$\theta_Y$ be any two formigrams over $X$ and $Y$, respectively. Then, we have:
\begin{equation}\label{eq:Gromov-Hausdorff between formigrams}
        \dintf(\theta_X,\theta_Y)=\frac{1}{2}\min_R\max_{\substack{ (x,y)\in R\\ (x',y')\in R}} \dhaus\left(\theta_X^{-1}(\{x,x'\}^{\uparrow}),\theta_Y^{-1}(\{y,y'\}^{\uparrow})\right), 
\end{equation}
where the minimum is taken over all tripods $\tripod$ between $X$ and $Y$.
\end{theorem}

\begin{proof} This follows from Thm.~\ref{thm:metric-dec-formi} and the equivalence in (\ref{eq:pullback equivalence}). Details are omitted.
\end{proof}

\begin{remark}
\label{rem:relating to GH} 
\begin{enumerate}[label=(\roman*)]
    \item Thm.~\ref{thm:structmax} shows that $\dintf$ in Eqn.~(\ref{eq:Gromov-Hausdorff between formigrams}) has exactly the same structure as the distance $d_{\mathcal{Q}}$ introduced in \cite[page~69]{carlsson2010multiparameter}, and the distance $d_{\mathrm{CI}}$ in \cite[Defn.~2.14]{rolle2020stable}. The only difference is that $\dgh$, $d_{\mathcal{Q}}$ and $d_{\mathrm{CI}}$ compare respectively $\Int$-indexed hierarchical clusterings, $\R^2$-indexed hierarchical clusterings, and $\R^n$-indexed hierarchical clusterings.  \label{item:relating to GH1}
    \item  In Thm.~\ref{thm:structmax}, assume that $\theta_X$ and $\theta_Y$ are \emph{dendrograms} over $X$ and $Y$, respectively (Defn.~\ref{dfn:dendrograms}). \label{item:relating to GH2}
Then, by Rmk.~\ref{rem:isometry theorem for formigrams} \ref{item:isometry theorem for formigrams2}, we have that
\[\dintf(\theta_X,\theta_Y)=\frac{1}{2} \min_R\max_{\substack{(x,y)\in R \\ (x',y')\in R }} |u_{\theta_X}(x,x')-u_{\theta_Y}(y',y)|. \] 
Note that the RHS coincides with the Gromov-Hausdorff between (ultra)metric spaces (Defn.~\ref{def:Gromov-Hausdorff}), which is known to be NP-hard to compute \cite{schmiedl2017computational}. Therefore, computing  $\dgh$ between formigrams is also NP-hard \cite{kim2017stable}. It is not difficult to see that the previous item implies  that computing $d_{\mathcal{Q}}$ and $d_{\mathrm{CI}}$ is also NP-hard.
\end{enumerate}
\end{remark}

All  distances mentioned in Rmk.~\ref{rem:relating to GH} can now be seen as specializations of the following (cf. Thm.~\ref{thm:hierarchical clustering}): 
\begin{definition}\label{def:dGH between HC over P}Let $\PPPP$ be a poset with a flow. Given any $\theta_X:\PPPP\rightarrow\subpart(X)$ and $\theta_Y:\PPPP\rightarrow \subpart(Y)$, the \textbf{Gromov-Hausdorff} distance between them is defined by:
    \[\dgh(\theta_X,\theta_Y)=\frac{1}{2}\min_R\dint(\varphi_X^\ast\theta_X,\varphi_Y^{\ast}\theta_Y)=\frac{1}{2}\min_R\max_{\substack{ (x,y)\in R\\ (x',y')\in R}} \dhaus\left(\theta_X^{-1}(\{x,x'\}^{\uparrow}),\theta_Y^{-1}(\{y,y'\}^{\uparrow})\right),\]
where the minimum is taken over all tripods $\tripod$ between $X$ and $Y$.
\end{definition}

\subsubsection{Computational cost for the calculation of $\df$}\label{sec:complexity}

In this section we elucidate the computational complexity of the formigram interleaving distance $\df$ (Thm.~\ref{thm:df complexity}). We do this by clarifying the complexity of each preliminary step for the calculation of $\df$. 
Many ideas in this section can be adapted to the case of the interleaving distance between multiparameter hierarchical clustering (cf. Eqn.~(\ref{eq:interleaving beween multiclustering}). 

Let $\theta$ be a formigram over $X$ (Defn.~\ref{dfn:formi}). Let $n:=|X|$ and $m:=\crit(\theta)$. 
\begin{proposition}\label{prop:cosheaf computation}
 Computing the corresponding poset map $\widehat{\theta}:\Int\rightarrow \subpart(X)$ (Defn.~\ref{dfn:theta hat}) requires time $O(n^2m^2)$.
\end{proposition}

\begin{proposition}\label{prop:complexity for the upcode of theta}
Given $\widehat{\theta}:\Int\rightarrow \subpart(X)$, computing the \upcode{} of $\theta$ takes time $O(n^4m^2)$ on average. 
\end{proposition}

\begin{proposition}\label{prop:complexity of df} Assume that  the \upcode{}s of two formigrams $\theta$ and $\theta'$ over $X$ are given where  $n:=|X|$, $m:=|\crit(\theta)|$ and $m':=|\crit(\theta')|$. Computing $\df(\theta,\theta')$ requires  time $O(n^2(m+m'))$.
\end{proposition}

In sum:
\begin{theorem}\label{thm:df complexity}
Given two formigrams $\theta$ and $\theta'$, computing $\df(\theta,\theta')$ requires  
$O(n^4\ell^2)$ in expectation where $\ell:=\max(m,m')$.
\end{theorem}
As we already saw in Sec.~\ref{sec:formigrams}, any formigram $\theta$ over $X$ can be visualized as a topological graph over the real line, annotated by elements in $X$. This graph is called the \textbf{underlying Reeb graph} of $\theta$. A rigorous definition is given in \cite{kim2017stable}. We can significantly reduce the complexity $O(n^4\ell^2)$ mentioned above to $O(n^2\cdot\ell^{1.5}\log \ell)$ by restricting ourselves to formigrams whose underlying Reeb graphs do not contain any loops\footnote{For example, the formigram depicted in Fig.~\ref{fig:precosheaf}~(A) contains a loop, whereas the ones in Fig.~\ref{fig:e-shifted-formi} (A) and (B) do not.}; see Thm.~\ref{thm:df complexity for trees} in the appendix. To prove this claim, we utilize a special relationship between the bottleneck distance and the Hausdorff distance on the real line which may be of independent interest (Thm.~\ref{thm:dH=dB}).

\paragraph{Proof of Prop.~\ref{prop:cosheaf computation}.}
We begin with the following lemma:

\begin{lemma}\label{lem:refinement}Let $P_1$ and $P_2$ be any two subpartitions of $X$, with $n:=|X|$. Computing $P_1\vee P_2$ requires at most time $O(n^2)$. 
\end{lemma}

\begin{proof} 
For $i=1,2$, let $X_i\subset X$ be the underlying set of $P_i$. Let us consider the undirected simple graph $G_i=(X_i,E_i)$ derived from $P_i$, where $\{x,y\}\in E_i$ if and only if $x$ and $y$ belong to the same block of $X_i$. Note that $P_1\vee P_2$ is the partition of $X_1\cup X_2$ where each block of $P_1\vee P_2$ constitutes a connected component of the graph $G_1\cup G_2=(X_1\cup X_2, E_1\cup E_2)$. Therefore, computing $P_1\vee P_2$ is equivalent to computing the connected components of $G_1\cup G_2$. One needs $O(|X_1\cup X_2|+|E_1\cup E_2|)$ in time to partition $X_1\cup X_2$ according to the components of $G_1\cup G_2$ \cite[Section 5]{erickson2019algorithms}. Since $|X_1\cup X_2|\leq |X|= n$ and $|E_1\cup E_2|\leq \binom{n}{2}$, at worst time $O(n^2)$ will be necessary.
\end{proof}

\begin{proof}[Proof of Prop.~\ref{prop:cosheaf computation}] The claim directly follows from Lem.~\ref{lem:refinement} and the observation that, in order to compute $\widehat{\theta}$, it suffices to compute $O(\binom{m}{2})=O(m^2)$ different join operations between two (sub)partitions of $X$. See Fig.~\ref{fig:cosheaf computation} (A) for an illustrative example. 
\end{proof}

\begin{figure}
    \centering
    \includegraphics[width=0.7\textwidth]{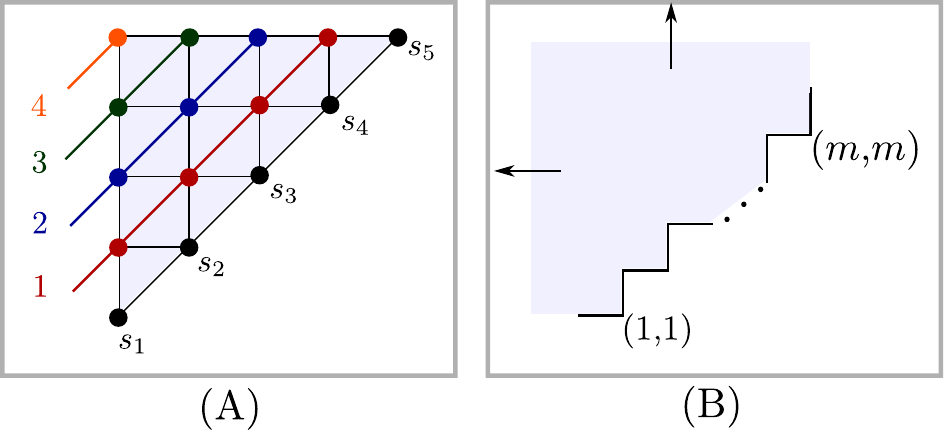}
    \caption{(A) For a given formigram $\theta$, let us assume that $\crit(\theta)=\{s_1<s_2<s_3<s_4<s_5\}$.  The main diagonal line stands for the real line via the bijection $(r,r)\in \R^2\leftrightarrow r\in \R$. We assign a subpartition of $X$ to each colored point in the grid as follows: For each point $v$ in the line $1$, assign $\theta(s_i)\vee\theta(s_{i+1})$ where $s_i$ and $s_{i+1}$ are adjacent to $v$ in the grid. For each point in line $i$, assign $P_1\vee P_2$ where $P_1$ and $P_2$ are the subpartitions assigned to two points in line $i-1$ which are adjacent to $v$. Observe that, given any $I\in\Int$ with $I\cap\crit(\theta)\neq \emptyset$, $\widehat{\theta}(I)$ is equal to the subpartition assigned to the maximal point $v=(v_1,v_2)$ in the grid $(\subset \R^{\mathrm{op}}\times\R)$ s.t. $v_1,v_2\in I$. (B) An illustration of $\theta_X^{-1}(\{x,x'\}^\uparrow)$  
    of which the number of its corner points is maximal, $2m-1$.}
    
    \label{fig:complexity illustration}   
    \label{fig:cosheaf computation}
\end{figure}

\paragraph{Proof of Prop.~\ref{prop:complexity for the upcode of theta}.}
Let $X=\{x_1,\ldots, x_n\}$, and let $P=\{C_1,\ldots, C_k\}$ be a subpartition of $X$. This $P$ can be encoded as the $n\times k$ \emph{membership matrix} $M_P=(m_{ij})$ where the $i$-th row and  the $j$-th column correspond to $x_i\in X$ and  $C_j\in P$ respectively, and 
\[m_{ij}:=\begin{cases}
      1,& \mbox{if $x_i$ belongs to $C_j$}\\
      0,& \mbox{otherwise.}
\end{cases}
\]
Note that the $(i,j)$-entry of $M_P(M_P)^t$ is $1$ iff $x_i$ and $x_j$ belong to the same block, and $0$ otherwise. \emph{Let us assume that all subpartitions of $X$ are equally likely choices for $P$}. Then, since $k$ cannot exceed $n$, computing $M_P(M_P)^t$ is expected to take \emph{at most} in time $O(n^2)$   \cite[Thm.~3.1]{o1973fast}. 

Now recall from Def. \ref{dfn:upcode} that the \upcode{} of a formigram $\theta$ over an $n$-point set consists of $\binom{n}{2}$ upper sets $\widehat{\theta}_{\{x,x'\}}$, $x,x'\in X$ of $\Int$. Assuming that $m:=|\crit(\theta)|$, by the previous argument, computing each $\widehat{\theta}_{\{x,x'\}}$ requires at most in time $O\left(n^2\binom{m}{2}\right)=O(n^2m^2)$. Therefore, we directly have Prop.~\ref{prop:complexity for the upcode of theta}.

\paragraph{Proof of Prop.~\ref{prop:complexity of df}.}
Recall from Thm.~\ref{thm:metric-dec-formi} that computing $\df$ reduces to the calculation of the Hausdorff distance between upper sets of $(\Int,\norm{-}_{\infty})$. The lemma below provides an insight into computing $\df$:

\begin{lemma}
\label{lem:lines}
If $A,B$ are upper sets of $(\Int,\norm{-}_{\infty})$, then their Hausdorff distance is given by the formula
$$d_{\mathrm{H}}(A,B)=\sup_{\ell}d_{\mathrm{H}}(A\cap \ell,B\cap \ell),$$
where $\ell$ ranges over all lines of slope $-1$ in $\R^2$.
\end{lemma}

\begin{proof}
\begin{itemize}
    \item[($\geq$)] Pick any line $\ell$ of slope $-1$. Let $\e:=d_{\mathrm{H}}(A,B)$.
Let $x=(x_1,x_2)\in A\cap \ell$.
Then, there exists $y\in B$ such that $\norm{x-y}_{\infty}\leq\e$.
Since $B$ is an upper set, $y\leq (x_1-\e,x_2+\e)\in B$.
Since $\ell$ is of slope $-1$, $(x_1-\e,x_2+\e)$ lies on the line $\ell$, and thus $(x_1-\e,x_2+\e)\in B\cap \ell$.
In the same way, one can prove that  for any $x\in B\cap \ell$, there exists a $y\in A\cap \ell$ such that $\norm{x-y}_{\infty}\leq \e$. 
    \item[($\leq$)] Assume the RHS is less or equal to $\e$.
Pick $a\in A$ and a line $\ell$ of slope $-1$ which passes through $a$.
By assumptions, there is a $b\in B\cap \ell$ such that $\norm{a-b}_{\infty}\leq \e$.
This $b$ also belongs to $B$.
By symmetry,  for all $b\in B$, there exists an $a\in A$ such that $\norm{a-b}_{\infty}\leq \e$. 
\end{itemize}
\end{proof}

\begin{proof}[Proof of Prop.~\ref{prop:complexity of df}]
For $x,x'\in X$, the upper sets 
$A_{\{x,x'\}}:=\theta_X^{-1}(\{x,x'\}^\uparrow)$ and $B_{\{x,x'\}}:=(\theta_X')^{-1}(\{x,x'\}^\uparrow)$ 
have at most $2m-1$ and $2m'-1$ corner points, respectively (see Fig.~\ref{fig:complexity illustration} (B)). By Thm.~\ref{thm:metric-dec-formi}, computing $\df(\theta,\theta')$ reduces to computing  $d_{\mathrm{H}}\left(A_{\{x,x'\}},B_{\{x,x'\}}\right)$ for all $x,x'\in X$. By Lem. \ref{lem:lines},  computing  $d_{\mathrm{H}}\left(A_{\{x,x'\}},B_{\{x,x'\}}\right)$ requires at most $O((2m-1)+(2m'-1))=O(m+m')$  computations of $d_{\mathrm{H}}\left(A_{\{x,x'\}}\cap \ell,B_{\{x,x'\}}\cap \ell\right)$ where $\ell$ are (-1)-slope-lines passing through at least one of $O(m+m')$ corner points of $A_{\{x,x'\}}$ or $B_{\{x,x'\}}$. For any $\ell$, let $p,q\in \Rop\times \R$ be the unique minimums of $A_{\{x,x'\}}\cap \ell$ and $B_{\{x,x'\}}\cap \ell$, respectively. Then, it is not difficult to check that $d_{\mathrm{H}}\left(A_{\{x,x'\}}\cap \ell,B_{\{x,x'\}}\cap \ell\right)=\norm{p-q}_{\infty}$. Hence the claim follows. 
\end{proof}

\section{Discussions}\label{sec:discussions}

Some open questions are the following.

\noindent \emph{(1) What is the relationship between $\dt$ in Defn.~\ref{def:generalized tripod distance} and the edit distance $d_{\mathrm{edit}}$  in \cite{mccleary2020edit}  between lattice-indexed simplicial filtrations}? Both distances are a generalization or a rendition of certain distances that satisfy \emph{universality}; For $\dgh$, see Thm.~\ref{thm:universality} (also \cite[Prop.6.2.21]{scoccola2020locally}). $d_{\mathrm{edit}}$ is a rendition of the edit distance on Reeb graphs which  satisfies another universal property \cite{bauer2020reeb} (and is itself inspired on $\dgh$). Currently we know that $\dgh$ and $d_{\mathrm{edit}}$ cannot be strongly equivalent; it is not difficult to find a pair 
of simplicial filtrations such that $\dgh$ vanishes, but $d_{\mathrm{edit}}$ does not.

\noindent\emph{(2) Realization of erosion geodesics.} From Thm.~\ref{thm:di is geodesic}, we know that any two poset maps $\Ffunc,\Gfunc:\Int\rightarrow \Zop$ at  finite erosion distance can be joined by a geodesic path in $[\Int,\Zop]$. Assume that $\Ffunc$ and $\Gfunc$ are the rank functions of two persistence modules $M$ and $N$, respectively. The geodesic path $g:[0,1]\rightarrow[\Int,\Zop]$ between $\Ffunc$ and $\Gfunc$ constructed  in the proof of Thm.~\ref{thm:di is geodesic} is actually \emph{not} always \emph{realizable} by a path in $\vect^\R$, i.e.~there is sometimes no continuous map $h:[0,1]\rightarrow (\vect^\R,\dero\circ\rk)$ 
such that  $\rk\circ h=g$.\footnote{\emph{Example:} For $a,b\in \R$ with $a<b$, let $\mathbb{I}[a,b):\R\rightarrow \vect$ the \emph{interval module} with support $[a,b)$ \cite{chazal2016structure}. Let $M,N:\R\rightarrow \vect$ be defined as $M=\mathbb{I}[0,10)\oplus\mathbb{I}[10,16)$ and $N=\mathbb{I}[4,14)$. Now define $\Ffunc,\Gfunc$ to be $\rk(M)$ and $\rk(N)$ respectively.}
This implies that, we do not know at this point whether or not  $\dero(\rk(-),\rk(-))$ is a geodesic distance on the space of $\R$-indexed persistence modules. We believe that studying this can potentially be useful for clarifying the relationship between the bottleneck distance and  $\dero(\rk(-),\rk(-))$.

\appendix

\section{Persistence modules and rank functions}
  Let $\vect$ be the category of finite dimensional vector spaces and linear maps over a field $\F$.
\begin{definition}\label{dfn:persistence module} Any functor $M:(\R,\leq)\rightarrow \vect$ is said to be a (standard) \textbf{persistence module}, i.e.~each $r\in \R$ is sent to a vector space $M(r)$ and each pair $r\leq s$ in $\R$ is sent to a linear map $M(r\leq s)$. In particular, for all $r\leq s\leq t$, we have:
\[M(s\leq t)\circ M(r\leq s)= M(r\leq t).\]
$M$ is called \textbf{constructible} if there exists a finite set $\{c_1,\ldots,c_n\}\subset \R$ such that (i) for $i=1,\ldots,n$, whenever $r,s\in [c_i,c_{i+1})$ with $r\leq s$, $M(r\leq s)$ is the identity map (let $c_{n+1}:=\infty$), (ii) for $r\in (-\infty,a_1)$, $M(r)=0$.

By replacing the indexing poset $(\R,\leq)$ by $(\R^d,\leq)$ for $d\geq 2$, we obtain a \textbf{multiparameter persistence module}.
\end{definition}

Let $M:(\R,\leq)\rightarrow \vect$. For any $r\leq r' \leq s'\leq s$ in $\R$, we have
\[M(r\leq s)=M(s'\leq s)\circ M(r'\leq s')\circ M(r\leq r'),\]
which implies $\mathrm{rank}(M(r\leq s))\leq \mathrm{rank}(M(r'\leq s'))$.

\begin{definition}\label{dfn:rank invariant} Let $M:(\R,\leq)\rightarrow \vect$. The \textbf{rank function} $\rk(M):(\Int,\subset)\rightarrow \Zop$ of $M$ is defined as $[a,b]\mapsto \mathrm{rank}\big(M(a\leq b)\big).$
\end{definition}

\section{Interleaving distances in general} 
We review the general notion of interleaving distance in the language of \cite{deSilva2018} (Defn.~\ref{def:Interleavings of persistent elements} is a special instance of the definition below). Consult \cite{mac2013categories} for
general definitions related to category theory.  

Let $\PPPP$ be a poset with a flow, and $\PP$ is viewed as a category. Then, for each $\eps\in [0,\infty)$, $\Omega_\eps$ is an endofunctor on $\PP$ and we have $\Ifunc_\PP\leq \Omega_\eps$. 
We view this inequality as a natural transformation $\eta_\eps:\Ifunc_\PP\to \Omega_\eps$. Let $\C$ be any category and let $M:\PP\to \C$ be any functor. Then, we have a natural transformation $M\eta_\eps:M\to M\Omega_\eps$.

\begin{definition}\label{def:interleaving distance between persistence modules} 
Let $\PPPP$ be a poset with a flow. Any two functors $M,N:\PP\to \C$ are called \textbf{$\Omega_\eps$-interleaved} if there exists a pair of natural transformations $\varphi:M\to N\Omega_\eps$ and $\psi:N\to M\Omega_\eps$ such that the diagram below commutes. 
\[\begin{tikzcd}[ampersand
  replacement=\&]
M \arrow[rr,"M\eta_\eps"]\arrow[ddrr,"\varphi", pos=0.8] \& \&  M\Omega_\eps \arrow[rr,"M\eta_\eps\Omega_\eps"]\arrow[ddrr,"\varphi\Omega_\eps",pos=0.8] \& \& M\Omega_\eps \Omega_\eps\\
\\
N \arrow[rr,"N\eta_\eps",swap]\arrow[uurr,"\psi", pos=0.2,crossing over] \& \& N\Omega_\eps \arrow[rr,"N\eta_\eps\Omega_\eps",swap] \arrow[uurr,"\psi\Omega_\eps", pos=0.2,crossing over] \& \& N\Omega_\eps \Omega_\eps
\end{tikzcd}\]
The interleaving distance (with respect to $\Omega$) is:
$$\dint^\C(M,N) = \inf\{\eps\geq 0: \ M,N \ \mathrm{are \ } \Omega_\eps-\mbox{interleaved}\}$$
where $\dint^\C(M,N):=\infty$ if there is no $\eps$-interleaving for any $\eps\geq 0$.
\end{definition}

When $\C$ is a poset $\QQ$, this definition reduces to Defn.~\ref{def:Interleavings of persistent elements}. When $\PP = \R^n$ with the flow in (\ref{eq:l-infinity}) and $\C=\vect$, the distance $\dint^\C$ is the standard interleaving distance \cite{chazal2009proximity,lesnick2015theory}.

\section{(Gromov-)Hausdorff and Bottleneck distances.}\label{sec:dH}

We recall the definitions of the  \textbf{Hausdorff distance},  \textbf{Gromov-Hausdorff distance} \cite[Section 7.3.3]{burago2001course} and the \textbf{Bottleneck distance} \cite{lawler2001combinatorial} in that order.
\begin{definition}[Hausdorff distance]\label{def:Hausdorff distance} Let $A$ and $B$ be closed subsets of a metric space $(M,d)$. The Hausdorff
distance between $A$ and $B$ is defined as
$\dhaus(A, B) = \inf\{r \in [0,\infty): A \subset B^\eps \mbox{ and } B \subset A^{\eps}\}$, where $A^\eps:=\{m \in M: \exists a\in A,\ d(a,m)\leq \eps \}.$
\end{definition}

Let $(X,d_X)$ and $(Y,d_Y)$ be any two metric spaces and let $\tripod$ be a tripod (i.e.~a pair of surjective maps) between $X$ and $Y$. Then, the \emph{distortion} of $R$ is defined as \[\displaystyle\dis(R):=\sup_{\substack{z,z'\in Z}}\abs{d_X\left(\varphi_X(z),\varphi_X(z')\right)-d_Y\left(\varphi_Y(z),\varphi_Y(z')\right)}.\]

The Gromov-Hausdorff distance $\dgh$ measures how far two metric spaces are from being isometric. 

\begin{definition}[Gromov-Hausdorff distance] \label{def:Gromov-Hausdorff} The Gromov-Hausdorff distance between compact metric spaces $(X,d_X)$ and $(Y,d_Y)$ is defined as
	\[\dgh\left((X,d_X),(Y,d_Y)\right):=\frac{1}{2}\inf_R\ \dis(R),\]
	where the infimum is taken over all tripods $R$ between $X$ and $Y$.
\end{definition}

\paragraph{Bottleneck distance.} Bottleneck distance is an extensively studied metric. We adopt notation in \cite{bauer2013induced} to describe it. 
Partial bijections are referred to as \emph{matchings}. Given two nonempty sets $A$ and $B$, we use $\sigma:A\nrightarrow B$ to denote a matching $\sigma\subset A\times B$. The canonical projections of $\sigma$ onto $A$ and $B$ are denoted by $\mathrm{coim} (\sigma)$ and $\mathrm{im} (\sigma)$, respectively. By $\langle a,b \rangle$ for $a<b$ in $\R$, we denote one of the real intervals $(a,b)$, $(a,b]$, $[a,b)$, and $[a,b]$.

Letting $\mathcal{A}$ be a multiset of intervals in $\R$ and $\eps\geq 0$, $$\mathcal{A}^\eps:=\{\langle b,d\rangle \in \mathcal{A}:b+\eps<d \}=\{ I\in\mathcal{A}:[t,t+\eps]\subset I\ \mbox{for some}\ t\in\R\}.$$ Note that $\mathcal{A}^0=\mathcal{A}$.  
	\begin{definition}[Bottleneck distance]\label{def:bottleneck} Let $\mathcal{A}$ and $\mathcal{B}$ be multisets of intervals in $\R$. We define a $\delta$-matching between $\mathcal{A}$ and $\mathcal{B}$ to be a matching $\sigma:\mathcal{A}\nrightarrow\mathcal{B}$ such that 
		$\mathcal{A}^{2\delta}\subset  \mathrm{coim} (\sigma)$, $\mathcal{B}^{2\delta}\subset \mathrm{im} (\sigma)$, and if $\sigma\langle b,d\rangle=\langle b',d'\rangle$, then $$\langle b,d\rangle \subset \langle b'-\delta, d'+\delta \rangle,\hspace{5mm} \langle b',d'\rangle \subset \langle b-\delta, d+\delta \rangle.$$
		with the convention $+\infty +\delta=+\infty$ and $-\infty- \delta=-\infty$. We define the bottleneck distance $\db$ by
		$$\db(\mathcal{A},\mathcal{B}):=\inf\{\delta\in[0,\infty): \exists \mbox{$\delta$-matching between $\mathcal{A}$ and $\mathcal{B}$}\}.$$We declare $\db(\mathcal{A},\mathcal{B})=+\infty$ when there is no $\delta$-matching between $\mathcal{A}$ and $\mathcal{B}$ for any $\delta\in[0,\infty)$.
	\end{definition}

\section{On the (generalized) Gromov-Hausdorff distance}\label{sec:on the tripod distance}

The goal of this section is to establish basic properties of $\dgh$ in Defn.~\ref{def:generalized tripod distance}; Props.~\ref{prop:dt is an extended pseudometric}, \ref{prop:dt-zero-implies-homotopy equivalence}, \ref{prop:coarse upper bound}, \ref{prop:HC lowerbound}, and \ref{prop:omnipresence of dt}.  Throughout this section, $X$, $Y$, $Z$, and $W$ will denote nonempty finite sets.

\begin{proposition}\label{prop:dt is an extended pseudometric}
$\dt$ in Defn. \ref{def:generalized tripod distance} is an extended pseudometric.
\end{proposition}

\begin{proof}
Symmetry and non-negativity are clear. We prove the triangle inequality. Consider any three filtrations $\Ffunc_X:\posetflow\rightarrow \Simp(X)$, $\Gfunc_Y:\posetflow\rightarrow \Simp(Y
)$, and $\Hfunc_W:\posetflow\rightarrow \Simp(W)$. Assume that, for $\eta_1,\eta_2>0$, we have: 
\[\dt(\Ffunc_X,\Gfunc_Y)<\eta_1\mbox{ and }\dt(\Gfunc_Y,\Hfunc_W)<\eta_2.\]
Then, there exist tripods $\tripodd$ and $\tripoddd$ such that \[\dint^{\Simp(Z_1)}(\varphi_X^\ast\Ffunc_X,\varphi_Y^\ast\Gfunc_Y)<\eta_1\mbox{ and }\dint^{\Simp(Z_2)}(\psi_Y^\ast\Gfunc_Y,\psi_Y^\ast\Gfunc_Y)<\eta_2.
\]
 Consider the set $Z:=\left\{(z_1,z_2)\in Z_1\times Z_2:\varphi_Y(z_1)=\psi_Y(z_2)\right\}$ and let $\pi_1:Z\rightarrow Z_1$ and $\pi_2:Z\rightarrow Z_2$ be the canonical projections to the first and the second coordinate, respectively. We define the composite tripod $R_2\circ R_1$ as follows:
 \begin{equation}\label{eq:comosition}
 	R_2\circ R_1:X \xtwoheadleftarrow{\omega_X}\ Z \xtwoheadrightarrow{\omega_W}\ W,\ \mbox{where}\ \  \omega_X:=\varphi_X\circ\pi_1, \ \ \omega_W:=\psi_W\circ \pi_2.
 	\end{equation}	
 	\begin{center}	
 		\begin{tikzcd}		
 			\&\&Z\arrow[swap,two heads]{ld}{\pi_1}\arrow[two heads]{rd}{\pi_2}\\
 			\&Z_1\arrow[swap,two heads]{ld}{\varphi_X}\arrow[two heads]{rd}{\varphi_Y}\&\&Z_2\arrow[swap,two heads]{ld}{\psi_Y}\arrow[two heads]{rd}{\psi_W}
 			\\X\&\&Y\&\&W
 		\end{tikzcd}
 	\end{center}
Also, let $\omega_Y:=\varphi_Y\circ \pi_1=\psi_Y\circ \pi_2$. 	
By taking the composite tripod, we have:
\begin{align*}
   \dt(\Ffunc_X,\Hfunc_W)&\leq \dint^{\Simp(Z)}(\omega_X^\ast \Ffunc_X,\omega_Z^\ast \Hfunc_Z)\\ &\leq\dint^{\Simp(Z)}(\omega_X^\ast \Ffunc_X,\omega_Y^\ast \Gfunc_Y)+\dint^{\Simp(Z)}(\omega_Y^\ast \Gfunc_Y,\omega_Z^\ast \Hfunc_Z)\\&=\dint^{\Simp(Z_1)}(\varphi_X^\ast \Ffunc_X,\varphi_Y^\ast \Gfunc_Y)+\dint^{\Simp(Z_2)}(\psi_Y^\ast \Gfunc_Y,\psi_Z^\ast \Hfunc_Z)\\&<\eta_1+\eta_2.
\end{align*}
The desired inequality follows by letting $\eta_1\searrow \dt(\Ffunc_X,\Gfunc_Y)$ and $\eta_2\searrow \dt(\Gfunc_Y,\Hfunc_W)$.
\end{proof}

Next we show that $\dt=0$ implies a homotopy equivalence between filtrations.

\begin{proposition}
\label{prop:dt-zero-implies-homotopy equivalence}
Let $\Ffunc_X:\PPPP\to\Simp(X)$ and $\Ffunc_Y:\PPPP\to\Simp(Y)$ be any two  filtrations. Suppose that $\dt(\Ffunc_X,\Ffunc_Y)=0$.  Then $\Ffunc_X\simeq\Ffunc_Y$ (the converse does not hold; see Ex.~\ref{ex:filtrations}).\footnote{When $\PP=\R$, a more general statement can be found in \cite[Rmk.~6.8.9]{scoccola2020locally} in connection with the \emph{homotopy interleaving distance} \cite{blumberg2017universality}. }
\end{proposition}
This proposition can be proved in a similar way to \cite[Cor.~2.1]{memoli2017distance}.
\begin{proof}
Let $\tripod$ be a minimizer for $\dt$ (cf. the first footnote in  Defn.~\ref{def:general formigram distance}).
Then, $\dint(\varphi_X^*\Ffunc_X,\varphi_Y^*\Ffunc_Y)=0$.
By Defn.~\ref{def:generalized tripod distance}, we obtain
\[\max_{\sigma\in \pow(Z)} d_{\widehat{\Omega}}\left((\varphi_X^{\ast}\Ffunc_X)^{-1}(\sigma^\uparrow),(\varphi_Y^{\ast}\Gfunc_Y)^{-1}(\sigma^{\uparrow})\right)=0.
\]
This implies $\varphi_X^*\Ffunc_X=\varphi_Y^*\Ffunc_Y$.
By Quillen’s Theorem A \cite{quillen1973higher},
we have that $\Ffunc_X\simeq\varphi_X^*\Ffunc_X$ and $\varphi_Y^*\Ffunc_Y\simeq \Ffunc_Y$, completing the proof. 
\end{proof}

\paragraph{Upper bound for $\dt$.} We aim at obtaining a coarse upper bound for $\dt$ (Prop.~\ref{prop:coarse upper bound}).

\begin{lemma}[Weak join-preserving property of $\domhat$]\label{lem:union domhat} Let $\posetflow$ be a poset with a flow. Let $(A_j)$ be a family of upper sets in $\PP$ and let $B$ be another upper set in $\PP$. Then,
\begin{equation}\label{eq:union domhat}
    \domhat\left(\bigcup_j A_j,B\right)\leq \sup_{j\in J} \domhat\left(A_j,B\right)
\end{equation}
\end{lemma}
It is not difficult to find an example that shows the inequality in (\ref{eq:union domhat}) can be strict.

\begin{proof}
Let $\eps>0$ such that $\domhat\left(A_j,B\right)<\eps$ for all $j\in J$. Then for each $j$ we have $A_j\subset \omegahat_\eps(B)$ and thus $\bigcup_j A_j \subset \omegahat_\eps(B)$. Also, for every $j$, we have $B\subset \omegahat(A_j)$, which implies $B\subset \bigcup_j \omegahat(A_j)$. By Rmk.~\ref{rem:upper sets and omega hat properties} \ref{item:union and intersection preserved}, we have $B\subset \omegahat_\eps \left(\bigcup_j A_j\right)$. Therefore, the left-hand side of inequality (\ref{eq:union domhat}) is at most $\eps$.
\end{proof}

For the singleton set $\{\ast\}$, note that $\Simp(\{\ast\})$ includes only the two simplicial complexes; the empty complex and $\{\{\ast\}\}$.
Defn.~\ref{def:generalized tripod distance} directly implies:
\begin{lemma}\label{lem:comparison with one-point}
Given any $\Ffunc_X:\posetflow\rightarrow \Simp(X)$ and $\Hfunc_{\{\ast\}}:\posetflow\rightarrow \Simp(\{\ast\})$, we have 
\begin{align*}
\dt(\Ffunc_X,\Hfunc_{\{\ast\}}) &= \max_{\substack{\sigma\subset X\\\sigma\neq \emptyset}}\domhat\big(\Ffunc_X^{-1}(\sigma^\uparrow),\Hfunc^{-1}(\ast^\uparrow))\big).
\end{align*}

\end{lemma}

By invoking the triangle inequality of $\dt$ and the lemma above, we have:
\begin{proposition}[Coarse upper bound for $\dt$]\label{prop:coarse upper bound}
Given any filtrations $\Ffunc_X:\posetflow\rightarrow \Simp(X)$ and $\Gfunc_Y:\posetflow\rightarrow \Simp(Y)$, we have  
\[\dt(\Ffunc_X,\Gfunc_Y)\leq \min_{A\in U(\PP)} \left(\max_{\substack{\sigma\subset X\\\sigma\neq \emptyset}}\domhat\big(\Ffunc_X^{-1}(\sigma^\uparrow),A\big)+\max_{\substack{\tau\subset Y\\\tau\neq \emptyset}}\domhat\big(\Gfunc_Y^{-1}(\tau^{\uparrow}),A\big)\right).\]
\end{proposition}

\paragraph{Lower bound for $\dt$.}

Let $K\in \Simp(X)$ and let $K^{(0)}$ be the vertex set of $K$. For $x,x'\in K^{(0)}$, we write $x\sim_{\pi_0(K)} x'$ if there exists a sequence of 1-simplices in $K$ connecting $x$ and $x'$, i.e.~there exist $x_1,\ldots,x_n\in K^{(0)}$ with $\{x,x_1\}$, $\{x_1,x_2\}$, $\ldots$, $\{x_n,x'\}\in K$ (if the sequence is empty, then $x=x'$). This defines an equivalence relation on $K^{(0)}$ and thus $\pi_0(K):=K^{(0)}/\sim_{\pi_0(K)}\in \subpart(X)$.

Then note that $\pi_0$ serves as a \pe{} $\Simp(X)\rightarrow \subpart(X)$. 
Recall $\dgh$ in Defn.~\ref{def:dGH between HC over P}.

\begin{proposition}\label{prop:HC lowerbound} 
For any $\Ffunc_X:\PPPP\rightarrow \Simp(X)$ and $\Gfunc_Y:\PPPP\rightarrow \Simp(Y)$, 
\[\dgh(\pi_0\circ \Ffunc_X,\pi_0\circ \Gfunc_Y)\leq \dt(\Ffunc_X,\Gfunc_Y).\]
\end{proposition}
We remark that the LHS is a better lower bound for $\dt$ than $\dint(\Hrm_0(\Ffunc_X),\Hrm_0(\Gfunc_Y))$ in Thm.~\ref{thm:H_k-lower-bound}. For example, if $\Ffunc_X$ and $\Gfunc_Y$ are the Rips filtrations of metric spaces $(X,d_X)$ and $(Y,d_Y)$, then the inequality above coincides with the inequality in Eqn.~(\ref{eq:SHLH GH stability}). The LHS of Eqn.~(\ref{eq:SHLH GH stability}) is known to be a better lower bound than the interleaving distance between the zeroth homology of the Rips filtrations of $(X,d_X)$ and $(Y,d_Y)$.

\paragraph{Ubiquity of $\dt$.}  
Definitions in the following three items are relevant to the three statements in Prop.~\ref{prop:omnipresence of dt} respectively in order.  

\begin{enumerate}[label=(\roman*)]
    \item Given a pseudometric space $(X,d_X)$, let $\VR(X,d_X):\R\rightarrow \Simp(X)$ be the \textbf{Vietoris-Rips filtration} of $(X,d_X)$, i.e.~each $r\in \R$ is sent to: \[\VR(X,d_X)(r)=\{\sigma\subset X:\mbox{$\abs{\sigma}<\infty$ and $d_X(x,x')\leq r$ for all $x,x'\in X$}\}.\]
    \item Let us recall the simplexization map $s$ in Defn.~\ref{def:simplexization}. For any nonempty finite set $X$ and any $P\in \subpart(X)$, let $s(P):=\{s(B):B\in P\}\in \Simp(X)$. Then $s$ serves as a \pe{} $\Simp(X)\rightarrow \subpart(X)$.
    \item Let $\lambda>0$. The $\lambda$-flow on the product poset $\Int\times \R_{\geq 0}$ is defined by
\[\Omega^\lambda:=\{\Omega_\eps^{\lambda}:((a,b),r)\mapsto (a-\eps,b+\eps),r+\lambda\eps\}_{\eps \geq 0}.\]
\end{enumerate}

\begin{proposition}[Ubiquity
of $\dt^\mathrm{Simp}$]\label{prop:omnipresence of dt} Let $\dgh^\mathrm{Met}$, $\dgh^\mathrm{HC}$, and $\dgh^\mathrm{Simp}$ be the Gromov-Hausdorff distances between metric spaces, between hierarchical clusterings, and between simplicial filtrations, respectively (Defns.~\ref{def:Gromov-Hausdorff}, \ref{def:dGH between HC over P}, \ref{def:tripod distance}). We have:
\begin{enumerate}[label=(\roman*),topsep=0pt,itemsep=-1ex,partopsep=1ex,parsep=1ex]
    \item \cite[Prop.~2.8]{memoli2019quantitative} For any finite pseudometric spaces $(X,d_X)$ and $(Y,d_Y)$, \[\dgh^\mathrm{Met}((X,d_X),(Y,d_Y))=\dt^\mathrm{Simp}(\VR(X,d_X),\VR(Y,d_Y)).\] \label{item:dgh=dt}
    \item For any $(\PP,\leq)$-indexed hierarchical clusterings $\theta_X$ and $\theta_Y$ over $X$ and $Y$ respectively,  \[\dgh^\mathrm{HC}(\theta_X,\theta_Y)=\dt^\mathrm{Simp}(s\circ \theta_X,s\circ \theta_Y).\]
  \label{item:dgh=dt formigrams}
    \item Let $\lambda>0$. For any dynamic metric spaces $\gamma_X$ and $\gamma_Y$, \[\dgh^\lambda(\gamma_X,\gamma_Y)=\dt^\mathrm{Simp}(\VR^\lambda(\gamma_X),\VR^{\lambda}(\gamma_Y)),\]
    where the LHS is the $\lambda$-slack interleaving distance
 \cite[Defn.~2.10]{kim2020spatiotemporal}.\label{item:dgh=dt DMSs}
\end{enumerate}
\end{proposition}
Since the respective proofs of \ref{item:dgh=dt formigrams} and \ref{item:dgh=dt DMSs}  directly follow from the definitions of the involved metrics, we omit them .

\section{Specialization of Thm.~\ref{thm:df complexity}}\label{sec:specialization}
We specialize Thm.~\ref{thm:df complexity} by restricting ourselves to formigrams whose underlying Reeb graphs do not contain any loops; Thm.~\ref{thm:df complexity for trees}. This theorem is a rather direct consequence of Thm.~\ref{thm:dH=dB} and Cor.~\ref{cor:dB complexity matches with dH complexity} below.

Recall the Hausdorff and bottleneck distances in Sec.~\ref{sec:dH}. Recall that, by $\langle a,b \rangle$ for $a<b$ in $\R$, we denote one of the real intervals $(a,b)$, $(a,b]$, $[a,b)$, and $[a,b]$. Given any real intervals $\langle a,b \rangle$ and $\langle c,d \rangle$, we write  $\langle a,b \rangle < \langle c,d \rangle$ if $b<c$.

\begin{theorem}\label{thm:dH=dB}

Let $\mathcal{A}:=\left\{I_i:=\langle a_i,b_i \rangle:  1\leq i \leq m\right\}$ and $\mathcal{B}:=\left\{J_j:=\langle c_j, d_j \rangle: 1\leq j \leq n\right\}$. Assume that $I_i<I_{i+1}$ for $i=1,\ldots,m-1$  and $J_j<J_{j+1}$ for $j=1,\ldots,n-1$. We have:
\[
\dhaus\big(C(\mathcal{A}) ,C(\mathcal{B})\big)= \db(\mathcal{A}, \mathcal{B}),
\]
where $C(\Acal):=\R\setminus \bigcup \Acal$.
\end{theorem}

We prove this theorem at the end of the section. By Thm~\ref{thm:dH=dB} and \cite[Thm.~3.1]{kerber2017geometry}, we obtain:

\begin{corollary}\label{cor:dB complexity matches with dH complexity}
Let $\ell:=\max(m,n)$. Computing $\dhaus\big(C(\mathcal{A}) ,C(\mathcal{B})\big)$ takes $O(\ell^{1.5}\log \ell)$.
\end{corollary}

\begin{theorem}\label{thm:df complexity for trees} 
Assume that the underlying Reeb graphs of $\theta,\theta'\in \Formi(X)$ do not contain any loops. Then, computing $\df(\theta,\theta')$ at worst requires time $O(n^2\ell^{1.5}\log \ell)$,  where $n:=\abs{X}$ and \[\ell:=\max\left(\abs{\crit(\theta_X)},\vert\crit(\theta')\vert\right).\]
\end{theorem}

\begin{proof}
Fix $x,x'\in X$ and let \[\theta_{\{x,x'\}}:=\left\{t\in \R: \mbox{$x$ and $x'$ belong to the same block in $\theta(t)$} \right\}.\] We claim that $\dhaus^{\Int}\left(\theta^{-1}(\{x,x'\}^{\uparrow}),\theta'^{-1}(\{x,x'\}^{\uparrow})\right)$ in Eqn.~(\ref{eq:structure thm}) is equal to $\dhaus^{\R}\left(\theta_{\{x,x'\}},\theta'_{\{x,x'\}}\right)$. To show this, it suffices to show that for any $\eta>0$ one distance is upper bounded by $\eta$ implies the other is too. Indeed, the no-loop assumption implies that, for any $\eta>0$, the both distances are upper bounded by $\eta$ if and only if the following holds; if $x$ and $x'$ belong to $\theta(t)$ (resp. $\theta'(t)$), then there exists $s\in [t-\eps,t+\eps]$ such that $x$ and $x'$ belong to $\theta(s)$ (resp. $\theta'(s)$).

By Cor.~\ref{cor:dB complexity matches with dH complexity}, computing $\dhaus^{\R}\left(\theta_{\{x,x'\}},\theta'_{\{x,x'\}}\right)$ requires $O(\ell^{1.5}\log \ell)$. Since there are $n^2$ singleton and doubleton subsets of $X$, by the equality in Eqn.~(\ref{eq:structure thm}), we compute $\df(\theta,\theta')$ in $O(n^2\cdot \ell^{1.5}\log \ell)$. 
\end{proof}

\paragraph{Proof of Thm.~\ref{thm:dH=dB}.} We prove Thm.~\ref{thm:dH=dB}. To avoid trivialities, assume that $\mathcal{A}\neq \mathcal{B}$. Let $E_{\Acal}$ be the collection of $I_i$'s endpoints, i.e. $\Acal=\{a_i\}_{i=1}^m\cup \{b_i\}_{i=1}^m$. Letting $b_0=-\infty$ and $a_{m+1}=\infty$, we have $C(\mathcal{A})=\mathbb{R}\setminus \bigcup \mathcal{A}=\bigsqcup_{i=0}^m \langle b_i,a_{i+1}\rangle$.  Similarly, we define $E_\Bcal$ and $C(\Bcal)$.

\begin{lemma}\label{lem:for dH leq dB} $\displaystyle \dhaus(C(\mathcal{A}), C(\mathcal{B}))=\max\left\{\max_{a\in C(\mathcal{A})\cap (\bigcup \mathcal{B})}\min_{b\in E_\mathcal{B}} \abs{a-b}, \max_{b\in C(\mathcal{B})\cap (\bigcup \mathcal{A})}\min_{a\in E_\mathcal{A}}\abs{a-b}\right\}$.
\begin{proof} Since
$$\dhaus(C(\mathcal{A}), C(\mathcal{B}))=\max\left\{\max_{a\in C(\mathcal{A})}\min_{b\in C(\mathcal{B})}\abs{a-b}, \max_{b\in C(\mathcal{B})}\min_{a\in C(\mathcal{A})}\abs{a-b}\right\},$$
by symmetry, it suffices to prove that
\begin{equation}\label{max}
\max_{a\in C(\mathcal{A})}\min_{b\in C(\mathcal{B})}\abs{a-b}=\max_{a\in C(\mathcal{A})\cap (\bigcup \mathcal{B})}\min_{b\in E_\mathcal{B}} \abs{a-b}.
\end{equation}
 For $a\in C(\mathcal{A})$, let $\displaystyle \varphi(a):=\min_{b\in C(\mathcal{B})}\abs{a-b}$. If $a\in C(\mathcal{B})$, then clearly $\varphi(a)=0$. Hence, restricting the domain $C(\mathcal{A})$ of $\varphi$ to the intersection of $C(\mathcal{A})$ and $\mathbb{R}\setminus C(\mathcal{B})=\bigcup \mathcal{B}$ does not affect the maximum of $\varphi$, which implies:
\begin{equation}\label{max1}
\max_{a\in C(\mathcal{A})}\min_{b\in C(\mathcal{B})}\abs{a-b}=\max_{a\in C(\mathcal{A})\cap (\bigcup\mathcal{B})}\min_{b\in C(\mathcal{B})}\abs{a-b}.\footnote{Since we assumed $\mathcal{A}\neq \mathcal{B}$, the set $C(\mathcal{A})\cap (\bigcup\mathcal{B})$ cannot be empty.}
\end{equation}
Next, fix an arbitrary $a\in \bigcup \mathcal{B}$. Then the closest point in $\mathbb{R}\setminus \bigcup \mathcal{B}=C(\mathcal{B})$ to $a$ is obviously located on the boundary of $C(\mathcal{B})$, the set of endpoints $E_\mathcal{B}$, which implies 
$\min_{b\in C(\mathcal{B})}\abs{a-b}=\min_{b\in E_{\mathcal{B}}}\abs{a-b}$.
Therefore, the RHS of (\ref{max}) coincides with the RHS of (\ref{max1}).
\end{proof}
\end{lemma}

\emph{Proof of $\dhaus\big(C(\mathcal{A}) ,C(\mathcal{B})\big)\leq \db(\mathcal{A}, \mathcal{B})$.}

Let $\sigma:\mathcal{A} \nrightarrow \mathcal{B}$
be a $\delta$-matching. By Lem.~\ref{lem:for dH leq dB} and symmetry, it suffices to prove that $$\max_{a\in C(\mathcal{A})\cap (\bigcup \mathcal{B})}\min_{b\in E_\mathcal{B}} \abs{a-b}\leq \delta.$$
Fix any $a\in C(\mathcal{A})\cap (\bigcup \mathcal{B})$. Then there are $0\leq i\leq m$ and $1\leq j \leq n$ such that 
\begin{equation}\label{intersection}
a\in \langle b_i, a_{i+1}\rangle\cap \langle c_j, d_j \rangle. 
\end{equation}
\textbf{Case 1.} Assume that length$\langle c_j,d_j \rangle\leq 2\delta$. Since $a\in \langle c_j, d_j \rangle$, we have: 
\begin{align*}
\min_{b\in E_{\mathcal{B}}}\abs{a-b}&\leq \min \{\abs{a-c_j}, \abs{a-d_j}\}\leq \delta.
\end{align*}
\textbf{Case 2.} Assume that length$\langle c_j,d_j \rangle>2\delta$. Then there exists $1\leq k\leq m$ such that $\langle a_k, b_k \rangle$ is matched with $\langle c_j,d_j \rangle$ via the matching $\sigma$. Note that the intersection in (\ref{intersection}) can possibly be expressed as follows:
$$\langle b_i, a_{i+1}\rangle\cap \langle c_j, d_j \rangle=\begin{cases} \langle c_j, d_j \rangle& \mbox{Case (a)} \\ \langle b_i, a_{i+1}\rangle& \mbox{Case (b)}\\ \langle b_i, d_j\rangle & \mbox{Case (c)} \\ \langle c_j, a_{i+1}\rangle& \mbox{Case (d)} \end{cases} $$

Assume Case (a), i.e. $J_j=\langle c_j, d_j \rangle\subset \langle b_i,a_{i+1}\rangle$ (See Fig.~\ref{fig:intervals}). Given any intervals $\langle a,b \rangle$ and $\langle c,d \rangle$, let $\norm{\langle a,b \rangle-\langle c,d \rangle}_{\infty}:=\max\{\abs{a-c},\abs{b-d}\}$. Note that the closest intervals to $\langle c_j, d_j \rangle$ in $\mathcal{A}$ in the metric  $\norm{\cdot}_\infty$ are $I_i=\langle a_i, b_i\rangle$ and $I_{i+1}=\langle a_{i+1}, b_{i+1}\rangle$. However, both $\norm{I_i-J_j}_\infty$ and $\norm{I_{i+1}- J_j}_\infty$ are greater than $\delta$ because $2\delta \leq|d_j-c_j|\leq \abs{b_i-d_j} \leq \norm{I_i-J_j}_\infty$ and $2\delta \leq|d_j-c_j|= \abs{a_{i+1}-c_j} \leq \norm{I_{i+1}- J_j}_\infty$. This contradicts the fact that $\sigma$ is a $\delta$-matching. Therefore, Case (a) cannot happen. 

\begin{figure}
\begin{tikzpicture}
\centering
\draw (-6,0) node {};

\draw plot (-3,1.5)--(-0.5,1.5);
\draw plot (6,1.5)--(8,1.5);
\draw plot (0,0.5)--(5,0.5);

\draw (6,1.5) node [anchor=south] {$a_{i+1}$};
\draw (8,1.5) node [anchor=south] {$b_{i+1}$};

\draw (-3,1.5) node [anchor=south] {$a_{i}$};

\draw (-0.5,1.5) node [anchor=south] {$b_{i}$};
\draw (0,0.5) node [anchor=south] {$c_{j}$};
\draw (5,0.5) node [anchor=south] {$d_{j}$};

\end{tikzpicture}
\caption{An illustration for Case 2, (a).\label{fig:intervals}}
\end{figure} \par

Assume Case (b). Also, assume that $J_j=\langle c_j, d_j \rangle$ is matched with $I_k$ for $k\leq i$ via $\sigma$. Then since $b_k\leq b_i\leq a \leq d_j$, $$\abs{a-d_j}\leq \abs{b_k-d_j}\leq \norm{I_k-J_j}_{\infty}\leq \delta.$$ Now, suppose that  $\langle c_j, d_j \rangle$ is matched with $I_k$ for $k> i$. Then since $c_j\leq a \leq a_{i+1} \leq a_k$, $$\abs{a-c_j}\leq \abs{a_k-c_j}\leq \norm{I_k-J_j}_{\infty}\leq \delta.$$ Therefore, we have $$\displaystyle \min_{b\in E_{\mathcal{B}}}\abs{a-b}\leq \min\{\abs{a-c_j}, \abs{a-d_j}\} \leq \delta$$ as desired. \par

Assume Case (c), i.e., $c_j \leq b_i \leq d_j \leq a_{i+1}$. Note that $I_k$ cannot be matched with $J_j$ for $k>i$ via $\sigma$ because $c_j<d_j\leq a_{i+1}$ and in turn $$\delta< 2\delta < d_j-c_j \leq a_{i+1}-c_j \leq a_k-c_j \leq \norm{I_k-J_j}_{\infty}.$$ Hence, $J_j$ must be matched with $I_k$ for some $k\leq i$. Take $k\leq i$ such that $I_k$ is matched with $J_j$ via $\sigma$. Since $b_k\leq b_i\leq a < d_j$, we have 

$$\abs{a-d_j}\leq \abs{d_j-b_k} \leq \norm{J_j-I_k}_{\infty}\leq \delta.$$
Therefore, we have $$\displaystyle \min_{b\in E_{\mathcal{B}}}\abs{a-b}\leq \abs{a-d_j} \leq \delta.$$\par
Assume Case (d). By a similar argument to Case (c), $J_j$ must be matched with $I_k$ for some $k>i$ and this in turn implies $\abs{a-c_j}\leq\delta$. Hence again $$\displaystyle \min_{b\in E_{\mathcal{B}}}\abs{a-b}\leq \abs{a-c_j} \leq \delta.$$  We have shown that $\displaystyle \min_{b\in E_{\mathcal{B}}}\abs{a-b}\leq \delta$ for all $a\in C(\mathcal{A})\cap (\bigcup \mathcal{B})$ as desired.

\paragraph{Proof of $\dhaus\big(C(\mathcal{A}) ,C(\mathcal{B})\big)\geq \db(\mathcal{A}, \mathcal{B})$.} Let $\eps>0$. Define $\mathcal{A}^{\eps}$ to be the collection of intervals in $\mathcal{A}$ whose length is at least $\eps$. Also, given any interval $I=\langle a,b \rangle$, let 
$$I^{-\eps}:=\begin{cases} \emptyset & \mbox{if\ $b-a \leq 2\eps$}\\ \langle a+\eps, b-\eps \rangle & \mbox{otherwise}. \end{cases}$$

Let $\big(C(\mathcal{A})\big)^\eta$ be the $\eta$-thickening of $C(\mathcal{A})$, i.e. $\left\{r\in \R: \exists p\in C(\mathcal{A}), \ \abs{p-r}\leq \eta\right\}$. We have:

\begin{lemma} $\displaystyle \mathbb{R}\setminus \big(C(\mathcal{A})\big)^\eta=\bigcup_{I\in \Acal^{2\eta}} I^{-\eta}$ (the proof is elementary but rather tedious so we omit it).

\end{lemma}
Suppose that $\dhaus\big(C(\mathcal{A}) ,C(\mathcal{B})\big)\leq \eta$ for some $\eta>0$. We wish to construct an $\eta$-matching $\sigma:\mathcal{A}\nrightarrow \mathcal{B}$. Note that $C(\mathcal{B})\subset \big(C(\mathcal{A})\big)^\eta$ by assumption and thus $\bigcup_{j=1}^nJ_j=\mathbb{R}\setminus C(\mathcal{B})\supset \mathbb{R}\setminus \big(C(\mathcal{A})\big)^\eta=\bigcup_{I_i\in \mathcal{A}^{2\eta}} I_i^{-\eta}$. This implies that there exists $j$ such that $I_i^{-\eta}\subset J_j$, equivalently $I_i\subset J_j^\eta$, for each $I_i\in \mathcal{A}^{2\eta}$ since the union $\bigcup_{j=1}^nJ_j$ is disjoint. We already have shown the following proposition.
\begin{proposition}
Assume that $\eta\geq \dhaus\big(C(\mathcal{A}),C(\mathcal{B})\big)$ for some $\eta>0$. Then, there exist functions $f:A^{2\eta}\rightarrow B$ and $g:B^{2\eta}\rightarrow A$ such that 
$$I_i \subseteq (J_{f(i)})^\eta\,\,\mbox{for all $i\in A^{2\eta}$ and } J_j \subseteq (I_{g(j)})^\eta\,\,\mbox{for all $j\in B^{2\eta}$}$$ where $A^{2\eta}=\left\{1\leq i \leq m: I_i\in \mathcal{A}^{2\eta}\right\}$ and $B^{2\eta}=\left\{1\leq j \leq n: J_j \in \mathcal{B}^{2\eta}\right\}$. 
\end{proposition}

Let $f,g$ be as in the proposition. We construct an $\eta$-matching between $\mathcal{A}$ and $\mathcal{B}$. We write $A^{2\eta}=A^{2\eta}_0\sqcup A^{2\eta}_\ast$ where
$A^{2\eta}_0:=\left\{i\in A^{2\eta}: f(i)\notin B^{2\eta}\right\}$ and $A^{2\eta}_{*}:=\left\{i\in A^{2\eta}:f(i)\in B^{2\eta}\right\}$. Similarly, we write $B^{2\eta}=B^{2\eta}_0\sqcup B^{2\eta}_\ast$ using the function $g$.

\begin{proposition}\label{one-to-one} $g\circ f|_{A^{2\eta}_{*}}=\mathrm{id}_{A^{2\eta}_{*}}$  and $f\circ g|_{B^{2\eta}_{*}}=\mathrm{id}_{B^{2\eta}_{*}}$.
\end{proposition}
\begin{proof}
We only show the first equality. Take any $i \in A^{2\eta}_\ast$. We know that 
\begin{equation*}
I_i \subseteq \big(J_{f(i)}\big)^\eta\subseteq \bigg(\big(I_{g(f(i))}\big)^\eta\bigg)^\eta = \big(I_{g(f(i))}\big)^{2\eta}.
\end{equation*}
Let $j=g(f(i))$. The above equation means that $\langle a_i,b_i \rangle\subseteq \langle a_j-2\eta,b_j+2\eta \rangle$. However, since $\mathrm{length}\langle a_i,b_i \rangle\geq 2\eta$, this is impossible unless either $\langle a_i,b_i \rangle$ and $\langle a_j,b_j \rangle$ share one of their endpoints or have nonempty intersection. Since the intervals in $\Acal$ are disjoint and do not share their endpoints, we have $i=j$.
\end{proof}

Notice two important implications of the above claim: The first is that $f(A_*^{2\eta})\subseteq B_*^{2\eta}$ and $g(B_*^{2\eta})\subseteq A_*^{2\eta}$. The second is that both $f|_{A^{2\eta}_{*}}$ and $g|_{B^{2\eta}_{*}}$ are injective. Now we are going to show that $f$ and $g$ are injective on $A^{2\eta}_{0}$ and $B^{2\eta}_{0}$ respectively as well.

\begin{claim}\label{injective}
The functions $f|_{A^{2\eta}_0}$ and $g|_{B^{2\eta}_0}$ are injective.
\end{claim}
\begin{proof}
We prove the claim for $f$. Assume that $i,j \in A^{2\eta}_0$, and $f(i) = f(j)=k$, which means $(J_{k})^\eta\supseteq I_i$ and $(J_{k})^\eta\supseteq I_{j}$ and hence $(J_{k})^\eta\supseteq I_i \cup I_{j}.$ Therefore,
\begin{align*}
4\eta&\geq 2\eta+\length{J_{k}}&\because i,j \in A_0^{2\eta}
\\&=\length{(J_{k})^\eta}
\\&\geq \length{I_{i}\cup I_{j}}
\end{align*}
This enforces $I_i$ and $I_j$ to have nonempty intersection since each of them has the length$\geq 2\eta$. Thus $i=j$ because the intervals in $\mathcal{A}$ are disjoint. \end{proof}
We are now ready to define an $\eta$-matching $\sigma:\mathcal{A}\nrightarrow \mathcal{B}$. For the sake of simplicity, we would regard $\sigma$ as a matching between index sets $A$ and $B$ of $\mathcal{A}$ and $\mathcal{B}$ respectively by identifying elements in $\mathcal{A}$ and $\mathcal{B}$ to their indexes. 
First, we define 
$$\mathrm{coim}(\sigma) = A^{2\eta}_0 \,\sqcup\, A^{2\eta}_\ast\,\sqcup \,g(B^{2\eta}_0),\ \ \  \mbox{and}\ \ \  \mathrm{im}(\sigma) = f(A^{2\eta}_0) \, \sqcup \, B^{2\eta}_\ast \,\sqcup \, B^{2\eta}_0.$$
Then, 
$\mathrm{coim}(\sigma)\supseteq A^{2\eta}$ and $\mathrm{im}(\sigma)\supseteq B^{2\eta}$. Now, define $\sigma: A \nrightarrow B$ as follows:

$$\sigma(i)=\begin{cases} f(i)& \mbox{if $i\in A^{2\eta}=A^{2\eta}_*\sqcup A^{2\eta}_0$}\\ g^{-1}(i)& \mbox{if $i \in g(B_{0}^{2\eta})$}.
\end{cases}$$
By Claim \ref{injective}, $\sigma$ is well-defined . The following diagram depicts the construction of the matching:
$$
\xymatrix{
A^{2\eta}_0 \ar[d]_f & \sqcup &  A^{2\eta}_\ast\ar@{<->}[d]_{f}^{g} & \sqcup & g(B^{2\eta}_0)\\
f(A^{2\eta}_0) & \sqcup & B^{2\eta}_\ast & \sqcup & B^{2\eta}_0\ar[u]_g
}
$$

It remains to show that  $\norm{I_i-J_{\sigma(i)}}_{\infty}\leq \eta$, i.e., $I_i \subseteq (J_{\sigma(i)})^{\eta}$ and $J_{\sigma(i)}\subseteq (I_i)^{\eta}$ for all $i\in \mathrm{coim}(\sigma)$. Recall that $I_i=\langle a_i,b_i \rangle$ and $J_j=\langle c_j, d_j \rangle$ for $1\leq i \leq m$ and $1\leq j \leq n$. \\

\noindent\textbf{Case 1.} Pick $i\in A^{2\eta}_0$ and let $\sigma(i)=f(i)=j$ so that $\length{I_i}\geq 2\eta$ whereas $\length{J_{j}}< 2\eta$. We wish to verify that $I_i \subseteq (J_j)^{\eta}$ and $J_j\subseteq (I_i)^{\eta}$. But, the first inclusion follows automatically from the definition of $f$ and this implies that (1) $a_i\geq c_j-\eta$ and (2) $b_i\leq d_j+\eta$. So we are going to prove $J_j\subseteq (I_i)^{\eta}$ only, which amounts to show that (3) $c_j\geq a_i-\eta$ and (4) $d_j\leq b_i+\eta$. Suppose that (3) is false, i.e., $c_j<a_i-\eta$. Then we have
\begin{align*}
d_j&=c_j+\length{J_j}
\\&< c_j+2\eta
\\&<a_i+\eta && \because c_j<a_i-\eta 
\\&\leq b_i-\eta && \because a_i=b_i-\length{I_i}\leq b_i-2\eta.
\end{align*}
This contradicts (2) and thus (3) must hold. Similarly, the negation of (4) deduce the contradiction to the inequality (1) and thus both (3) and (4) should hold as desired. This strategy  works for the case of $i\in g(B^{2\eta}_0)$ as well since $g$ has the same property as $f$.\\

\noindent\textbf{Case 2.} Pick $i\in A_*^{2\eta}$ and let $\sigma(i)=f(i)=j$. Again by the definition of $f$, we know $I_i\subset (J_j)^\eta$. Further, $J_j\subseteq (I_{g(j)})^\eta$ by the definition of $g$ but recalling $g(j)=g(f(i))=i$ by Claim \ref{one-to-one}, we have  $\norm{I_i-J_{\sigma(i)}}_{\infty}\leq \eta$. 

Assuming $\eta\geq \dhaus\big(C(\mathcal{A}),C(\mathcal{B})\big)$ for some $\eta>0$, we have constructed $\eta$-matching between $\mathcal{A}$ and $\mathcal{B}$.  Therefore, we have inequality $\dhaus\big(C(\mathcal{A}),C(\mathcal{B})\big)\geq \db(\mathcal{A}, \mathcal{B})$ as desired.

\bibliographystyle{plain}  
\bibliography{bibliography}

\begin{thebibliography}{10}

\bibitem{bjerkevik2016stability}
H{\aa}vard Bakke~Bjerkevik.
\newblock On the stability of interval decomposable persistence modules.
\newblock {\em Discrete \& Computational Geometry}, 66(1):92--121, 2021.

\bibitem{bauer2020reeb}
Ulrich Bauer, Claudia Landi, and Facundo M{\'e}moli.
\newblock The {R}eeb graph edit distance is universal.
\newblock {\em Foundations of Computational Mathematics}, 2020.

\bibitem{bauer2013induced}
Ulrich Bauer and Michael Lesnick.
\newblock Induced matchings and the algebraic stability of persistence
  barcodes.
\newblock {\em Journal of Computational Geometry}, 6(2):162--191, 2015.

\bibitem{Bauer2015b}
Ulrich Bauer, Elizabeth Munch, and Yusu Wang.
\newblock {Strong Equivalence of the Interleaving and Functional Distortion
  Metrics for {R}eeb Graphs}.
\newblock In Lars Arge and J{\'a}nos Pach, editors, {\em 31st International
  Symposium on Computational Geometry (SoCG 2015)}, volume~34 of {\em Leibniz
  International Proceedings in Informatics (LIPIcs)}, pages 461--475, Dagstuhl,
  Germany, 2015. Schloss Dagstuhl--Leibniz-Zentrum fuer Informatik.

\bibitem{betthauser2019graded}
Leo Betthauser, Peter Bubenik, and Parker~B Edwards.
\newblock Graded persistence diagrams and persistence landscapes.
\newblock {\em Discrete \& Computational Geometry}, 67(1):203--230, 2022.

\bibitem{billera2001geometry}
Louis~J Billera, Susan~P Holmes, and Karen Vogtmann.
\newblock Geometry of the space of phylogenetic trees.
\newblock {\em Advances in Applied Mathematics}, 27(4):733--767, 2001.

\bibitem{bjerkevik2019computing}
H{\aa}vard~Bakke Bjerkevik, Magnus~Bakke Botnan, and Michael Kerber.
\newblock Computing the interleaving distance is {NP}-hard.
\newblock {\em Foundations of Computational Mathematics}, pages 1--35, 2019.

\bibitem{bjerkevik2021asymptotic}
H{\aa}vard~Bakke Bjerkevik and Michael Kerber.
\newblock Asymptotic improvements on the exact matching distance for
  2-parameter persistence.
\newblock {\em arXiv preprint arXiv:2111.10303}, 2021.

\bibitem{blumberg2017universality}
Andrew~J Blumberg and Michael Lesnick.
\newblock Universality of the homotopy interleaving distance.
\newblock {\em arXiv preprint arXiv:1705.01690}, 2017.

\bibitem{botnan2020relative}
Magnus Botnan, Justin Curry, and Elizabeth Munch.
\newblock A relative theory of interleavings.
\newblock {\em arXiv preprint arXiv:2004.14286}, 2020.

\bibitem{botnan2018algebraic}
Magnus Botnan and Michael Lesnick.
\newblock Algebraic stability of zigzag persistence modules.
\newblock {\em Algebraic \& geometric topology}, 18(6):3133--3204, 2018.

\bibitem{bubenik2015metrics}
Peter Bubenik, Vin De~Silva, and Jonathan Scott.
\newblock Metrics for generalized persistence modules.
\newblock {\em Foundations of Computational Mathematics}, 15(6):1501--1531,
  2015.

\bibitem{bubenik2017interleaving}
Peter Bubenik, Vin De~Silva, and Jonathan Scott.
\newblock Interleaving and {G}romov-{H}ausdorff distance.
\newblock {\em arXiv preprint arXiv:1707.06288}, 2017.

\bibitem{bubenik2015statistical}
Peter Bubenik et~al.
\newblock Statistical topological data analysis using persistence landscapes.
\newblock {\em J. Mach. Learn. Res.}, 16(1):77--102, 2015.

\bibitem{bubenik2014categorification}
Peter Bubenik and Jonathan~A Scott.
\newblock Categorification of persistent homology.
\newblock {\em Discrete \& Computational Geometry}, 51(3):600--627, 2014.

\bibitem{buchin2013trajectory}
Kevin Buchin, Maike Buchin, Marc van Kreveld, Bettina Speckmann, and Frank
  Staals.
\newblock Trajectory grouping structure.
\newblock In {\em Workshop on Algorithms and Data Structures}, pages 219--230.
  Springer, 2013.

\bibitem{burago2001course}
Dmitri Burago, IU~D Burago, Yuri Burago, Sergei~A Ivanov, and Sergei Ivanov.
\newblock {\em A course in metric geometry}, volume~33.
\newblock American Mathematical Soc., Providence, Rhode Island, 2001.

\bibitem{cai2020elder}
Chen Cai, Woojin Kim, Facundo M{\'e}moli, and Yusu Wang.
\newblock Elder-rule staircodes for augmented metric spaces.
\newblock In {\em Proceedings of the thirty-sixth International Symposium on
  Computational Geometry (SoCG 2020)}, 2020.

\bibitem{cardona2013cophenetic}
Gabriel Cardona, Arnau Mir, Francesc Rossell{\'o}, Lucia Rotger, and David
  S{\'a}nchez.
\newblock Cophenetic metrics for phylogenetic trees, after {S}okal and {R}ohlf.
\newblock {\em BMC bioinformatics}, 14(1):3, 2013.

\bibitem{carlsson2009topology}
Gunnar Carlsson.
\newblock Topology and data.
\newblock {\em Bulletin of the American Mathematical Society}, 46(2):255--308,
  2009.

\bibitem{carlsson2010characterization}
Gunnar Carlsson and Facundo M{\'e}moli.
\newblock Characterization, stability and convergence of hierarchical
  clustering methods.
\newblock {\em Journal of machine learning research}, 11(Apr):1425--1470, 2010.

\bibitem{carlsson2010multiparameter}
Gunnar Carlsson and Facundo M{\'e}moli.
\newblock Multiparameter hierarchical clustering methods.
\newblock In {\em Classification as a Tool for Research}, pages 63--70.
  Springer, 2010.

\bibitem{cerri2013betti}
Andrea Cerri, Barbara~Di Fabio, Massimo Ferri, Patrizio Frosini, and Claudia
  Landi.
\newblock Betti numbers in multidimensional persistent homology are stable
  functions.
\newblock {\em Mathematical Methods in the Applied Sciences},
  36(12):1543--1557, 2013.

\bibitem{chambers2020family}
Erin~Wolf Chambers, Elizabeth Munch, and Tim Ophelders.
\newblock {A Family of Metrics from the Truncated Smoothing of Reeb Graphs}.
\newblock In {\em 37th International Symposium on Computational Geometry (SoCG
  2021)}, volume 189 of {\em Leibniz International Proceedings in Informatics
  (LIPIcs)}, pages 22:1--22:17, Dagstuhl, Germany, 2021. Schloss Dagstuhl --
  Leibniz-Zentrum f{\"u}r Informatik.

\bibitem{chazal2009proximity}
Fr{\'e}d{\'e}ric Chazal, David Cohen-Steiner, Marc Glisse, Leonidas~J Guibas,
  and Steve~Y Oudot.
\newblock Proximity of persistence modules and their diagrams.
\newblock In {\em Proceedings of the twenty-fifth annual symposium on
  Computational geometry}, pages 237--246. ACM, 2009.

\bibitem{chazal2009gromov}
Fr{\'e}d{\'e}ric Chazal, David Cohen-Steiner, Leonidas~J Guibas, Facundo
  M{\'e}moli, and Steve~Y Oudot.
\newblock Gromov-{H}ausdorff stable signatures for shapes using persistence.
\newblock {\em Computer Graphics Forum}, 28(5):1393--1403, 2009.

\bibitem{chazal2016structure}
Fr{\'e}d{\'e}ric Chazal, Vin De~Silva, Marc Glisse, and Steve Oudot.
\newblock {\em The structure and stability of persistence modules}, volume~10.
\newblock Springer, Switzerland, 2016.

\bibitem{chazal2014persistence}
Fr{\'e}d{\'e}ric Chazal, Vin De~Silva, and Steve Oudot.
\newblock Persistence stability for geometric complexes.
\newblock {\em Geometriae Dedicata}, 173(1):193--214, 2014.

\bibitem{chowdhury2019geodesics}
Samir Chowdhury.
\newblock Geodesics in persistence diagram space.
\newblock {\em arXiv preprint arXiv:1905.10820}, 2019.

\bibitem{chowdhury2018explicit}
Samir Chowdhury and Facundo Mémoli.
\newblock Explicit geodesics in {G}romov-{H}ausdorff space.
\newblock {\em Electronic Research Announcements}, 25(0):48--59, 2018.

\bibitem{clause2020spatiotemporal}
Nate Clause and Woojin Kim.
\newblock Spatiotemporal persistent homology computation tool.
\newblock \url{https://github.com/ndag/PHoDMSs}, 2020.

\bibitem{curry2013sheaves}
Justin Curry.
\newblock {\em Sheaves, {C}osheaves and {A}pplications}.
\newblock PhD thesis, University of Pennsylvania, 2013.

\bibitem{curry2016classification}
Justin Curry and Amit Patel.
\newblock Classification of constructible cosheaves.
\newblock {\em Theory and Applications of Categories}, 35(27):1012--1047, 2020.

\bibitem{de2016categorified}
Vin De~Silva, Elizabeth Munch, and Amit Patel.
\newblock Categorified {R}eeb graphs.
\newblock {\em Discrete \& Computational Geometry}, 55(4):854--906, 2016.

\bibitem{deSilva2018}
Vin {de Silva}, Elizabeth Munch, and Anastasios Stefanou.
\newblock Theory of interleavings on categories with a flow.
\newblock {\em Theory and Applications of Categories}, 33(21):583--607, 2018.

\bibitem{edelsbrunner2008persistent}
Herbert Edelsbrunner and John Harer.
\newblock Persistent homology-a survey.
\newblock {\em Contemporary mathematics}, 453:257--282, 2008.

\bibitem{erickson2019algorithms}
Jeff Erickson.
\newblock {\em Algorithms}.
\newblock Independent Publish, Urbana-Champaign, IL, 2019.

\bibitem{erne1987compact}
Marcel Ern{\'e}.
\newblock Compact generation in partially ordered sets.
\newblock {\em Journal of the Australian Mathematical Society}, 42(1):69--83,
  1987.

\bibitem{erne2003posets}
Marcel Ern{\'e}, Branimir {\v{S}}e{\v{s}}elja, and Andreja
  Tepav{\v{c}}evi{\'c}.
\newblock Posets generated by irreducible elements.
\newblock {\em Order}, 20(1):79--89, 2003.

\bibitem{gasparovic2019intrinsic}
Ellen Gasparovic, Elizabeth Munch, Steve Oudot, Katharine Turner, Bei Wang, and
  Yusu Wang.
\newblock Intrinsic interleaving distance for merge trees.
\newblock {\em arXiv preprint arXiv:1908.00063}, 2019.

\bibitem{ghrist2008barcodes}
Robert Ghrist.
\newblock Barcodes: the persistent topology of data.
\newblock {\em Bulletin of the American Mathematical Society}, 45(1):61--75,
  2008.

\bibitem{griffiths1997}
R.C. Griffiths and P.~Marjoram.
\newblock An ancestral recombination graph.
\newblock In {\em Donnelly, P. and Tavaré, S. (Eds.), Progress in Population
  Genetics and Human Evolution, IMA Volumes in Mathematics and its
  Applications}, volume~87, pages 257--270. Springer Verlag, Berlin, Berlin,
  1997.

\bibitem{huson2010phylogenetic}
Daniel~H Huson, Regula Rupp, and Celine Scornavacca.
\newblock {\em Phylogenetic networks: concepts, algorithms and applications}.
\newblock Cambridge University Press, Cambridge, UK, 2010.

\bibitem{kerber2018exact}
Michael Kerber, Michael Lesnick, and Steve Oudot.
\newblock {Exact Computation of the Matching Distance on 2-Parameter
  Persistence Modules}.
\newblock In {\em 35th International Symposium on Computational Geometry (SoCG
  2019)}, volume 129 of {\em Leibniz International Proceedings in Informatics
  (LIPIcs)}, pages 46:1--46:15, Dagstuhl, Germany, 2019. Schloss
  Dagstuhl--Leibniz-Zentrum fuer Informatik.

\bibitem{kerber2017geometry}
Michael Kerber, Dmitriy Morozov, and Arnur Nigmetov.
\newblock Geometry helps to compare persistence diagrams.
\newblock {\em Journal of Experimental Algorithmics (JEA)}, 22:1--20, 2017.

\bibitem{kim2018formigrams}
Woojin Kim and Facundo M{\'e}moli.
\newblock Formigrams: Clustering summaries of dynamic data.
\newblock In {\em Proceedings of the thirtieth Canadian Conference on
  Computational Geometry}, pages 180--188, 2018.

\bibitem{kim2018generalized}
Woojin Kim and Facundo M{\'e}moli.
\newblock Generalized persistence diagrams for persistence modules over posets.
\newblock {\em Journal of Applied and Computational Topology}, 5(4):533--581,
  2021.

\bibitem{kim2020spatiotemporal}
Woojin Kim and Facundo M{\'e}moli.
\newblock Spatiotemporal persistent homology for dynamic metric spaces.
\newblock {\em Discrete \& Computational Geometry}, 66(3):831--875, 2021.

\bibitem{kim2017stable}
Woojin Kim and Facundo M{\'e}moli.
\newblock Extracting persistent clusters in dynamic data via {M\"o}bius
  inversion.
\newblock {\em arXiv preprint arXiv:1712.04064v5}, 2022.

\bibitem{kim2020analysis}
Woojin Kim, Facundo M{\'e}moli, and Zane Smith.
\newblock Analysis of dynamic graphs and dynamic metric spaces via zigzag
  persistence.
\newblock In {\em Topological Data Analysis}, pages 371--389. Springer, 2020.

\bibitem{kozlov2008combinatorial}
Dimitry Kozlov.
\newblock {\em Combinatorial algebraic topology}, volume~21.
\newblock Springer Science \& Business Media, 2008.

\bibitem{landi2018rank}
Claudia Landi.
\newblock The rank invariant stability via interleavings.
\newblock In {\em Research in computational topology}, pages 1--10. Springer,
  Switzerland, 2018.

\bibitem{lawler2001combinatorial}
Eugene~L Lawler.
\newblock {\em Combinatorial optimization: networks and matroids}.
\newblock Courier Corporation, Mineola, NY, 2001.

\bibitem{lesnick2015theory}
Michael Lesnick.
\newblock The theory of the interleaving distance on multidimensional
  persistence modules.
\newblock {\em Foundations of Computational Mathematics}, 15(3):613--650, 2015.

\bibitem{mac2013categories}
Saunders Mac~Lane.
\newblock {\em Categories for the working mathematician}, volume~5.
\newblock Springer Science \& Business Media, New York, 2013.

\bibitem{martin2013genome}
Simon~H Martin, Kanchon~K Dasmahapatra, Nicola~J Nadeau, Camilo Salazar,
  James~R Walters, Fraser Simpson, Mark Blaxter, Andrea Manica, James Mallet,
  and Chris~D Jiggins.
\newblock Genome-wide evidence for speciation with gene flow in heliconius
  butterflies.
\newblock {\em Genome research}, 23(11):1817--1828, 2013.

\bibitem{mccleary2020edit}
Alexander McCleary and Amit Patel.
\newblock Edit distance and persistence diagrams over lattices.
\newblock {\em SIAM Journal on Applied Algebra and Geometry}, 6(2):134--155,
  2022.

\bibitem{memoli2017distance}
Facundo M{\'e}moli.
\newblock A distance between filtered spaces via tripods.
\newblock {\em arXiv preprint arXiv:1704.03965}, 2017.

\bibitem{memoli2019quantitative}
Facundo M{\'e}moli and Osman~Berat Okutan.
\newblock Quantitative simplification of filtered simplicial complexes.
\newblock {\em Discrete \& Computational Geometry}, 65:554–--583, 2021.

\bibitem{morozov2013interleaving}
Dmitriy Morozov, Kenes Beketayev, and Gunther Weber.
\newblock Interleaving distance between merge trees.
\newblock In {\em Proceedings of Topology-Based Methods in Visualization},
  2013.

\bibitem{munch2019ell}
Elizabeth Munch and Anastasios Stefanou.
\newblock The $\ell^\infty$-cophenetic metric for phylogenetic trees as an
  interleaving distance.
\newblock In {\em Research in Data Science}, pages 109--127. Springer,
  Switzerland, 2019.

\bibitem{munkres2018elements}
James~R Munkres.
\newblock {\em Elements of algebraic topology}.
\newblock CRC press, Boca Raton, FL, USA, 2018.

\bibitem{o1973fast}
Patrick~E O'Neil and Elizabeth~J O'Neil.
\newblock A fast expected time algorithm for boolean matrix multiplication and
  transitive closure.
\newblock {\em Information and Control}, 22(2):132--138, 1973.

\bibitem{parida2015topological}
Laxmi Parida, Filippo Utro, Deniz Yorukoglu, Anna~Paola Carrieri, David Kuhn,
  and Saugata Basu.
\newblock Topological signatures for population admixture.
\newblock In {\em International Conference on Research in Computational
  Molecular Biology}, pages 261--275. Springer, 2015.

\bibitem{patel2018generalized}
Amit Patel.
\newblock Generalized persistence diagrams.
\newblock {\em Journal of Applied and Computational Topology}, 1(3-4):397--419,
  2018.

\bibitem{puuska2017erosion}
Ville Puuska.
\newblock Erosion distance for generalized persistence modules.
\newblock {\em Homology, Homotopy and Applications}, 22(1):233--254, 2020.

\bibitem{quillen1973higher}
Daniel Quillen.
\newblock Higher algebraic k-theory: I.
\newblock In {\em Higher K-theories}, pages 85--147. Springer, Seattle, WA,
  1973.

\bibitem{rolle2020stable}
Alexander Rolle and Luis Scoccola.
\newblock Stable and consistent density-based clustering.
\newblock {\em arXiv preprint arXiv:2005.09048}, 2020.

\bibitem{roman2008lattices}
Steven Roman.
\newblock {\em Lattices and ordered sets}.
\newblock Springer Science \& Business Media, New York, 2008.

\bibitem{schmiedl2017computational}
Felix Schmiedl.
\newblock Computational aspects of the gromov--hausdorff distance and its
  application in non-rigid shape matching.
\newblock {\em Discrete \& Computational Geometry}, 57(4):854--880, 2017.

\bibitem{scoccola2020locally}
Luis~N Scoccola.
\newblock {\em Locally Persistent Categories And Metric Properties Of
  Interleaving Distances}.
\newblock PhD thesis, The University of Western Ontario, 2020.

\bibitem{serra1998hausdorff}
Jean Serra.
\newblock Hausdorff distances and interpolations.
\newblock {\em Computational Imaging and Vision}, 12:107--114, 1998.

\bibitem{smith2016hierarchical}
Zane Smith, Samir Chowdhury, and Facundo M{\'e}moli.
\newblock Hierarchical representations of network data with optimal distortion
  bounds.
\newblock In {\em 2016 50th Asilomar Conference on Signals, Systems and
  Computers}, pages 1834--1838. IEEE, 2016.

\bibitem{sokal1962comparison}
Robert~R Sokal and F~James Rohlf.
\newblock The comparison of dendrograms by objective methods.
\newblock {\em Taxon}, pages 33--40, 1962.

\bibitem{Stefanou_2020}
Anastasios Stefanou.
\newblock Tree decomposition of {R}eeb graphs, parametrized complexity, and
  applications to phylogenetics.
\newblock {\em Journal of Applied and Computational Topology}, Feb 2020.

\bibitem{weibel2013k}
Charles~A Weibel.
\newblock {\em The K-book: An introduction to algebraic K-theory}, volume 145.
\newblock American Mathematical Society, Providence, RI, 2013.

\bibitem{woese1990towards}
Carl~R Woese, Otto Kandler, and Mark~L Wheelis.
\newblock Towards a natural system of organisms: proposal for the domains
  archaea, bacteria, and eucarya.
\newblock {\em Proceedings of the National Academy of Sciences},
  87(12):4576--4579, 1990.

\end{thebibliography}

\end{document}